\documentclass{amsart}[12pt, article]
\usepackage{amsmath}
\usepackage{mathrsfs}
\usepackage{amsfonts}
\usepackage{txfonts}
\usepackage{amssymb}
\usepackage{color}

\newcommand\NN{{\mathbb N}}
\newcommand\RR{{{\mathbb R}}}

\def\SS {\mathbb{S}}

\newcommand\cB{{\mathcal B}}
\newcommand\cD{{\mathcal D}}
\newcommand\cH{{\mathcal H}}

\newcommand\cL{{\mathcal L}}
\newcommand\cS{{\mathcal S}}
\newcommand\cT{{\mathcal T}}
\newcommand\cM{{\mathcal M}}

\def\vk{\mbox{\boldmath $k$}}

\newcommand\pP{{\bf P}}
\newcommand\iI{{\bf I}}
\newcommand\p{{\partial}}
\newcommand\pa{{\partial}}
\newcommand\ra{{\rangle}}
\newcommand\la{{\langle}}

\newtheorem{theo}{Theorem}[section]
\newtheorem{lemm}[theo]{Lemma}

\newtheorem{prop}[theo]{Proposition}
\newtheorem{rema}[theo]{Remark}

\renewcommand{\theequation}{\thesection.\arabic{equation}}

\begin{document}

\title[The Boltzmann equation without angular cutoff]
{Global existence and full regularity\\
of the Boltzmann equation without angular cutoff}
\author{R. Alexandre}
\address{R. Alexandre,
IRENAV Research Institute, French Naval Academy
Brest-Lanv\'eoc 29290, France}
\email{radjesvarane.alexandre@ecole-navale.fr}
\author{Y. Morimoto }
\address{Y. Morimoto, Graduate School of Human and Environmental Studies,
Kyoto University
\newline\indent
Kyoto, 606-8501, Japan} \email{morimoto@math.h.kyoto-u.ac.jp}
\author{S. Ukai}
\address{S. Ukai, 17-26 Iwasaki-cho, Hodogaya-ku, Yokohama 240-0015, Japan}
\email{ukai@kurims.kyoto-u.ac.jp}
\author{C.-J. Xu}
\address{C.-J. Xu, School of Mathematics, Wuhan University 430072,
Wuhan, P. R. China
\newline\indent
and \newline\indent
Universit\'e de Rouen, UMR 6085-CNRS,
Math\'ematiques
\newline\indent
Avenue de l'Universit\'e,\,\, BP.12, 76801 Saint
Etienne du Rouvray, France } \email{Chao-Jiang.Xu@univ-rouen.fr}
\author{T. Yang}
\address{T. Yang, Department of mathematics, City University of Hong Kong,
Hong Kong, P. R. China} \email{matyang@cityu.edu.hk}


%\date{24 December 2009}


\begin{abstract}

We prove the global existence and uniqueness of classical solutions around an equilibrium to the Boltzmann
equation without angular cutoff in some Sobolev spaces. In addition, the solutions thus
obtained are shown to be non-negative and $C^\infty$ in all variables for any positive time. In this paper, we study  the Maxwellian molecule type collision operator with
mild singularity. One of the
key observations is the introduction of a new important norm related to
the singular behavior of the cross section in the collision operator. This norm captures the essential properties of the singularity and
yields precisely the dissipation %description
of the linearized collision operator
through the celebrated H-theorem.
\end{abstract}

\maketitle

\tableofcontents

\section{Introduction}\label{s1}

We consider the Cauchy problem for the inhomogeneous Boltzmann equation% in dimension $3$
\begin{equation}\label{1.1}
f_t+v\cdot\nabla_x f=Q(f, f)\,,\,\,\,\,\,\,f|_{t=0}=f_0,
\end{equation}
where $f= f(t,x,v)$ is the density distribution function
of particles, having position
$x\in \RR^3$ and velocity $v\in \RR^3$ at time $t$.
Here, the right hand side of (\ref{1.1}) is given by the
Boltzmann bilinear collision operator, which is given in the classical $\sigma -$representation by
\[
Q(g, f)=\int_{\RR^3}\int_{\mathbb S^{2}}B\left({v-v_*},\sigma
\right)
 \left\{g'_* f'-g_*f\right\}d\sigma dv_*\,,
\]
where $f'_*=f(t,x,v'_*), f'=f(t,x,v'), f_*=f(t,x,v_*), f=f(t,x,v)$, and for
$\sigma\in \mathbb S^{2}$,
$$
v'=\frac{v+v_*}{2}+\frac{|v-v_*|}{2}\sigma,\,\,\, v'_*
=\frac{v+v_*}{2}-\frac{|v-v_*|}{2}\sigma,\,
$$
which gives the relation between the post and pre collisional
velocities. Recall that  we have conservation of momentum and kinetic energy, that is, $v+v_\ast = v' +v'_\ast$ and $| v|^2+|v_\ast|^2 = | v'|^2 + | v'_\ast|^2$.
The kernel $B$ is the cross-section which can be computed in different physical
settings.

In particular,  the non-negative cross section
 $B(z, \sigma)$ depends only on $|z|$ and the scalar product
$\langle\frac{z}{|z|}, \sigma\rangle$. In most cases, the
kernel $B$ cannot be expressed explicitly, but
to capture its main properties, %and to avoid the unnecessary technical computations
one may assume
that it takes  the form
\begin{equation*}
B(|v-v_*|, \cos \theta)=\Phi (|v-v_*|) b(\cos \theta),\,\,\,\,\,
\cos \theta=\Big\langle\frac{v-v_*}{|v-v_*|} \, ,\,\sigma\Big\rangle,\,\,\,
0\leq\theta\leq\frac{\pi}{2}.
\end{equation*}
An important example is %, let us recall that for potential of
the inverse power law potential $\rho^{-r}$ with $r>1$, $\rho$
being the distance between two  particles, in which the
cross section %in the above form
has a  kinetic factor %related to the
given by
\begin{equation*}
\Phi(|v-v_*|)\approx |v-v_*|^{\gamma},\qquad \gamma=1-\frac{4}{r},
\end{equation*}
and a factor related to the collision angle containing a singularity,
\begin{equation*}
\begin{array}{l}
 b(\cos \theta)\approx K\theta^{-2-2s} \
\
 \mbox{when} \ \ \theta\rightarrow 0+,
\end{array}
\end{equation*}
for some constants $K>0$ and $0<s=\frac{1}{r}<1$.

The cases with
$1<r<4$, $r=4$ and $r>4$ correspond to so-called soft, Maxwellian molecule
and hard
 potentials respectively. In the following discussion, this type of cross sections, with the parameters $\gamma$ and
$s$ given above, will be kept in mind.

As a fundamental equation in kinetic theory and
a key stone in statistical physics, the Boltzmann equation
has attracted, and is still attracting, a lot of research investigations since its derivation in 1872.

A large number of mathematical works have been performed under the Grad's cutoff
assumption, avoiding the non-integrable angular singularity
of the cross-sections, see for example \cite{Ce88,Cercignani-Illner-Pulvirenti,D-L,Grad,grad1,
guo-1,Lu,K-S,Lions94,ukai-2} to cite only a few, further references being given in the review \cite{villani2}.

However, except for the hard sphere model, for most of the other molecule
interaction potentials such as the inverse
power laws recalled above,
 the cross section  $B(v-v_*,\sigma)$ is a non-integral function in
angular variable
and the collision operator $Q(f,f)$ is a nonlinear singular integral operator
in %with
%respect to  %the angular
%microscopic
velocity
variable.% $v$.

By no means to be complete, let us now review some previous works related to the Boltzmann equation in the context of such singular (or non-cutoff) cross-sections. For other
 references and comments,  readers are referred to \cite{alex-review,villani2}.

The mathematical study for the Boltzmann equation, without assuming Grad's cutoff assumption, can be traced back at least to the work by Pao in
1970s \cite{Pao} which is about the spectrum of the linearized operator.
%, and then many %others afterwards. %After the work
In 1980s, the existence of weak solutions to the spatially
homogeneous case was  proved by Arkeryd in \cite{Arkeryd} for the mild
singular case, that is, when $0<s<\frac 12$, and by using an abstract
Cauchy-Kovalevskaya theorem, Ukai in \cite{ukai} proved the local
existence of solutions to the inhomogeneous equation, in the space of functions which are analytic
in $x$ and Gevrey in $v$.

For a long time, the mathematical study of  singular cross-sections was limited to these results and a few others, most of them related to the spatially  homogeneous case
concerning only the existence. An important step was initiated by the
 works of Desvillettes and his collaborators in 1990s, showing partial regularization results for some
simplified kinetic models, cf. \cite{D95,desv-ajout1,
desv-ajout2,desv-F,desv-golse,desv-vil,Vi99}.

After the well known result of DiPerna and Lions \cite{D-L} for the cutoff case, Lions was able to show the gain of regularity of solutions in the Landau case \cite{lions-landau},
which is a model arising as a grazing limit of the Boltzmann equation. It was then expected that this kind of singular cross sections should lead to smoothing
effect on solutions, that is, the solutions have higher regularity
than the initial data. For example, it should be similar to
the case when one replaces the
collision operator in the Boltzmann equation
by a fractional Laplacian in velocity variable, that is, a fractional Kolmogorov-type equation
\cite{M-X}.

Certainly, the results of Lions \cite{Lions98} and Desvillettes  have influenced
the  research in this direction. It is therefore not surprising that a systematic approach, using the entropy
dissipation and/or the smoothing property of the gain part of the
collision operator, was initiated and has
been developed to an almost optimal stage through the efforts of many researchers, such as Alexandre, Bouchut, Desvillettes,
Golse, Lions,
Villani and Wennberg.
The underlying tools have proven to be very useful for the
study on the mathematical theory
 regarding the regularizing effect
for the spatially homogeneous problems for which the theory can now be considered as quite
satisfactory, cf.
\cite{al-1,al-saf,al-3,chen-li-xu1,
desv-F,Des-Fu-Ter,desv-wen1,HMUY,MU,MUXY-DCDS,villani},
and the references therein, see also for a much more detailed discussion \cite{alex-review}.

Compared to the spatially homogeneous problems, the original
spatially inhomogeneous Boltzmann equation is of course physically
more interesting and mathematically more challenging. For existence
of weak solutions, we mention two
results regarding the Cauchy problem. One is about
the local existence between two moving Maxwellians  proved in
\cite{alex-two-maxwellian} by constructing the upper and lower
solutions, another is   the global existence of renormalized solutions
with defect measures shown in \cite{al-3} where
the solutions become weak solutions if the defect measures vanish.
On the other hand, the local existence
of classical solutions was proved in \cite{amuxy-nonlinear-3} in some
weighted Sobolev spaces.

However, in view of the above available results, the mathematical
theory for non-cutoff
cross-sections is so far not satisfactory. This is in sharp contrast to the cutoff case, for which the theories have been well
developed, see \cite{BD,bouchut-1,desv,D-L,dly,guo-1,liu-1,liu-2,ukai-1a,ukai-1b, ukai-2} and the references therein.

For the study of the regularizing effect,
one of the main difficulties  comes {}from
the coupling of the
transport operator with the collision operator,
which is similar to the Landau equation studied in \cite{desv-landau}. To overcome this difficulty, a
generalized uncertainty principle \`a la Fefferman
\cite{Fefferman} (see also \cite{morimo873, morimo921,
morimoto-mori}) was introduced in \cite{amuxy-nonlinear,amuxy-nonlinear-b} for the study
of smoothing effects of the linearized and spatially inhomogeneous
Boltzmann equation  with non-cutoff cross sections.

In order to complete the full regularization process, recently, in \cite{amuxy-nonlinear-3}, by using suitable pseudo-differential operators and harmonic analysis, we have developed sharp coercivity and upper bounds of the collision
operators in Sobolev space, together with the estimation on the commutators
with these pseudo-differential operators.
 More precisely, in \cite{amuxy2,amuxy2b,amuxy-nonlinear-3},
for classical solutions, we established the hypo-ellipticity of the
Boltzmann operator, using the generalized version of the uncertainty
principle.

The present work is a continuation of our collaborative program since 2006 %and in particular of our series of works
\cite{amuxy-nonlinear-b,amuxy2, amuxy2b, amuxy-nonlinear-3}. Comparing to the cutoff case, we aim to settle a mathematical framework similar to the studies first proved by Ukai, see \cite{ukai-1a, ukai-1b}, and fitted into an energy method by Liu and collaborators \cite{liu-1,liu-2}
and Guo \cite{guo-1} which has led to a clean theory for the Cauchy problem in the cutoff case, for solutions close to a global equilibrium.

In this paper, we will
establish the global existence of  non-negative solutions in some Sobolev space
for the Boltzmann equation near a global equilibrium
and prove the full
regularity  in all variables for any positive time.% for the solutions

As mentioned
in the abstract, one of the main ingredients in the proof is the
introduction of a new non-isotropic norm which captures the main feature
of the singularity in the cross-section. This new norm is in fact the counterpart of the coercive norm which was introduced by Guo \cite{guo} as an essential step for Landau equation. 

It is not known if there is any equivalence of this norm to some Sobolev norm, in contrast 
to  the case of the Landau equation. 
However, since it is designed to be 
equivalent to, and to have much simpler expression than,  
 the Dirichlet form of the linearized collision operator, %this norm not only works $\cdots$"
% As will be clear soon, 
% {\red though, in contrast to  the case of the Landau equation,
% any equivalence to some Sobolev norm is not available  but
% it is  equivalent to, and has much simpler expression than,  
%  the Dirichlet form of the linearized collision operator. 
% }
% however, 
% we cannot have the explicit equivalence of this norm to some
% Sobolev norm, which is in contrast to  the case of the Landau equation: There is a  gap 
% between the upper and lower
% estimates of this norm in terms of the weighted Sobolev norms 
% cannot
% have
% by the standard norms by
% }
% it is important to note that we do not
% need the explicit equivalence of this norm to some
% Sobolev norm, which is in contrast to  the case of the Landau equation.
%{\red However,} %
%In fact, 
this norm not only works
extremely well for the description of the dissipative effect of the linearized
collision operator through the H-theorem, but also well fits for the
upper bound estimation on the nonlinear collision operator. Here,
we would like to mention the work by Mouhot and Strain \cite{mouhot,mouhot-strain} about the gain of moment
in a linearized context due to the singularity in the cross-section. Such
a gain of moment which is well described by the new non-isotropic norm
is in fact  crucial for the proof of global
existence.

We now come back to the problem
considered in this paper.
To make the presentation as simple as possible, and  to concentrate  on the
singularity of the grazing effect, we shall study the Maxwellian molecule type cross-sections with mild singularity, that is, the case when
\begin{equation*}\label{singularity-1}
B(|v-v_*|, \cos \theta)=b(\cos \theta),\,\,\,\,\, \cos
\theta=\Big\langle\frac{v-v_*}{|v-v_*|} \, ,\,\sigma\Big\rangle,\,\,\,
0\leq\theta\leq\frac{\pi}{2},
\end{equation*}
and
\begin{equation}\label{singularity}
b(\cos \theta)\approx  K \theta^{-2-2s},\,\,\,\,\,
\theta\,\rightarrow\, 0^+ ,
\end{equation}
with $0<s<\frac12$. The general case will be left to our future work.

In order to prove the global existence, we need to use
the complete dissipative effect of the collision operator. Similar to angular cutoff case, such dissipative effect
can be fully represented by the dissipation of the linearized
collision operator on the microscopic component of the solution
through the  H-theorem.

Thus, as usual, we consider the Boltzmann equation around a
normalized Maxwellian distribution
\[
\mu (v)=(2\pi )^{-\frac{3}{2}}e^{-\frac{|v|^2}{2\,\,}}.
\]
Since $\mu$ is the global equilibrium state satisfying $Q(\mu,
\mu)=0$, by setting $f=\mu+\sqrt{\mu}  g$, we have
\[
Q(\mu+\sqrt{\mu}  g,\, \mu+\sqrt{\mu}\,  g)=Q(\mu,\, \sqrt{\mu}\,g)+
Q(\sqrt{\mu}\,g,\, \mu)+Q(\sqrt{\mu}\,g,\, \sqrt{\mu}\,g).
\]
Denote
\begin{equation*}\label{1.2}
\Gamma(g,\, h)=\mu^{-1/2}Q(\sqrt{\mu}\,g,\, \sqrt{\mu}\,h).
\end{equation*}
Then the linearized Boltzmann operator takes the form
\begin{equation*}\label{1.3}
\cL g=\cL_1\, g+\cL_2\, g=-\Gamma(\sqrt{\mu}\,,\, g)- \Gamma(g,\,
\sqrt{\mu}\,).
\end{equation*}
And the original problem (\ref{1.1}) is now reduced to the Cauchy problem
for the perturbation
$g$
\begin{equation}\label{cauchy-problem}
\left\{\begin{array}{l} g_t + v\cdot\nabla_x g + \cL g=\Gamma(g,\,
g), \,\,\, t>0\,;\\
g|_{t=0}=g_0.\end{array}\right.
\end{equation}

This problem will be considered in the following
weighted Sobolev spaces.
For
$k,\,\,\ell\in\RR$, set
$$
H^k_\ell(\RR^6_{x, v})=\left\{f\in\cS'(\RR^6_{x, v})\, ;\,\, W^\ell
f\in H^k(\RR^6_{x, v})\,\right\}\, ,
$$
where $\RR^6_{x, v}=\RR^3_{x}\times \RR^3_{v}$ and $ W^\ell
(v)=\langle v\rangle^\ell=(1+|v|^2)^{\ell/2}$ is the weight with respect to the
velocity variable $v\in\RR^3_v$.

The main theorem can be
stated as follows.

\begin{theo}\label{theo1}
Assume that the cross-section satisfies (\ref{singularity}) with
$0<s<1/2$. Let $ g_0 \in H^{k}_\ell(\RR^6)$ for some $
k\geq 3, \ell\geq 3$ and
$$
f_0(x, v)=\mu+\sqrt{\mu} \,g_0(x, v)\geq 0.
$$
Then there exists $\varepsilon_0 >0$, such that if ~
$\|g_0\|_{H^{k}_\ell(\RR^6)}\leq \varepsilon_0$,  the Cauchy problem
(\ref{cauchy-problem}) admits a unique global solution $$g\in
L^\infty([0, +\infty[\,;\,\,H^{k}_\ell(\RR^6)).$$ Moreover,
$f(t, x,
v)=\mu+\sqrt{\mu}\,\, g(t, x, v)\geq 0$ and
$$
g\in C^\infty(]0, +\infty[\,\times\RR^6).
$$
\end{theo}

Actually, for the uniqueness, we can prove the following stronger result, which might be  of independent interest. Note that here we do not need to assume that $f$ is a small
perturbation of $\mu$.

\begin{theo}\label{theo2}
Under the same condition on the cross-section, for $ 0<T\leq +\infty$ and $
 l>2s+7/2$, let $f_0\geq 0, \, f_0\in
L^\infty(\RR_x^3; H^{2s}_{l+2}(\RR^3_v))$. Suppose that $f_1,\, f_2\in
 L^\infty(]0, T[ \times \RR_x^3;\,\, H^{2s}_{l+2}(\RR_v^3))$
are two solutions to the Cauchy problem
(\ref{1.1}).
If one solution is non-negative, then  $f_1\,\equiv\, f_2$.
\end{theo}
%}

Throughout this paper, we assume that the cross-section satisfies
the condition (1.6) with $0<s<1/2$ except otherwise stated.

The rest of the paper will be organized as follows. In the next section,
we will introduce a new non-isotropic norm and prove some essential
coercivity and upper bound estimates on the collision operators
with respect to this new norm.
In order to study the gain of regularity of the solution, we need to apply
some  pseudo-differential operators
on the Boltzmann equation. For this purpose, in Section 3,
 we  study
the commutators of the collision operators with the
pseudo-differential operators.
 In Section 4, we will apply
the energy method for the Boltzmann equation and obtain
the local existence theorem. In Section 5, we will study the
uniqueness and the non-negativity of the solutions. 
This new method for proving non-negativity 
can be applied to the case with angular cutoff. For more detail
discussion on the non-negativity problem,  refer to \cite{amuxy4-4}.
In Section 6,
the full regularity is proved along the approach of
\cite{amuxy-nonlinear-3}. Finally,  the global existence of the solution
will be given in the last section. For this, the macro-micro decomposition introduced by Guo \cite{guo}
will be used for the estimation on the macroscopic component.\\

{\bf Note:} After finishing this paper,  we were informed by R. Strain of his recent paper in collaboration with P. Gressmann
\cite{gr-st}, showing also the existence of global solutions to the Cauchy problem by using different approach. Notice that their solution is in different
function space which does not lead to full regularity because of the weak regularity in
the velocity variable.

{\bf Note added in September, 2010:} Several new results have been announced along the same
line of development since the submission of the current paper. For the reader's references
we mention \cite{gr-st2,gr-st3,gr-st4, amuxy4-1,  amuxy4-3, amuxy4-4}.
The main difference of the results is the range of admissible values of $\gamma$:
$\gamma>-1-2s$ in the first 3 papers and $\gamma>\max(-3, -3/2-2s)$ in the latter 4 paper.

\renewcommand{\theequation}{\thesubsection.\arabic{equation}}
%%%%%%%%%%%%%%%%%%%%%%%%%%%%%%%%%%%%%%%%%%%%%%%%%%%%%%%%%%%%%%%%%%%%%%
%%%%%%%%%%%%%%%%%%%%%%%%%%%%%%%%%%%%%%%%%%%%%%%%%%%%%%%%%%%%%%%%%%%%%%
\section{Non-isotropic norms}\label{section2}
\smallskip

In this section, we study the bilinear collision operator given by
\[
Q(g, f)=\int_{\RR^3}\int_{\mathbb S^{2}}b(\cos\theta)
 \left\{g'_* f'-g_*f\right\}d\sigma dv_*\,,
\]
through harmonic analysis.
Since the collision operator acts only with respect to the velocity variable
$v\in\RR^3$, $(t, x)$ is regarded as a parameter in this section.

%%%%%%%%%%%%%%%%%%%%%%%%%%%%%%%%%%%%%%%%%%%%%%%%%%%%%%%%%%%%%%%%%%%%%%%
\subsection{Coercivity and upper bound estimates }\label{section2.1}
\setcounter{equation}{0}
%\smallskip

Let $g\geq 0, \, g \equiv \hspace{-3mm} / \: 0, \, g\in L^1_2\bigcap
L\log L (\RR^3_v)$. It was shown in \cite{al-1} that there exists a constant $c_g>0$ depending
only on the values of $\|g\|_{L^1_2}$ and $\,\|g\|_{L\log L}$ such
that for any smooth function $f\in H^{s}(\RR^3_v)$, we have
\begin{equation}\label{2.1.1}
c_{g}\| f\|^2_{H^s(\RR^3_v)} \leq (-Q(g, f),
 f)_{L^2(\RR^3_v)}+C \| g\|_{L^1(\RR^3_v)}\|f\|^2_{L^2(\RR^3_v)}.
\end{equation}

Besides this, we still need some functional estimates on the Boltzmann collision
operators. The first one, given below, is
 about the boundedness of the collision operator in
weighted Sobolev spaces, see  \cite{alex-lin,alex-2,alex-sing,alex-review,amuxy-nonlinear-3,HMUY} .

\begin{theo}\label{theo2.1}
Assume that the cross-section satisfies (\ref{singularity}) with
$0<s<1$ . Then for any $m\in\RR$ and any $\alpha \in \RR$, there
exists $C>0$ such that
\begin{equation}\label{2.1.2}
\|Q(f,\, g)\|_{H^{m}_{\alpha}(\RR^3_v)}\leq C\|f\|_{L^{1}_{{\alpha^+
+}2s}(\RR^3_v)} \| g\|_{H^{m+2s}_{({\alpha +} 2s)^+}(\RR^3_v)}\,
\end{equation}
for all $f\in L^{1}_{{\alpha^+ +}2s}(\RR^3_v)$ and $ g\in
H^{m+2s}_{({\alpha +} 2s)^+}(\RR^3_v)$ .
\end{theo}

We now turn to the linearized operator. First of all, by using the conservation of energy
$$
|v'_*|^2+|v'|^2=|v_*|^2+|v|^2,
$$
we have $\mu(v_*)=\mu^{-1}(v)\,\mu(v'_*)\,\mu(v')$. Thus,
\begin{eqnarray}\label{Gamma}
&\Gamma(f,\, g)(v)=\mu^{-1/2}\iint b(\cos\theta) \Big(\sqrt{\mu'_*}
f'_* \sqrt{\mu'} g'- \sqrt{\mu_*} f_*
\sqrt{\mu}\,\, g\,\Big) d v_* d \sigma \nonumber\\
&=\iint b(\cos\theta)\sqrt{\mu_*}\, \Big( f'_* g'- f_* g\Big) d v_*
d \sigma.
\end{eqnarray}
It is well-known that $\cL$ (acting with respect to the velocity variable) is an
unbounded symmetric operator on $L^2(\RR^3_v)$. Moreover,  its
Dirichlet form satisfies
\begin{align}\label{linear-op}
&\Big(\cL g,\, g\Big)_{L^2(\RR^3_v)} =-\Big(\Gamma(\sqrt\mu\,,
g)+\Gamma(g,\,\sqrt\mu\,),\, g\Big)_{L^2(\RR^3_v)}\nonumber\\
&= \iiint b(\cos\theta) \Big((\mu_*)^{1/2} g- (\mu'_\ast)^{1/2}
g'+g_*(\mu)^{1/2}-
g'_*(\mu')^{1/2} \Big)(\mu_*)^{1/2}\, g  dv_*d\sigma dv\nonumber
\\
&= \iiint b(\cos\theta) \Big((\mu'_*)^{1/2} g'- (\mu_\ast)^{1/2}
g+g'_*(\mu')^{1/2}- g_*(\mu)^{1/2} \Big)(\mu'_*)^{1/2}\, g'
dv_*d\sigma dv\nonumber
\\
&= \iiint b(\cos\theta) \Big((\mu)^{1/2} g_*- (\mu')^{1/2}
g'_*+g(\mu_*)^{1/2}- g'(\mu'_*)^{1/2} \Big)(\mu)^{1/2}\, g_*
dv_*d\sigma dv
\\
&= \iiint b(\cos\theta) \Big((\mu')^{1/2} g'_*- (\mu)^{1/2}
g_*+g'(\mu'_*)^{1/2}- g(\mu_*)^{1/2} \Big)(\mu')^{1/2}\, g'_*
dv_*d\sigma dv\nonumber\\
%\\
%&= \frac{1}{4} \iiint b(\cos\theta)\left( \Big((\mu_\ast)^{1/2} g
%-(\mu'_*)^{1/2} g'\Big)+ \Big((\mu)^{1/2} g_* -(\mu')^{1/2}
%g'_*\Big)\right)^2 dv_*d\sigma dv\\
&=\frac{1}{4} \iiint b(\cos\theta)\left( \Big((\mu_\ast)^{1/2} g
-(\mu'_*)^{1/2} g'\Big)+ \Big((\mu)^{1/2} g_* -(\mu')^{1/2}
g'_*\Big)\right)^2 dv_*d\sigma dv\nonumber\\&\ge 0\nonumber.
\end{align}
 The third line in the above equation
  is obtained  by using the change of variables  $(v,v_\ast ) \rightarrow (v', v'_\ast )$. The fourth line follows {}from  the change of variables $(v,v_\ast ) \rightarrow (v_\ast ,v)$ and then the fifth line follows {}from the fourth one by using the change of variables $(v,v_\ast ) \rightarrow (v',v'_\ast )$.
  And the second last line is just the summation of the previous
  four lines. Note that the Jacobians of the above coordinate
  transformations are equal to $1$.

Moreover, it follows {}from the above formula that $\Big(\cL g,\, g\Big)_{L^2(\RR^3_v)}=0$ if and only if $\pP
g=g$ where
$$
\pP g=\Big(a+b\,\cdot\, v+c|v|^2\Big)\sqrt \mu,
$$
with $a, c\in\RR, b\in\RR^3$. Here, $\pP$ is the $L^2$-orthogonal
projection onto the null space
\begin{equation*}\label{2.1.3}
\mathcal{N}=\mbox{Span}\left\{\sqrt \mu\,,\,
v_1\sqrt\mu\,,\,v_2\sqrt\mu\,,\,v_3\sqrt\mu\,,\,
|v|^2\sqrt\mu\,\,\right\}.
\end{equation*}

The following result on the gain of
 moment of order $s$ in the linearized framework is essential in the sequent analysis.
\begin{theo}\label{theo2.2}(Theorem 1.1 of \cite{mouhot-strain})

Assume that the cross-section satisfies (\ref{singularity}) with
$0<s<1$. Then there exists a constant $C>0$ such that
\begin{equation*}%\label{2.1.4}
\Big(\cL g,\, g\Big)_{L^2(\RR^3_v)}\geq C \left\|(\iI-\pP)
g\right\|^2_{L^2_s(\RR^3_v)}\, .
\end{equation*}
\end{theo}

\smallbreak
For the bilinear operator $\Gamma(\,\cdot,\,\cdot\,)$, we  need
the following  two formulas.
For suitable functions $f, g$, the first formula coming {}from
(\ref{Gamma}) is
\begin{align}\label{1formula}
&\Gamma(f,\, g)(v)=Q(\sqrt{\mu}\, f,\, g)+ \iint b(\cos\theta) \Big(\sqrt{\mu_*}\,
-\sqrt{\mu'_*}\,\,\Big)f'_* g' d v_* d \sigma\,.
\end{align}

On the other hand, applying the change of variables $(v, v_*) \rightarrow (v', v'_*)$ in (\ref{Gamma}) gives
\begin{align*}
\Big(\Gamma(f,\, g),\,  h\Big)_{L^2(\RR^3_v)} &= \iiint
b(\cos\theta)
\sqrt{\mu_\ast}\, \big( f'_\ast g' - f_\ast g\big) h\\
&= \iiint b(\cos\theta) \sqrt{\mu'_\ast}\, \big(f_\ast g - f'_\ast
g'\big) h'\,.
\end{align*}
By adding these two lines, the second formula is
\begin{equation}\label{2.2.5}
\Big(\Gamma (f,\, g),\, h\Big)_{L^2(\RR^3_v)} = \frac{1}{2} \iiint
b(\cos\theta) \Big(f'_\ast g' - f_\ast g \Big) \Big(
\sqrt{\mu_\ast}\, h - \sqrt{\mu'_\ast}\, h' \Big)\,.
\end{equation}

The following lemma shows that $\cL_1$ controls  $\cL$.
\begin{lemm}\label{lemm2.1.1}
Under the condition (1.2) on the cross-section with
$0<s<1$, we have
\begin{equation}\label{2.1.5}
\Big(\cL_1 g,\, g\Big)_{L^2(\RR^3_v)}\geq \frac{1}{2} \Big(\cL g,\,
g\Big)_{L^2(\RR^3_v)}\, .
\end{equation}
\end{lemm}
\begin{proof} {}From (\ref{Gamma}) and  similar
changes of variables, we have
\begin{align*}
&\Big(\cL_1 g,\, g\Big)_{L^2(\RR^3_v)} =-\Big(\Gamma(\sqrt\mu\,,
g),\, g\Big)_{L^2(\RR^3_v)}\\
&= \frac{1}{2} \iiint b(\cos\theta) \Big((\mu'_*)^{1/2} g'-
(\mu_\ast)^{1/2} g \Big)^2dv_*d\sigma dv\\
&=\frac{1}{2} \iiint b(\cos\theta) \Big((\mu')^{1/2} g'_*-
(\mu)^{1/2} g_* \Big)^2dv_*d\sigma dv\\
&= \frac{1}{4} \iiint b(\cos\theta)\left\{ \Big((\mu'_*)^{1/2} g'-
(\mu_\ast)^{1/2} g \Big)^2+ \Big((\mu')^{1/2} g'_*- (\mu)^{1/2} g_*
\Big)^2\right\} dv_*d\sigma dv\, .
\end{align*}

Therefore, (\ref{2.1.5}) follows {}from $ (A+B)^2 \leq 2 (A^2+B^2)$
and (\ref{linear-op}).
\end{proof}
%\bigbreak

%%%%%%%%%%%%%%%%%%%%%%%%%%%%%%%%%%%%%%%%%%%%%%%%%%%%%%%%%%%%%%%%%%%%%%%
\subsection{Definition and properties of the non-isotropic norm}\label{section2.2}
\setcounter{equation}{0}
\smallskip
The non-isotropic norm associated with the cross-section
$b(\cos\theta)$ is defined by
\begin{align}\label{2.2.1}
||| g|||^2=& \iiint b(\cos\theta) \mu_*\, \big(g'-g\,\big)^2\, +
\iiint b(\cos\theta) g^2_* \big(\sqrt{\mu'}\,\, - \sqrt{\mu}\,\,
\big)^2\,,
\end{align}
where the integration  is over
~~$\RR^3_v\times\RR^3_{v_\ast}\times\SS^2_\sigma$. Thus, it is a
norm with respect to the velocity variable $v\in\RR^3$ only.
As we will see later, the reason that this
 norm is  called non-isotropic is because  it combines both derivative and   weight of order $s$ due to the singularity
of cross-section $b(\cos\theta)$.

The following lemma gives
an upper bound of this non-isotropic norm
by some  weighted Sobolev norm.

\begin{lemm}\label{lemm2.2.1}
Assume that the cross-section
satisfies (\ref{singularity}) with  $0<s<1$. Then there exists $C>0$
such that
\begin{equation}\label{2.2.2}
||| g |||^2 \leq C || g||^2_{H^s_s}
\end{equation}
for any $g\in H^s_s(\RR^3_v)$.
\end{lemm}

\begin{proof} Applying (\ref{2.1.2}) with $\alpha=-s$ and $m=-s$
gives
\begin{equation}\label{2.2.3}
\left|\Big(Q (f^2,\, g),\, g \Big)_{L^2(\RR^3_v)} \right| \leq C
||f^2||_{L^1_{2s}} || g ||_{H^{s}_{s}} || g ||_{H^{s}_{s}} \leq C
||f||^2_{L^2_{s}} || g ||^2_{H^{s}_{s}}.
\end{equation}
On the other hand,
\begin{align*}
&\Big(Q(f^2,\,g),\,g\Big)_{L^2(\RR^3_v)} = \iiint b(\cos\theta)
\Big(f^{2'}_* g' -
f^2_* g\Big ) g\\
&= \iiint b(\cos\theta)f^{2'}_* \Big(g'-g\Big) g + \int g^2\iint
b(\cos\theta) \Big(f^{2'}_* -f^2_*\Big ).
\end{align*}
For the first term in the last equation, using
$b(a-b) = \frac{1}{2} (a^2 -b^2) - \frac{1}{2}
(a-b)^2$ yields
\begin{align*}
&\Big(Q(f^2,\, g),\, g\Big)_{L^2(\RR^3_v)} = \frac{1}{2}\iiint
b(\cos\theta)
f^{2'}_*\Big(g^{2'} - g^2\Big)\\
& - \frac{1}{2}\iiint b(\cos\theta) f^{2'}_* \Big(g'- g\Big)^2 +
\int g^2 \iint b(\cos\theta) \Big(f^{2'}_* -f^2_* \Big).
\end{align*}
By the change of variables $(v'_\ast, v') \rightarrow (v_\ast, v)$,
the first term above is also $\frac 1 2 \iiint b g^2 (f^2_\ast -f^{2
'}_\ast )$. Thus, it follows that
\begin{equation*}\label{2.2.3+0}
\Big(Q (f^2,\, g) ,\, g\Big)_{L^2(\RR^3_v)}
 = -\frac 1 2 \iiint b f^{2 '}_\ast \big(g'-g\big)^2 + \frac 1 2 \int g^2
 \iint b \Big(f^{2 '}_\ast -f^2_\ast \Big),
\end{equation*}
and then
$$
\iiint b f^2_\ast \big(g'-g\big)^2 \leq 2 \big| \Big(Q (f^2,\, g),\,
g\Big) \big| + \big|\iiint b g^2 \Big(f^{2 '}_\ast -f^2_\ast \Big) \big|.
$$
By using (\ref{2.2.3}) and the cancellation lemma {}from
\cite{al-1}, we get
\begin{equation}\label{2.2.3+}
\iiint b f^2_\ast \big(g'-g\big)^2\leq C ||f ||^2_{L^2_{s}} ||
g||^2_{H^s_s} + C|| g||^2_{L^2} || f ||^2_{L^2}\leq C || f
||^2_{L^2_s} || g ||^2_{H^s_s}.
\end{equation}
Thus, choosing $f=\sqrt{\mu}$ gives
$$
||| g |||^2 \leq C ( \|\sqrt{\mu}\,\|^2_{L^2_s}|| g ||^2_{H^s_s} +
|| g||^2_{L^2_s}\|\sqrt{\mu}\,\|^2_{H^s_s} ) \leq C || g
||^2_{H^s_s}.
$$
This completes the proof of the lemma.
\end{proof}
In the context of usual weighted Sobolev spaces, this last result is likely to be optimal.
 Next we will show that this non-isotropic norm is controlled by the linearized operator.  First of all, we shall need the
 following preliminary computation.

\begin{lemm}\label{lemm2.2.2}
For any $\phi\in C^1_b$, we have
\begin{equation*}\label{2.2.6}
\int_\sigma b(\cos\theta) | \phi(v_\ast) - \phi(v'_\ast) |d\sigma
\leq C_\phi | v-v_\ast|^{2s}\leq C \langle v\rangle^{2s} \langle
v_*\rangle^{2s},
\end{equation*}
where $C_\phi$ depends on
$\|\phi\|_{C^1_b}=\|\phi\|_{L^\infty}+\|\bigtriangledown
\phi\|_{L^\infty}$.
\end{lemm}
\begin{proof}  It follows {}from Taylor's formula that
\begin{equation*}\label{2.2.2+0}
|\phi(v_\ast) - \phi (v'_\ast) | \leq C_\phi | v_\ast-v'_\ast | \leq
C_\phi \sin \left(\frac{\theta}{2}\right) | v-v_\ast |,
\end{equation*}
and $| \phi(v_\ast) -\phi (v'_\ast) | \leq C_\phi$. Then for any
$\delta \in (0, \pi /2)$,
\begin{align*}
\int_\sigma b(\cos\theta) | \phi(v_\ast) - \phi (v'_\ast) |
d\sigma&\leq C_\phi \left\{ | v-v_\ast| \int^\delta_0
\frac{\sin(\theta/2)}{\theta^{1+2s}} d\theta + \int^{\pi /2}_\delta
\frac{1}{\theta^{1+2s}} d\theta\right\}\\
& \leq C_\phi \left\{ | v-v_\ast| \delta^{-2s +1} +
\delta^{-2s}\right\}.
\end{align*}
If $| v-v_\ast|^{-1} \leq \frac\pi 2$, by choosing $\delta = |
v-v_\ast|^{-1}$, we get
$$
\int_\sigma b(\cos\theta) | \phi(v_\ast) - \phi (v'_\ast) |
d\sigma\leq C_\phi | v-v_\ast|^{2s}\leq C \langle v\rangle^{2s}
\langle v_*\rangle^{2s}.
$$
If $| v-v_\ast | \leq \frac{2}{\pi}$, we have
\begin{align*}
\int_\sigma b(\cos\theta) | \phi(v_\ast) - \phi (v'_\ast) |
d\sigma\leq C_\phi | v-v_\ast | \leq C_\phi\frac{2}{\pi} \leq C
\langle v\rangle^{2s} \langle v_*\rangle^{2s}.
\end{align*}
And this completes the proof of the lemma.
\end{proof}

 Up to the kernel of $\cL$, the following lemma
 gives the equivalence between the non-isotropic norm
 and the Dirichlet form of $\cL$.

\begin{lemm}\label{lemm2.2.3}
For
$g\in \mathcal{N}^{\perp}$, we have
\begin{equation}\label{2.2.2+11}
 \Big(\cL g,\, g\Big)_{L^2(\RR^3_v)} \sim ||| g|||^2.
\end{equation}
Here $A\sim B$ means that there exists two generic constants
 $C_1, C_2>0$ such that $C_1A\le B\le C_2A$.
\end{lemm}

\begin{proof} We first deal with the lower bound estimate
starting with the terms linked to $\cL_2$. Since
$$ -\Big(\cL_2 g,\, g\Big)_{L^2(\RR^3_v)} =\Big(\Gamma(g,\,
\sqrt\mu\,), g\Big)_{L^2(\RR^3)},$$
we get {}from (\ref{1formula}) that
\begin{equation}\label{radja-2}
-\Big(\cL_2 g,\, g\Big)_{L^2(\RR^3_v)} = \Big(Q (\sqrt{\mu} g,\, \sqrt{\mu} ),\, g\Big)_{L^2(\RR^3_v)} + \iiint
b(\cos\theta)\Big(\sqrt{\mu_\ast}\, - \sqrt{\mu'_\ast}\, \Big)
g'_\ast \sqrt{\mu'}\, g .
\end{equation}

Using (\ref{2.1.2}) with $\alpha=0, m=0$, the first term on the right hand side of \eqref{radja-2} can be estimated by
\begin{align*}
\left|(Q (\sqrt{\mu} g,\, \sqrt{\mu} ),\,
g\Big)_{L^2(\RR^3_v)}\right|&\leq || Q (\sqrt{\mu} g, \sqrt{\mu}
)||_{L^2}\|g\|_{L^2}\\
&\leq C || \sqrt{\mu} g||_{L^1_{2s}} || \sqrt{\mu}
||_{H^{2s}_{2s}}\|g\|_{L^2} \leq C || g ||^2_{L^2}.
\end{align*}
For the second term on the right hand side of \eqref{radja-2}, we have
\begin{align*}
&\iiint b(\cos\theta)\Big(\sqrt{\mu_\ast}\, - \sqrt{\mu'_\ast}\,
\Big)\, g'_\ast \sqrt{\mu'}\, g\,dv dv_* d\sigma\\
&= \iiint b(\cos\theta)\Big(\sqrt{\mu_\ast}\, - \sqrt{\mu'_\ast}\,
\Big)\, g'_\ast (\mu')^{1/4}\,\Big((\mu')^{1/4}-(\mu)^{1/4}\Big)
g\,\\
&+\iiint b(\cos\theta)\Big(\sqrt{\mu_\ast}\, - \sqrt{\mu'_\ast}\,
\Big)\, g'_\ast (\mu')^{1/4}\,(\mu)^{1/4} g.
\end{align*}
Thus,
\begin{align*}
&\left| \iiint b(\cos\theta)\Big(\sqrt{\mu_\ast}\, -
\sqrt{\mu'_\ast}\, \Big)\, g'_\ast \sqrt{\mu'}\, g\,\right| \\
&\leq \left( \iiint b(\cos\theta) \left( \sqrt{\mu_*} -
\sqrt{\mu'_*}\,\, \, \right)^2 |g |^2 (\mu')^{1/4}
\right)^{1/2}\\
&\qquad\qquad\times\left( \iiint b(\cos\theta) \left((\mu')^{1/4}\,
- (\mu)^{1/4}\, \right)^2\,\, |g'_*|^2 (\mu')^{1/4} \right)^{1/2}
\\
&+ \left(\iiint b(\cos\theta) \left| \sqrt{\mu_*} -
\sqrt{\mu'_*}\,\, \, \right| |g |^2 (\mu')^{1/4}(\mu)^{1/4}
\right)^{1/2}\\
&\qquad\qquad\times\left( \iiint b(\cos\theta) \left| \sqrt{\mu_*} -
\sqrt{\mu'_*}\,\, \, \right|\,\, |g'_*|^2 (\mu')^{1/4}(\mu)^{1/4}
\right)^{1/2}\\
&\leq I_1^{1/2}\times I_2^{1/2}+I_3^{1/2}\times I_4^{1/2} .
\end{align*}
Using Lemma \ref{lemm2.2.2} with $\phi=\mu^{1/4}$ gives
$$\int_\sigma b(\cos\theta) \Big| \big(\mu'_\ast\big)^{1/4}
- \big(\mu_\ast\big)^{1/4} \Big| d\sigma
\leq C | v-v_\ast |^{2s} \leq C < v>^{2s} < v_\ast >^{2s}.
$$
Since $\big(\mu'_\ast\big)^{1/4} (\mu')^{1/2} = (\mu'_*)^{1/4}
(\mu')^{1/4} (\mu')^{1/4} = (\mu_*)^{1/4} \mu^{1/4} (\mu')^{1/4}$,
we get
\begin{align*}
&I_1+I_3\leq C \iiint b(\cos\theta) | (\mu_*)^{1/2} - (\mu'_*)^{1/2}
|\,\, |g|^2
(\mu')^{1/2}dv dv_* d\sigma\\
&\leq C  \iiint b(\cos\theta)\Big| (\mu_*)^{1/4} - (\mu'_*)^{1/4}
\Big| \Big( (\mu_*)^{1/4} + (\mu'_*)^{1/4} \Big)
|g|^2 (\mu')^{1/2}\\
&\leq C \iiint b(\cos\theta) \Big|(\mu_*)^{1/4} -
(\mu'_*)^{1/4}\Big|
(\mu_*)^{1/4} |g|^2 \\
&+ C\iiint b(\cos\theta) \Big|(\mu_*)^{1/4} -(\mu'_*)^{1/4}
\Big|(\mu'_*)^{1/4} (\mu')^{1/2}|g|^2\\
&\leq C\iint \Big(\langle v_*\rangle^{2s}(\mu_*)^{1/4} \langle
v\rangle^{2s} |g|^2+\langle v_*\rangle^{2s}(\mu_*)^{1/4} \langle
v\rangle^{2s}\mu^{1/4} |g|^2\Big) dv dv_*\\
&\leq C (\|g\|^2_{L^2_s(\RR^3)}+\|g\|^2_{L^2(\RR^3)}).
\end{align*}
For $I_2$, by using the change of variables $(v,v_\ast )
\rightarrow (v_*,v)$ and then $(v', v'_\ast ) \rightarrow (v,v_\ast) $, one has
\begin{align*}
 &\iiint b(\cos\theta) \left((\mu')^{1/4}\,
- (\mu)^{1/4}\, \right)^2\,\, |g'_*|^2 (\mu')^{1/4}\\
&= \iiint b(\cos\theta) \left((\mu'_*)^{1/4}\,
- (\mu_*)^{1/4}\, \right)^2\,\, |g|^2 (\mu_*)^{1/4}\\
&\leq C\iint \Big(\langle v_*\rangle^{2s}(\mu_*)^{1/4} \langle
v\rangle^{2s} |g|^2\Big) dv dv_*\leq C \|g\|^2_{L^2_s(\RR^3)}.
\end{align*}
For  $I_4$, using the change of variables $(v,v_\ast )\rightarrow (v',v'_\ast )$ implies that
\begin{align*}
 &\iiint b(\cos\theta) \left| \sqrt{\mu_*} -
\sqrt{\mu'_*}\,\, \, \right|\,\, |g'_*|^2 (\mu')^{1/4}(\mu)^{1/4}\\
&= \iiint b(\cos\theta) \left| \sqrt{\mu_*} -
\sqrt{\mu'_*}\,\, \, \right|\,\, |g_*|^2 (\mu')^{1/4}(\mu)^{1/4}\\
&\leq C\iint \Big(\langle v\rangle^{2s}(\mu)^{1/4} \langle
v_*\rangle^{2s} |g_*|^2\Big) dv dv_*\leq C \|g\|^2_{L^2_s(\RR^3)}.
\end{align*}
In summary, we obtain
\begin{equation}\label{2.2.7}
|(\cL_2 g , g ) | \leq C \| g\|^2_{L^2_s}.
\end{equation}

For the term involving $\cL_1$, using (\ref{2.2.5}) yields
\begin{align*}
&\Big(\cL_1 g,\, g\Big)_{L^2(\RR^3_v)}= -\Big(\Gamma(\sqrt{\mu},\,
g),\, g\Big)_{L^2(\RR^3_v)} \\
&= \frac 1 2 \iiint b(\cos\theta) \left(
(\mu'_*)^{1/2} g' - (\mu_*)^{1/2}\, g \right)^2\\
&= \frac 12 \iiint b(\cos\theta)  \left( (\mu'_*)^{1/2} (g' -g)+
g((\mu'_*)^{1/2}\,-(\mu_*)^{1/2})\right)^2\\
&\geq \frac 1 4 \iiint b(\cos\theta) \mu '_\ast (g'-g)^2 - \frac 12
\iiint b(\cos\theta) g^2 \Big((\mu'_*)^{1/2} -(\mu_*)^{1/2} \Big)^2,
\end{align*}
where we used the inequality $(a+b)^2 \geq {\frac 1 2} a^2 -b^2$. Then
\begin{align*}
\Big(\cL_1 g,\, g\Big)_{L^2(\RR^3_v)}&\geq \frac 1 4 \left(\iiint
b(\cos\theta) \mu '_\ast (g'-g)^2 + \iiint b(\cos\theta) g^2
\Big((\mu'_*)^{1/2} -(\mu_*)^{1/2} \Big)^2\right)\\
&- \frac 34 \iiint b(\cos\theta) g^2 \Big((\mu'_*)^{1/2}
-(\mu_*)^{1/2} \Big)^2.
\end{align*}
We now apply (\ref{2.2.3+})  and
the change of variables $(v,v_\ast )\rightarrow (v_*,v)$ to get
$$\iiint b(\cos\theta) g^2 \Big((\mu'_*)^{1/2}
-(\mu_*)^{1/2} \Big)^2 \leq C || g||^2_{L^2_s} ||
\mu^{1/2}\,\,||^2_{H^s_s}\leq C || g||^2_{L^2_s}\,.
$$
Therefore,
\begin{equation*}\label{2.2.10}
\Big(\cL_1 g,\, g\Big)_{L^2} \geq \frac 1 4 ||| g|||^2 - C \|
g\|^2_{L^2_s}.
\end{equation*}
Thus, we have {}from (\ref{2.2.7})
\begin{align*}
\Big(\cL g,\, g\Big)_{L^2} &= \Big(\cL_1  g,  g\Big)_{L^2} + \Big(\cL_2
 g,  g\Big)_{L^2}\\
&\geq \frac 14 ||| g|||^2  -C ||  g||^2_{L^2_s}.
\end{align*}
By Theorem \ref{theo2.2}, we have {}from
the assumption  $g\in \mathcal{N}^{\perp}$ that
$$
|||  g |||^2 \leq 4\Big(\cL g,\, g\Big)_{L^2} + C ||
g||^2_{L^2_s} \leq \tilde C \Big(\cL g,\, g\Big)_{L^2}\,,
$$
which gives the lower bound estimation.

For the upper bound estimate, we have
\begin{align*}\label{2.2.11}
&\Big(\cL_1 g,\, g\Big)_{L^2(\RR^3_v)}= \frac 1 2 \iiint
b(\cos\theta) \left(
(\mu'_*)^{1/2} g' - (\mu_*)^{1/2}\, g \right)^2\nonumber\\
&= \frac 12 \iiint b(\cos\theta)  \left( (\mu'_*)^{1/2} (g' -g)+
g((\mu_*)^{1/2}\,-(\mu_*)^{1/2}) \right)^2\\
&\leq \iiint b(\cos\theta) \mu '_\ast (g'-g)^2+\iiint b(\cos\theta)
g^2 \Big((\mu'_*)^{1/2} -(\mu_*)^{1/2} \Big)^2\nonumber\\
&\leq ||| g|||^2\,\nonumber.
\end{align*}
By (\ref{2.1.5}), we have
\begin{equation*}\label{2.2.12}
\Big(\cL g,\, g\Big)_{L^2(\RR^3_v)}\leq 2||| g|||^2\,.
\end{equation*}
The proof of Lemma \ref{lemm2.2.3} is then completed.
\end{proof}
The next result shows that the non-isotropic norm controls the Sobolev norm of both derivative
 and  weight of order $s$.

\begin{lemm}\label{lemm2.2.4}
There exists $C>0$ such that
\begin{equation}\label{2.2.13}
||| g |||^2\ge C\big(|| g||^2_{H^s}+|| g||^2_{L^2_s}\big).
\end{equation}
\end{lemm}

\begin{proof} Write
\begin{align*}\label{radja-3}
||| g|||^2 &= \int_{\RR^{6}}\int_{\mathbb S^{2}}b(\cos\theta)
\mu_*\Big(g(v)-g(v')\Big)^2
d\sigma d v_* d v \\
&+ \int_{\RR^{6}}\int_{\mathbb S^{2}}b(\cos\theta)
g^2_*\Big(\mu^{1/2}(v)-\mu^{1/2}(v')\Big)^2
d\sigma d v_* d v \equiv A+B.\nonumber
\end{align*}

 According to the calculation of  Propositions 1 and  2 in \cite{al-1},
 we have
\begin{align*}
A
% \equiv 
% &
% \int_{\RR^{6}}\int_{\mathbb S^{2}}b(k\,\cdot\,\sigma)
% \mu_*\Big(g(v)-g(v')\Big)^2
% d\sigma d v_* d v\\
&=(2\pi)^{-3}\int_{\RR^{3}}\int_{\mathbb
S^{2}}b\Big(\frac{\xi}{|\xi|} \,\cdot\,\sigma\Big) \Big\{
\hat{\mu}(0)|\hat{g}(\xi)|^2+\hat{\mu}(0)|\hat{g}(\xi^+)|^2\\
&\qquad-2Re\,\hat{\mu}(\xi^-) \hat{g}(\xi^+)\bar{\hat{g}}(\xi)\Big\} d\sigma d \xi\\
&\geq \frac{1}{2(2\pi)^{3}} \int_{\RR^{3}} |\hat{g}(\xi)|^2
\left\{\int_{\mathbb S^{2}}b\Big(\frac{\xi}{|\xi|}
\,\cdot\,\sigma\Big) (\hat{\mu}(0)- |\hat{\mu}(\xi^-)|) d\sigma\right\}
d \xi
\\
&\geq C_1 \int_{|\xi|\geq 1}|\xi|^{2s} |\hat{g}(\xi)|^2d \xi
\geq C_1 2^{-2s} \int_{|\xi|\geq 1}(1+|\xi|^2)^{s} |\hat{g}(\xi)|^2 d \xi\\
&\geq C_1 2^{-2s}\|g\|^2_{H^s(\RR^3_v)}-C_1\|g\|^2_{L^2(\RR^3_v)},
\end{align*}
where we have used Lemma 3 in \cite{al-1}  that
\begin{equation}\label{lemma3}
\int_{\mathbb S^{2}}b\Big(\frac{\xi}{|\xi|}
\,\cdot\,\sigma\Big) (\hat{\mu}(0)- |\hat{\mu}(\xi^-)|) d\sigma
\geq C_1|\xi|^{2s},\qquad\,\forall |\xi|\geq 1.
\end{equation}
Similarly,
\begin{align*}
B%\equiv 
%=&\int_{\RR^{6}}\int_{\mathbb S^{2}}b(k\,\cdot\,\sigma)
% g^2_*\Big(\mu^{1/2}(v)-\mu^{1/2}(v')\Big)^2
% d\sigma d v_* d v\\
&=(2\pi)^{-3}\int_{\RR^{3}}\int_{\mathbb
S^{2}}b\Big(\frac{\xi}{|\xi|} \,\cdot\,\sigma\Big) \Big\{
\widehat{g^2}(0)|\widehat{\mu^{1/2}}(\xi)|^2+\widehat{g^2}(0)
|\widehat{\mu^{1/2}}(\xi^+)|^2\\
&\qquad-2Re\,\widehat{g^2}(\xi^-) \widehat{\mu^{1/2}}(\xi^+)\,
\overline{\widehat{\mu^{1/2}}}
(\xi)\Big\} d\sigma d \xi\\
&= \frac{1}{2(2\pi)^{3}} \int_{\RR^{3}}\int_{\mathbb S^{2}}
b\Big(\frac{\xi}{|\xi|}
\,\cdot\,\sigma\Big)\widehat{g^2}(0) \big|\widehat{\mu^{1/2}}(\xi^+)-
\widehat{\mu^{1/2}}(\xi)\big|^2 d\sigma d \xi
\\
&+\frac{1}{(2\pi)^{3}} \int_{\RR^{3}}\int_{\mathbb S^{2}}
b\Big(\frac{\xi}{|\xi|}
\,\cdot\,\sigma\Big)\Big(\widehat{g^2}(0)-Re\, \widehat{g^2}(\xi^-)\Big)
\overline{\widehat{\mu^{1/2}}}(\xi)\,
\widehat{\mu^{1/2}}(\xi^+) d\sigma d \xi \\
&=B_1+B_2\, .
\end{align*}
For $B_1$, one has
\begin{align*}
B_1=&\int_{\RR^{3}}\int_{\mathbb S^{2}}
b\Big(\frac{\xi}{|\xi|}
\,\cdot\,\sigma\Big)\widehat{g^2}(0) \big|\widehat{\mu^{1/2}}(\xi^+)-
\widehat{\mu^{1/2}}(\xi)\big|^2 d\sigma d \xi\\
&=C_1\|g\|^2_{L^2(\RR^3_v)}\int_{\RR^{3}_\xi}
\widehat{\mu}(2\xi)
\int_{\mathbb S^{2}}
b\Big(\frac{\xi}{|\xi|}
\,\cdot\,\sigma\Big) \big|\widehat{\mu^{1/2}}(\xi^-)
-1\big|^2 d\sigma d \xi\\
&\geq C_2 \|g\|^2_{L^2(\RR^3_v)},
\end{align*}
where
$$
C_2=C_1\int_{\RR^{3}_\xi}
\widehat{\mu}(2\xi)
\int_{\mathbb S^{2}}
b\Big(\frac{\xi}{|\xi|}
\,\cdot\,\sigma\Big) \big|\widehat{\mu^{1/2}}(\xi^-)
-1\big|^2 d\sigma d \xi>0.
$$

For the second term on the right hand side, by using
$$
\overline{\widehat{\mu^{1/2}}}(\xi)\,
\widehat{\mu^{1/2}}(\xi^+)\geq C\widehat{\mu}(2\xi),
$$
for some positive constant $C$, we have
\begin{align*}
B_2=&\int_{\RR^{3}}\int_{\mathbb S^{2}}
b\Big(\frac{\xi}{|\xi|}
\,\cdot\,\sigma\Big)\Big(\widehat{g^2}(0)-Re\, \widehat{g^2}(\xi^-)\Big)
\overline{\widehat{\mu^{1/2}}}(\xi)\,
\widehat{\mu^{1/2}}(\xi^+) d\sigma d \xi\\
&\geq C\int_{\RR^{3}}\int_{\mathbb S^{2}}
b\Big(\frac{\xi}{|\xi|}
\,\cdot\,\sigma\Big)\Big(\widehat{g^2}(0)-Re\, \widehat{g^2}(\xi^-)\Big)
\widehat{\mu}(2\xi) d\sigma d \xi.
\\
&= C\int_{\RR^{3}}\int_{\mathbb S^{2}}
b\Big(\frac{\xi}{|\xi|}
\,\cdot\,\sigma\Big)\int_{\RR^3_v} g^2(v)\Big(1-\cos(\xi^-\,\cdot\, v)\Big)d v
\widehat{\mu}(2\xi) d\sigma d \xi.
\end{align*}
We now use Bobylev's technique \cite{bobylev} to have
\[
\int_{\mathbb S^{2}}
b\Big(\frac{\xi}{|\xi|}
\,\cdot\,\sigma\Big)\psi(\xi^-\cdot v)d\sigma
=\int_{\mathbb S^{2}}
b\Big(\frac{v}{|v|}
\,\cdot\,\sigma\Big)\psi(\xi\cdot v^-)d\sigma,
\]
so that
\begin{align*}
B_2
&\ge C\int_{\RR^3_v} g^2(v) \left(
\int_{\RR^{3}}\int_{\mathbb S^{2}}
b\Big(\frac{v}{|v|}
\,\cdot\,\sigma\Big)\Big(1-\cos(\xi\,\cdot\, v^-)\Big)
\widehat{\mu}(2\xi) d\sigma d \xi \right)d v\\
&=C\int_{\RR^3_v} g^2(v) \left(
\int_{\mathbb S^{2}}
b\Big(\frac{v}{|v|}
\,\cdot\,\sigma\Big)\Big(
{\mu}(0) -{\mu} \big(\frac{v^-}{2}\big)\Big)
d\sigma  \right)d v\\
&\geq C \int_{|v|\geq 1} g^2(v)|v|^{2s}dv\geq C 2^{-2s}\|g\|^2_{L^2_s(\RR^3_v)}
-C\|g\|^2_{L^2(\RR^3_v)}.
\end{align*}
where 
% Here, on the last 4th line in the above we have used Bobylev's formula \cite{bobylev}
% and on the second last line 
we have used (2.2.9) and the change of variables in $\sigma$
by exchanging $\xi/|\xi|$ and $v/|v|$.

Finally, by choosing a suitably small constant $0<\lambda<1$,
\begin{align*}
||| g |||^2&=A+B_1+B_2 \geq \lambda A+B_1+\lambda B_2\\
&\geq C(\|g\|^2_{H^s(\RR^3_v)}+
\|g\|^2_{L^2(\RR^3_v)}),
\end{align*}
and this concludes the proof of the lemma.
\end{proof}

%%%%%%%%%%%%%%%%%%%%%%%%%%%%%%%%%%%%%%%%%%%%%%%%%%%%%%%%%%%%%
\subsection{Upper bound estimates}\label{section2.3}
\setcounter{equation}{0}
\smallskip

 To apply the energy method, we  need  some upper bound
 estimate on the collision operator
 in terms of the non-isotropic norm which will be given in the
 following proposition. For  this, we first prove

\begin{lemm}\label{lemm2.3.1}
 There exists
$C>0$ such that
\begin{equation}\label{2.3.1}
\iiint b(\cos\theta) f^2_\ast (g'-g)^2 \leq C\, || f||^2_{L^2_s}
\,||| g|||^2\,.
\end{equation}
\end{lemm}
\begin{proof}
Different {}from Lemma
\ref{lemm2.2.1},  we apply Bobylev formula \cite{bobylev} to have
\begin{align*}
&\iiint b(\cos\theta) \mu_\ast (g'-g)^2dv_*d\sigma dv \\
&=\frac{1}{(2\pi)^3}\iint b \left(\frac{\xi}{|\xi|}\cdot \sigma
\right) \Big(\hat \mu (0) (|\hat g (\xi )|^2 + |\hat g (\xi^+ )|^2 )- 2
Re\,\, \hat \mu (\xi^-)  \hat g (\xi^+ ) \overline{\hat g} (\xi )
\Big )d\xi d\sigma \\
&=\frac{1}{(2\pi)^3}\iint b \Big(\frac{\xi}{|\xi|}\cdot \sigma \Big)
\Big(\hat \mu (0) | \hat g (\xi ) - \hat g (\xi^+) |^2 + 2 Re\,
\Big(\hat \mu (0) - \hat \mu (\xi^-) \Big) \hat g (\xi^+ )
\overline{\hat g} (\xi ) \Big)d\xi d\sigma ,
\end{align*}
and
\begin{align*}
&\iiint b(\cos\theta) f^2_\ast (g'-g)^2dv_*d\sigma dv\\
&=\frac{1}{(2\pi)^3}\iint b \left(\frac{\xi}{|\xi|}\cdot \sigma
\right) \Big(\widehat{f^2} (0) | \hat g (\xi ) - \hat g (\xi^+) |^2+
2 Re\, \Big(\widehat{f^2} (0) - \widehat{f^2} (\xi^-) \Big) \hat g
(\xi^+ ) \overline{\hat g} (\xi ) \Big)d\xi d\sigma\, .
\end{align*}
Since $\hat \mu (0)=1, \widehat{f^2} (0)=\|f\|^2_{L^2}$, we obtain
\begin{align*}
&\iiint b(\cos\theta) f^2_\ast (g'-g)^2dv_*d\sigma dv\\
&=\|f\|^2_{L^2}\iiint b(\cos\theta) \mu_\ast (g'-g)^2dv_*d\sigma dv\\
&\qquad-\frac{2}{(2\pi)^3}\|f\|^2_{L^2}\iint b
\left(\frac{\xi}{|\xi|}\cdot \sigma \right)\,Re\, \Big(\hat \mu (0)
- \hat \mu (\xi^-) \Big) \hat g (\xi^+ ) \overline{\hat g} (\xi )
\Big)d\xi d\sigma\\
&\qquad\qquad+ \frac{2}{(2\pi)^3}\iint b
\left(\frac{\xi}{|\xi|}\cdot \sigma \right) Re\, \Big(\widehat{f^2}
(0) - \widehat{f^2} (\xi^-) \Big) \hat g (\xi^+ ) \overline{\hat g}
(\xi ) \Big)d\xi d\sigma\, .
\end{align*}
For the last term, we note that
$$
\int b\left(\frac{\xi}{|\xi|}\cdot \sigma \right) | \widehat{f^2} (0)
- \widehat{f^2} (\xi^- ) |d\sigma
 \leq \int b\left(\frac{\xi}{|\xi|}\cdot \sigma \right)
 \left(\int_v f^2 (v) \Big| 1- e^{-iv\cdot\xi^-}
 \Big| dv\right) d\sigma.
$$
Now consider
$$
\int_{\SS^2} b\left(\frac{\xi}{|\xi|}\cdot \sigma \right) \Big| 1-
e^{-iv\cdot\xi^-} \Big| d\sigma.
$$
If $| v| | \xi | \geq \frac{2}{\pi}$, we choose $\delta = \frac{1}{
|\xi| |v|}\,\,  \leq \pi /2 $ to have $| 1- e^{-iv.\xi^-} | \leq |
v| | \xi | \sin \theta$ for any $0 \leq \theta \leq \delta $. And
if ${\frac\pi 2} \geq \theta \geq \delta$, we have $| 1-
e^{-iv.\xi^-} | \leq 2$. Hence,
\begin{align*}
\int_{\SS^2} b\left(\frac{\xi}{|\xi|}\cdot \sigma \right) \Big| 1-
e^{-iv\cdot\xi^-} \Big| d\sigma&\leq C | v | | \xi | \int^\delta_0
\frac{1}{{\theta^{1+2s}}} \,\sin \theta \,d\theta + C' \int^{\pi
/2}_\delta \frac{1}{\theta^{1+2s}} d\theta\\
&\leq C | v | | \xi | \delta^{-2s +1 } + C' \delta^{-2s} \leq \tilde
C | v|^{2s} | \xi |^{2s}.
\end{align*}
On the other hand, if $| v| | \xi| \leq  \frac{2}{\pi}$, we have directly
\begin{align*}
\int_{\SS^2} b\left(\frac{\xi}{|\xi|}\cdot \sigma \right) \Big| 1-
e^{-iv\cdot\xi^-} \Big| d\sigma&\leq C \int^{\pi /2}_0
\frac{1}{{\theta^{1+2s}}} | v |\, | \xi|  \sin \theta  d\theta\\
&\leq \tilde C | v | | \xi | \leq \tilde C | v|^{2s} | \xi |^{2s}.
\end{align*}

Thus, we have
\begin{align*}
\iint b\left(\frac{\xi}{|\xi|}\cdot \sigma \right)
\big|\widehat{f^2} (0) - \widehat{f^2} (\xi^-) \big| |{\hat g}|^2 (\xi ) d\xi d\sigma\leq
C\|f\|^2_{L^2_s}\|g\|^2_{H^s}.
\end{align*}
By using the regular change of variables $\xi\rightarrow \xi^+$,
and by noticing that
 \begin{eqnarray*}
 &\xi^-=\phi(\xi^+,\sigma)=\xi^+-\frac{|\xi^+|}{\cos\frac \theta 2}\sigma,\qquad
 |\xi^-|=|\xi^+|\tan\frac\theta 2,\qquad \cos\frac\theta 2=\frac{\xi^+}{|\xi^+|}\cdot\sigma,
 \\&
\Big|\frac {\partial (\xi^+)}{\partial ( \xi)}\Big| = \frac 14 \cos^2\theta/2,
 \end{eqnarray*}
  we have
\begin{align*}
&\iint b\left(\frac{\xi}{|\xi|}\cdot \sigma \right)
\big|\widehat{f^2} (0) - \widehat{f^2} (\xi^-) \big| |{\hat g}|^2 (\xi^+ ) d\xi d\sigma\\
&=\iint \frac{1}{\cos^2\theta/2}b\left(2(\frac{\xi^+}{|\xi^+|}\cdot \sigma)^2-1 \right)
\big|\widehat{f^2} (0) - \widehat{f^2} (\phi(\xi^+,\sigma)) \big| |{\hat g}|^2 (\xi^+ )  d\xi^+ d\sigma\\
&\leq
C\|f\|^2_{L^2_s}\|g\|^2_{H^s}.
\end{align*}
Hence,
\begin{align*}
\big|\iint b
\left(\frac{\xi}{|\xi|}\cdot \sigma \right) Re\, \Big(\widehat{f^2}
(0) - \widehat{f^2} (\xi^-) \Big) \hat g (\xi^+ ) \overline{\hat g}
(\xi ) d\xi d\sigma\big|
\leq
C\|f\|^2_{L^2_s}\|g\|^2_{H^s}.
\end{align*}
Similarly, we have
\begin{align*}
\left| \iint b \left(\frac{\xi}{|\xi|}\cdot \sigma \right)\,Re\,
\Big(\hat \mu (0) - \hat \mu (\xi^-) \Big) \hat g (\xi^+ )
\overline{\hat g} (\xi ) d\xi d\sigma \right|\leq
C\|\sqrt{\mu}\,\|^2_{L^2_s}\|g\|^2_{H^s}.
\end{align*}
Therefore, we have proved (\ref{2.3.1}) by using (\ref{2.2.13}).
\end{proof}

In view of future application of the energy method, the scalar product of the collision operator with a test function is given by

\begin{prop}\label{prop2.3.1}
 There exists
$C>0$ such that
\begin{equation*}\label{2.3.2}
\left|\Big(\Gamma (f,\, g) ,\,  h\Big)_{L^2(\RR^3)} \right| \leq
C\left(|| f ||_{L^2_s}\, ||| g||| + ||
g||_{L^2_s}\,||| f|||\, \right)\,\, |||
h|||\, .
\end{equation*}
\end{prop}

\begin{proof} Note that
\begin{align*}
\Big(\Gamma (f,\, g) ,\, h\Big)_{L^2(\RR^3)} &= \Big(\mu^{-1/2}Q
(\mu^{1/2} f,\, \mu^{1/2} g ),\,  h\Big)_{L^2(\RR^3)}\\
&= \iiint b(\cos\theta) \mu^{1/2}_\ast \Big( f'_\ast g' - f_\ast g
\Big)\,\,h\\
& =\frac 12 \iiint b(\cos\theta) \Big(f'_\ast g' - f_\ast g\Big)
\Big( \mu^{1/2}_\ast h - \mu^{1/2 '}_\ast h' \Big)\\
&\leq \frac 12 \left(\iiint b (\cos\theta) \Big(f'_\ast g'- f_\ast
g\Big)^2 \right)^{1/2}\\
&\qquad\qquad\times \left( \iiint b (\cos\theta) \Big((\mu_*)^{1/2}
h - (\mu'_*)^{1/2} h' \Big)^2 \right)^{1/2}\\
&\leq \frac 12  A^{1/2}\times B^{1/2}.
\end{align*}
For $B$, we have
\begin{align*}
B
% & 
% = \iiint b (\cos\theta) \Big( (\mu'_*)^{1/2} (h'-h) + h
% \big((\mu'_*)^{1/2} - (\mu_*)^{1/2}\big)\Big)^2\\
% &\leq 2 \iiint b(\cos\theta) \left\{\mu'_\ast (h'-h)^2 +  h^2
% \Big((\mu'_*)^{1/2} -(\mu_*)^{1/2} \Big)^2 \right\}
% \\
&\leq 2 \iiint b(\cos\theta) \mu_\ast (h'-h)^2 +  2 \iiint
b(\cos\theta) h^2_\ast \Big((\mu')^{1/2} -\mu^{1/2} \Big)^2 = 2 |||
h|||^2,
\end{align*}
where we have used the change of variables $(v,v_\ast )\rightarrow (v',v'_\ast )$ for the first
term and $(v,v_\ast )\rightarrow (v_*,v)$ for the second term. Similarly,
\begin{align*}
A
% & = \iiint b (\cos\theta) \Big(  f'_\ast (g'-g) + g (f'_\ast
% -f_\ast )\Big)^2\\
% &\leq 2 \iiint b(\cos\theta) \left\{ {f'_*}^{2} (g'-g)^2 + g^2
% (f'_\ast -f_\ast )^2 \right\}\\
&\leq 2 \iiint b(\cos\theta) {f_*}^{2} (g'-g)^2+ 2 \iiint
b(\cos\theta) {g_*}^{2} (f'-f)^2.
\end{align*}
Then (\ref{2.3.1}) implies that
$$
A\leq C\left( || f||^2_{L^2_s} \,
||| g|||^2+ || g||^2_{L^2_s}\, ||| f|||^2\, \right),
$$
which completes the proof of the proposition.
\end{proof}

%%%%%%%%%%%%%%%%%%%%%%%%%%%%%%%%%%%%%%%%%%%%%%%%%%%%%%%%%%%%%%%%%%%%%%
%%%%%%%%%%%%%%%%%%%%%%%%%%%%%%%%%%%%%%%%%%%%%%%%%%%%%%%%%%%%%%%%%%%%%%
\section{Commutator estimates}\label{section3}
\smallbreak

%%%%%%%%%%%%%%%%%%%%%%%%%%%%%%%%%%%%%%%%%%%%%%%%%%%%%%%%%%%%%%%%%%%%%%%
\subsection{Non-isotropic norm in $\RR^6_{x,v}$}\label{section3.1}
\setcounter{equation}{0} \smallbreak

We now define the norm associated with the collision operator on the
space of  $(x, v)$.  For $m\in\NN, \ell\in \RR$, set
$$
\cB^m_\ell(\RR^6_{x, v})=\left\{ g\in\cS'(\RR^6_{x, v});\,\,
|||g|||^2_{\cB^m_\ell(\RR^6)}=\sum_{|\alpha|\leq m}\int_{\RR^3_x}|||W^\ell\,
\partial^\alpha_{x, v} g(x,\, \cdot\,)|||^2dx <+\infty\right\},
$$
where $|||\,\cdot\,|||$ is the non-isotropic norm
defined in (\ref{2.2.1}).

First of all, one has
\begin{lemm}\label{lemm3.1.1} For any $\ell\geq 0, \gamma, \beta\in  \NN^3$,
\begin{equation}\label{3.1.0}
|||W^\ell\partial^{\gamma}_x\partial^{\beta}_v\,\,\pP
g|||_{\cB^0_0(\RR^6)}+|||\pP (W^\ell
\partial^{\gamma}_x\partial^{\beta}_v\,\, g)|||_{\cB^0_0(\RR^6)}
\leq C_{\ell, \beta}||\partial^{\gamma}_x
g||^2_{L^2(\RR^6)},
\end{equation}
\begin{equation}\label{3.1.0+1}
C_0|||g|||^2_{\cB^0_0(\RR^6)}-C_2||g||^2_{L^2(\RR^6)} \leq \Big(\cL
g,\, g\Big)_{L^2(\RR^6_{x, v})}
 \leq C_3 |||g|||^2_{\cB^0_0(\RR^6)},
\end{equation}
and
\begin{equation}\label{3.1.0+2}
||g||^2_{L^2_{l+s}(\RR^6)} +||g||^2_{L^2(\RR^3_x; H^s_{l}(\RR^3_v))} \leq
C|||g|||^2_{\cB^0_l(\RR^6)}\leq C ||g||^2_{L^2(\RR^3_x;
H^{s}_{l+s}(\RR^3_v))}.
\end{equation}
\end{lemm}

\begin{proof} By definition of the projection operator $\pP$, we have
$$
\pP g = a_g (t,x) \mu^{1/2} + \sum^3_{j=1} b_{g, j} (t,x)\, v_j
\mu^{1/2} + c_ g (t,x) | v|^2 \mu^{1/2},
$$
with
\[
a_g (t, x) = \int_{\RR^3_v} g(t,x,v) \mu^{1/2} (v) dv,\,\, c_g (t,x)
= \int_v g(t,x,v) | v|^2 \mu^{1/2}(v) dv,
\]
and
\[
b_{g, j}(t, x) = \int_v g (t,x,v)\, v_j \mu^{1/2} (v) dv,\qquad\qquad j=1,2,3.
\]
Thus (\ref{3.1.0}) can be obtained by integration by parts.
To get (\ref{3.1.0+1}), we use (\ref{2.2.2}) and (\ref{2.2.2+11})
to obtain
\begin{align*}
|||g|||^2_{\cB^0_0(\RR^6)}&\geq C \Big(\cL
g,\, g\Big)_{L^2(\RR^6_{x, v})}
\geq C_0|||(\iI-\pP) g|||^2_{\cB^0_0(\RR^6)}\\
&\geq \frac{C_0}{2}|||g|||^2_{\cB^0_0(\RR^6)}-
C_0|||\pP g|||^2_{\cB^0_0(\RR^6)}\\
&\geq \frac{C_0}{2}|||g|||^2_{\cB^0_0(\RR^6)}
-C_2||g||^2_{L^2(\RR^6)}.
\end{align*}
Finally, (\ref{3.1.0+2}) follows directly {}from (\ref{2.2.2}) and (\ref{2.2.13}).
\end{proof}

%%%%%%%%%%%%%%%%%%%%%%%%%%%%%%%%%%%%%%%%%%%%%%%%%%%%%%%%%%%%%%%%%%%%%%%
\subsection{Weighted estimates on commutators}\label{section3.2}
\setcounter{equation}{0}
 We will use the following notation, for $\gamma\in\NN^3$,
\begin{equation}\label{3.1.1}
\cT(F,\, G,\, \mu_\gamma\,)=Q(\mu_\gamma\, F,\, G)+ \iint
b(\cos\theta) \Big((\mu_\gamma)_*\, -(\mu_\gamma)'_*\,\,\Big)F'_* G'
d v_* d \sigma\,,
\end{equation}
where $\mu_\gamma=p_{\gamma}(v)\sqrt{\mu(v)}=\partial^\gamma
(\sqrt\mu\,)\,$ is a Maxwellian type function of variable $v$.

In this notation,
(\ref{1formula}) is equivalent to
$$
\Gamma(f, g) = \cT(f, g, \sqrt\mu\,).
$$
And the Leibniz formula gives
\begin{equation}\label{3.1.2}
\partial^\alpha_x\partial^\beta_v\Gamma(f,\, g)=
\sum_{\alpha_1+\alpha_2=\alpha,\,\, \beta_1+\beta_2+\beta_3=\beta}
C^{\alpha_1, \alpha_2}_{\beta_1, \beta_2, \beta_3}
\cT(\partial^{\alpha_1}_x\partial^{\beta_1}_v f,\,
\partial^{\alpha_2}_x\partial^{\beta_2}_v g,\, \mu_{\beta_3}\,)\, .
\end{equation}

First of all, let us recall the following lemma {}from \cite{amuxy-nonlinear-3}.
\begin{lemm}\label{lemm3.2.1}
Let $\ell \geq 0, 0<s<1/2 $. There exists $C>0$
such that
\begin{equation*}\label{3.2.1}
\left|\Big(\big(W^\ell  \,\, Q(f,\,\,g)-Q(f,\,\,W^\ell  \,\,
g)\big),\,\,\, h\Big)_{L^2(\RR^3)}\right| \leq C
\|f\|_{L^1_{\ell}(\RR^3_v))}\|g\|_{L^2_{\ell} (\RR^3)}
\|h\|_{L^2(\RR^3)}.
\end{equation*}
\end{lemm}

Using this result, we shall show that

\begin{prop}\label{prop3.2.1}
For any $\ell\geq 0$,
\begin{equation}\label{3.2.2}
\Big(W^\ell \cT(F,\, G,\, \mu_\gamma\,)\, -\cT(F,\, W^\ell\,G,\,
\mu_\gamma\,),\,\, h\Big)_{L^2(\RR^3_{v})}\leq C||F||_{L^2_\ell} ||
G ||_{L^2_\ell} || h ||_{L^2_s}.
\end{equation}
\end{prop}
\begin{proof} {}From (\ref{3.1.1}), it follows that
\begin{align*}
&\Big(W^\ell \cT(F,\, G,\, \mu_\gamma\,)\, -\cT(F,\, W^\ell\,G,\,
\mu_\gamma\,),\,\, h\Big)_{L^2(\RR^3_{v})}\\
 &= \Big(W^\ell\,
Q(\mu_\gamma\, F, \,G)-Q(\mu_\gamma F , W^\ell G ),\,
h\Big)_{L^2(\RR^3_{v})}\\
& + \iiint b(\cos\theta) (\mu_{\gamma\ast} - \mu '_{\gamma\ast} )
F'_\ast G' \big(W^\ell -W'^\ell\big)\,h\\
&= B_1 + B_2.
\end{align*}
Lemma \ref{lemm3.2.1} implies that
$$
B_1 \leq C||\mu_\gamma F||_{L^1_\ell} || G ||_{L^2_\ell} || h
||_{L^2}\leq C||F||_{L^2} || G ||_{L^2_\ell} || h ||_{L^2}.
$$
For $B_2$, since we have assumed that $0<s< 1/2$, we get
\begin{align*}
B_2 &\leq \left(\iiint b(\cos\theta) | F'_\ast |^2 | G'|^2
\frac{{| W^\ell - {W'}^\ell |^2}}{\sin \theta} \right)^{1/2}
\left(\iiint b(\cos\theta) \sin \theta \Big( \mu_{\gamma\ast} -
\mu'_{\gamma\ast} \Big)^2 |h|^2 \right)^{1/2}\, .
\end{align*}
(\ref{2.2.3+}) implies that
\begin{align*}
& \iiint b(\cos\theta) \Big( \mu_{\gamma\ast} - \mu'_{\gamma\ast}
\Big)^2 |h|^2 \leq C || \mu_\gamma ||^2_{H^s_s} ||h||^2_{L^2_s},
\end{align*}
while, using
$$
| W^\ell - W'^\ell |^2 \leq \sin^2\theta\Big( (W^\ell_\ast)^2 +
(W'^{\ell} )^2 \Big) \leq \sin^2\theta \,\,(W^\ell_\ast)^2
(W'^{\ell} )^2,
$$
we get
\begin{align*}
\iiint b(\cos\theta) | F'_\ast |^2 | G'|^2 \frac{{| W^\ell -
{W'}^\ell |^2}}{\sin \theta} &\leq \iiint b(\cos\theta) \sin\theta\,
(W^\ell\, F )'^{2}_\ast (W^\ell\, G )'^{2}\\
&\leq C || F||^2_{L^2_\ell} || G ||^2_{L^2_\ell},
\end{align*}
which leads to  completion of  the proof of the
proposition.
\end{proof}
Similarly, we have also

\begin{prop}\label{prop3.2.2}
There exists a constant $C>0$ such that
\begin{align}\label{3.2.4}
\left|\Big(\cT(F,\, G,\, \mu_\gamma\,),\,\,
h\Big)_{L^2(\RR^3_{v})}\right|\leq C\left( || F||_{L^2_s} \,|||
G|||+ ||
G||_{L^2_s}\,||| F||| \right) ||| h|||.
\end{align}
\end{prop}

\begin{proof} By the Cauchy-Schwarz inequality, we have
\begin{align*}
\Big(\cT(F,\, G,\, \mu_\gamma\,),\,\, h\Big)_{L^2(\RR^3_{v})} &=
\iiint b(\cos\theta) (\mu_{\gamma *})^{1/2} \Big( F'_\ast G' -
F_\ast G \Big)\,\,h\\
& =\frac 12 \iiint b(\cos\theta) \Big(F'_\ast G' - F_\ast G\Big)
\Big( (\mu_{\gamma *})^{1/2} h -(\mu'_{\gamma *})^{1/2} h' \Big)\\
&\leq \frac 12 \left(\iiint (\cos\theta) \Big(F'_\ast G'- F_\ast
G\Big)^2 \right)^{1/2}\\
&\qquad\qquad\times \left( \iiint b (\cos\theta) \Big((\mu_{\gamma
*})^{1/2} h - (\mu'_{\gamma *})^{1/2} h' \Big)^2 \right)^{1/2}\\
&\leq \frac 12  \tilde A^{1/2}\times \tilde B^{1/2}.
\end{align*}
By using the estimation of the term ${A}$ in the proof of Proposition
\ref{prop2.3.1}, it follows that
$$
\tilde A\leq C\left( || F||^2_{L^2_s} \,||| G|||^2 + || G||^2_{L^2_s}\,
||| F|||^2 \right)
$$
and
$$
\tilde B\leq C\Big( ||\mu_\gamma||^2_{L^2_s} \,||| h|||^2 + || h||^2_{L^2_s}\,
||| \mu_\gamma|||^2\Big) \leq C ||| h|||^2.
$$
\end{proof}

We are now ready to prove the following estimate with differentiation and weight.

\begin{prop}\label{prop3.2.3}
For any $\ell\geq 3$, and $ N\geq 3$, we have, for all
$\beta\in\NN^6, |\beta|\leq N$,
\begin{equation*}\label{3.2.5}
\left|\Big(W^\ell \,\partial^\beta_{x, v} \Gamma(f,\, g\,),\,\,
h\Big)_{L^2(\RR^6_{x, v})}\right|\leq C ||f||_{H^N_\ell(\RR^6)}\,\,
||| g|||_{\cB^{N}_\ell(\RR^6)}\,\, ||| h|||_{\cB^0_0(\RR^6)}.
\end{equation*}
\end{prop}

\begin{rema}
In fact, this proposition holds even when
 $\ell> \frac 32 +2s$, and $ N> \frac 32 +2s$.
  Here, we consider the case
   when $\ell\geq 3, N\geq 3$ with $0<s<1/2$ for the
    simplicity of  the notations.
\end{rema}

\begin{proof}
Using the Leibniz formula (\ref{3.1.2}) gives
\begin{align*}%\label{3.2.5+1}
&\Big(W^\ell \partial^\beta_{x, v}\Gamma(f,\, g),\,\,
h\Big)_{L^2(\RR^6_{x, v})}= \sum C^{\beta_1}_{\beta_2, \beta_3}
\Big(\cT(\partial^{\beta_1} f,\, W^\ell \partial^{\beta_2} g,\,
\mu_{\beta_3}\,) ,\,\, h\Big)_{L^2(\RR^6_{x, v})}\\
&+\sum C^{\beta_1}_{\beta_2, \beta_3}
\Big(W^\ell\cT(\partial^{\beta_1} f,\, \partial^{\beta_2} g,\,
\mu_{\beta_3}\,) - \cT(\partial^{\beta_1} f,\, W^\ell
\partial^{\beta_2} g,\, \mu_{\beta_3}\,) ,\,\,
h\Big)_{L^2(\RR^6_{x, v})}.%\nonumber
\end{align*}
Then {}from (\ref{3.2.2}), we get
\begin{align*}
& \left|\Big(W^\ell\cT(\partial^{\beta_1} f,\,
\partial^{\beta_2} g,\, \mu_{\beta_3}\,)
- \cT(\partial^{\beta_1}f,\, W^\ell
\partial^{\beta_2} g,\, \mu_{\beta_3}\,)
,\,\, h\Big)_{L^2(\RR^6_{x, v})}\right|\\
&\leq C \left(\int_{\RR^3_x} \|\partial^{\beta_1}
f\|^2_{L^2_\ell(\RR^3_v)} \|\partial^{\beta_2}
g\|^2_{L^2_\ell(\RR^3_v)} dx\right)^{1/2}\|h\|_{L^2_s(\RR^6_{x,
v})}\\
&\leq C\left\{\begin{array}{ll} \|\partial^{\beta_1}
f\|_{L^\infty(\RR^3_x;\, L^2_\ell(\RR^3_v))} \|\partial^{\beta_2}
g\|_{L^2_\ell(\RR^6_{x, v})}\, \|h\|_{L^2_s(\RR^6_{x, v})},
&\qquad\mbox{if}\,\, |\beta_1|\leq 1\, ;\\
\|\partial^{\beta_1} f\|_{L^2_\ell(\RR^6_{x, v})}
\|\partial^{\beta_2} g\|_{L^\infty(\RR^3_x; L^2_\ell(\RR^3_v))}\,
\|h\|_{L^2_s(\RR^6_{x, v})},&\qquad\mbox{if}\,\, |\beta_1|\geq 2\, .
\end{array}
\right.
\end{align*}
Since $|\beta_1|\leq 1$ implies $|\beta_1|+3/2< 3\leq N$ and
$|\beta_1|\geq 2$ implies $|\beta_2|+3/2< |\beta|$, it follows that
\begin{align}\label{3.2.6+1}
& \left|\Big(W^\ell\cT(\partial^{\beta_1} f,\,
\partial^{\beta_2} g,\, \mu_{\beta_3}\,)
- \cT(\partial^{\beta_1}f,\, W^\ell
\partial^{\beta_2} g,\, \mu_{\beta_3}\,)
,\,\, h\Big)_{L^2(\RR^6_{x, v})}\right|\\
&\leq C
\|f\|_{H^N_\ell(\RR^6)}\|g\|_{H^{|\beta\,|}_\ell(\RR^6)}|||h|||_{\cB^0_0(\RR^6_{x,
v})}\, .\nonumber
\end{align}
On the other hand, if $ |\beta_1|\leq
1$ so that $ |\beta_1|+\frac 3 2+s<3\leq N$, we get {}from (\ref{3.2.4})
\begin{align*}
& \left|\Big(\cT(\partial^{\beta_1} f,\, W^\ell \partial^{\beta_2}
g,\, \mu_{\beta_3}\,)
,\,\, h\Big)_{L^2(\RR^6_{x, v})}\right|\\
&\leq C \left(\int_{\RR^3_x} \|\partial^{\beta_1}
f\|^2_{L^2_s(\RR^3_v)} \Big(|||W^\ell
 \partial^{\beta_2} g|||^2+||W^\ell
 \partial^{\beta_2} g||^2_{H^s(\RR^3_v)}\Big)dx\right.\\
&\qquad+ \left.\int_{\RR^3_x}\|W^\ell \partial^{\beta_2}
g\|^2_{L^2_s(\RR^3_v)} \Big(|||
\partial^{\beta_1} f|||^2+||
\partial^{\beta_1} f||^2_{H^s(\RR^3_v)}\Big) dx\right)^{1/2}|||h|||_{\cB^0_0(\RR^6)} \\
&\leq C \Big(\|\partial^{\beta_1} f\|_{L^\infty(\RR^3_x;\,
L^2_s(\RR^3_v))} +||\partial^{\beta_1} f||^2_{L^\infty(\RR^3_x;\,
H^s_s(\RR^3_v))}\Big)|||g|||_{\cB^{|\beta_2|}_\ell(\RR^6)}
|||h|||_{\cB^0_0(\RR^6)} \\
&\leq C \| f\|_{H^{|\beta_1|+3/2+s+\epsilonup}_s(\RR^6)}
|||g|||_{\cB^{|\beta_2|}_\ell(\RR^6)} |||h|||_{\cB^0_0(\RR^6)},
\end{align*}
Hence, for $|\beta_1|\le 1$, we have
\begin{equation*}\label{3.2.6+2}
\left|\Big(\cT(\partial^{\beta_1} f,\, W^\ell \partial^{\beta_2}
g,\, \mu_{\beta_3}\,) ,\,\, h\Big)_{L^2(\RR^6_{x, v})}\right|\leq C
\| f\|_{H^{3}_s(\RR^6)} |||g|||_{\cB^{|\beta_2|}_\ell(\RR^6)}
|||h|||_{\cB^0_0(\RR^6)}\, .
\end{equation*}
We now consider the case when $|\beta_1|\geq 2$.
First of all, assume $2\leq|\beta_1|\leq
|\beta|-1$ so that $|\beta_2|= |\beta|-|\beta_1|-|\beta_3|\leq
|\beta|-2$. Then, we get
\begin{align*}
& \left|\Big(\cT(\partial^{\beta_1} f,\, W^\ell\partial^{\beta_2}
g,\, \mu_{\beta_3}\,)
,\,\, h\Big)_{L^2(\RR^6_{x, v})}\right|\\
&\leq C \Big( \|f\|_{H^{|\beta_1\,|}_s(\RR^6)} ||W^\ell
\partial^{\beta_2} g||_{L^\infty(\RR^3_x;\, H^s_s(\RR^3_v))}\\
&\qquad+ \|W^\ell\partial^{\beta_2} g\|_{L^\infty(\RR^3_x;\,
L^2_s(\RR^3_v))} ||\partial^{\beta_1}
f||_{L^2(\RR^3_x;\, H^s_s(\RR^3_v))}\Big) |||h|||_{\cB^0_0(\RR^6)} \\
&\leq C \| f\|_{H^{|\beta_1|+s}_s(\RR^6)} \,\,||W^\ell
\Lambda^{3/2+\epsilon}_x\Lambda^s_v
\partial^{\beta_2}
g||^2_{L^2_s(\RR^6)} |||h|||_{\cB^0_0(\RR^6)}\\
& \leq C \| f\|_{H^{|\beta|-1+s}_s(\RR^6)}
|||g|||_{\cB^{|\beta|-2+3/2+s+\epsilon}_\ell(\RR^6)} |||h|||_{\cB^0_0(\RR^6)}\\
& \leq C \| f\|_{H^{|\beta|}_s(\RR^6)}
|||g|||_{\cB^{|\beta|}_\ell(\RR^6)} |||h|||_{\cB^0_0(\RR^6)}.
\end{align*}
We turn next to the case when $\beta_1=\beta$, for which we have
\begin{align*}
&\Big(\cT(\partial^{\beta} f,\, W^\ell g,\, \sqrt\mu\,) ,\,\,
h\Big)_{L^2(\RR^6_{x, v})}=\Big(\Gamma(\partial^{\beta} f,\, W^\ell
g\,) ,\,\, h\Big)_{L^2(\RR^6_{x, v})}.
\end{align*}
Since we want to avoid using the non-isotropic norm of $f$ on the
right hand side, we
can not use the estimate (\ref{2.2.3}) to complete the proof. So we proceed in a different way, use firstly (\ref{1formula}) to get
\begin{align*}
&\Big(\Gamma(\partial^{\beta} f,\,
 W^\ell g) ,\,\, h\Big)_{L^2(\RR^6_{x, v})}=\Big(Q(\sqrt{\mu}\,
 \partial^{\beta}  f,\,  W^\ell g) ,\,\,
 h\Big)_{L^2(\RR^6_{x, v})}\\
 &+ \int\iiint b(\cos\theta)
\Big(\sqrt{\mu_*}\, -\sqrt{\mu'_*}\,\,\Big)(\partial^{\beta} f)'_*
(W^\ell\,g)' h d v_* d \sigma dvdx\,.
\end{align*}
On one hand, (\ref{2.1.2}) with $m=0, \alpha=-s$, implies that
\begin{align*}
 \left|\Big(Q(\sqrt{\mu}\,
 \partial^{\beta}  f,\,  W^\ell g) ,\,\,
 h\Big)_{L^2(\RR^6_{x, v})}\right|&\leq C\|h\|_{L^2_s(\RR^6)}\| \sqrt{\mu}\,
 \partial^{\beta} f\|_{L^2(\RR^3_x;\, L^1_{2s}(\RR^3_v))}
 \|W^\ell g\|_{L^\infty(\RR^3_x;\, H^{2s}_s(\RR^3_v))}\\
& \leq C \| f\|_{H^{|\beta|}(\RR^6)} ||W^\ell
g|||_{H^{3/2+2s+\epsilon}_s(\RR^6)} |||h|||_{\cB^0_0(\RR^6)}.
 \end{align*}
On the other hand, we can write
\begin{align*}
&\int\iiint b(\cos\theta) \Big(\sqrt{\mu_*}\,
-\sqrt{\mu'_*}\,\,\Big)(\partial^{\beta} f)'_*
(W^\ell\,g)' h d v_* d \sigma dvdx\,\\
&=\int\iiint b(\cos\theta) \Big(\sqrt{\mu_*}\,
-\sqrt{\mu'_*}\,\,\Big)(\partial^{\beta}f)'_*
(W^\ell\,g)' \Big(h-h'\Big) d v_* d \sigma dvdx\\
& +\int\iiint b(\cos\theta) \Big(\sqrt{\mu_*}\,
-\sqrt{\mu'_*}\,\,\Big)(\partial^{\beta} f)'_*
(W^\ell\,g)' h' d v_* d \sigma dvdx\\
&= D_1 + D_2\,.
\end{align*}
By the Cauchy-Schwarz inequality, one has
\begin{align*}
|D_1| &\leq \left(\int\iiint b(\cos\theta) | (\partial^\beta
f)'_\ast |^2 |(W^\ell g )'|^2 \Big((\mu_*)^{1/4} -
(\mu'_*)^{1/4}\Big)^2 d v_* d
\sigma dvdx\right)^{1/2}\\
&\qquad\times \left(\int\iiint b(\cos\theta) \Big( \mu^{1/4}_\ast +
(\mu'_*)^{1/4} \Big)^2 (h-h')^2 d v_* d \sigma dvdx\right)^{1/2}\, .
\end{align*}
Lemma \ref{lemm2.2.2} yields
\begin{align*}
&\int\iiint b(\cos\theta) | (\partial^\beta f)'_\ast |^2 |(W^\ell g
)'|^2 \Big((\mu_*)^{1/4} - (\mu'_*)^{1/4}\Big)^2 d v_* d \sigma
dvdx\\
&\leq C \int_{\RR^3_x}\int_{\RR^6_{v, v_*}} | (\partial^\beta
f)_\ast |^2 |(W^\ell g )|^2 \langle v\rangle^{2s}\langle
v_*\rangle^{2s} d v_* d dvdx
\\
&\leq C \int_{\RR^3_x}||\partial^\beta f||^2_{L^2_s(\RR^3_v)}
||W^\ell g ||^2_{L^2_s(\RR^3_v)} dx \leq C ||\partial^\beta
f||^2_{L^2_s(\RR^6)}
||W^\ell g ||^2_{L^\infty(\RR^3_x;\, L^2_s(\RR^3_v))}\\
&\leq C ||f||^2_{H^N_s(\RR^6)} ||W^\ell \Lambda^{3/2+\epsilon}_x g
||^2_{L^2_s(\RR^6)}\leq C ||f||^2_{H^{|\beta|}_s(\RR^6)} ||| g
|||^2_{\cB^2_\ell(\RR^6)}\, ,
\end{align*}
while {}from Lemma \ref{lemm2.3.1}, we get
\begin{align*}
& \int\iiint b(\cos\theta) \Big( \mu^{1/4}_\ast + (\mu'_*)^{1/4}
\Big)^2 (h-h')^2 d v_* d \sigma dvdx\\
&\leq 4 \int\iiint b(\cos\theta)  \mu^{1/2}_\ast (h-h')^2 d v_* d
\sigma dvdx\leq C|||h|||^2_{\cB^0_0(\RR^6)}.
\end{align*}
Therefore, we obtain
$$
|D_1|\leq C ||f||_{H^{|\beta|}_s(\RR^6)} ||| g
|||_{\cB^2_\ell(\RR^6)}|||h|||_{\cB^0_0(\RR^6)}.
$$
For the term $D_2$, we have
\begin{align*}
&\left|\int\iiint b(\cos\theta) \Big(\sqrt{\mu_*}\,
-\sqrt{\mu'_*}\,\,\Big)(\partial^{\beta} f)'_* (W^\ell\,g)' h' d v_*
d \sigma dvdx\right|\\
&=\left|\int\iiint b(\cos\theta) \Big(\sqrt{\mu'_*}\,
-\sqrt{\mu_*}\,\,\Big)(\partial^{\beta} f)_* (W^\ell\,g) h d v_* d
\sigma dvdx\right|\\
&\leq C \int_{\RR^3_x}\int_{\RR^6_{v, v_*}}\Big|(\partial^{\beta}
f)_*\Big|\,\, \Big|W^\ell\,g\Big|\,\, |h| \langle
v\rangle^{2s}\langle v_*\rangle^{2s} d v_* d  dvdx\\
&\leq C \int_{\RR^3_x}\|\partial^{\beta} f\|_{L^1_{2s}(\RR^3_v)}
\|W^\ell\,g\|_{L^2_{s}(\RR^3_v)} \|h\|_{L^2_{s}(\RR^3_v)} dx\\
&\leq C \|\partial^{\beta} f\|_{L^2(\RR^3_x;\,
L^2_{3/2+2s+\epsilon}(\RR^3_v))} \|W^\ell\,g\|_{L^\infty(\RR^3_x;\,
L^2_{s}(\RR^3_v))} \|h\|_{L^2_{s}(\RR^6)},
\end{align*}
so that
$$
|D_2|\leq C ||f||_{H^{|\beta|}_{3/2+2s+\epsilon}(\RR^6)} ||| g
|||_{\cB^2_\ell(\RR^6)}|||h|||_{\cB^0_0(\RR^6)}.
$$
Therefore, it follows that
\begin{align}\label{3.2.6+4}
\left|\Big(\Gamma(\partial^{\beta} f,\,
 W^\ell g) ,\,\, h\Big)_{L^2(\RR^6_{x, v})}\right|\leq C
 ||f||_{H^{|\beta|}_{3/2+2s+\epsilon}(\RR^6)} ||| g
|||_{\cB^2_\ell(\RR^6)}|||h|||_{\cB^0_0(\RR^6)}.
\end{align}

Finally, for the case $|\beta_1|\geq 2$, since $3/2+2s<3\leq N$, we have
also
\begin{align*}\label{3.2.6+3}
\left|\Big(\cT(\partial^{\beta_1} f,\, W^\ell\partial^{\beta_2} g,\,
\mu_{\beta_3}\,) ,\,\, h\Big)_{L^2(\RR^6_{x, v})}\right|
 \leq C \| f\|_{H^{N}_\ell(\RR^6)}
|||g|||_{\cB^{N}_\ell(\RR^6)} |||h|||_{\cB^0_0(\RR^6)}\,.
\end{align*}
The proof of the proposition is then completed.
\end{proof}

By using the argument in the  proof of the above proposition,
 the following proposition follows {}from the Sobolev
 imbedding theorems.

\begin{prop}\label{cor3.2.3}
For any $\ell\geq 3$, we have, for all $\beta\in\NN^6,\, |\beta|\leq
2$,
\begin{equation}\label{3.2.5+1+1}
\left|\Big(W^\ell \,\partial^\beta_{x, v} \Gamma(f,\, g\,),\,\,
h\Big)_{L^2(\RR^6_{x, v})}\right|\leq C ||f||_{H^3_\ell(\RR^6)}\,\,
||| g|||_{\cB^{3}_\ell(\RR^6)}\,\, |||
h|||_{\cB^0_0(\RR^6)}.
\end{equation}
\end{prop}

Finally, the linear operators can be also estimated as follows.

\begin{prop}\label{prop3.2.4}
For $\ell\geq 3$, we have for any $\beta\in\NN^6$,
\begin{equation}\label{3.2.7}
\left|\Big(W^\ell \,\partial^\beta_{x, v} \cL_2(f),\,\,
h\Big)_{L^2(\RR^6)}\right|\leq C_{|\beta|, \ell}
||f||_{H^{|\beta|}_\ell(\RR^6)}\,\, ||| h|||_{\cB^0_0(\RR^6)}.
\end{equation}
If~ $|\beta|\ge 1$, we have
\begin{align}\label{3.2.8}
&\left|\Big(\cL_1(W^\ell \,\partial^\beta_{x, v} g)-W^\ell
\,\partial^\beta_{x, v} \cL_1(g),\,\,
h\Big)_{L^2(\RR^6)}\right|\\
&\qquad\leq C_{|\beta|, \ell} \Big( ||
g||_{H^{|\beta|}_\ell(\RR^6)}+ |||
g|||_{\cB^{|\beta\,|-1}_\ell(\RR^6)}\, \Big)\,\, |||
h|||_{\cB^0_0(\RR^6)},\nonumber
\end{align}
and for $|\beta|=0$,
\begin{equation}\label{3.2.8+1}
\left|\Big(\cL_1(W^\ell \, g)-W^\ell \, \cL_1(g),\,\,
h\Big)_{L^2(\RR^6)}\right|\leq C || g||_{L^{2}_\ell(\RR^6)}\,\, |||
h|||_{\cB^0_0(\RR^6)}.
\end{equation}
\end{prop}
\begin{rema} On the right hand side of (\ref{3.2.5+1+1}), the term $||| g|||_{\cB^{3}_\ell(\RR^6)}$
comes {}from the Sobolev imbedding
$$
L^\infty(\RR^3_x;\, H^{2s}_s(\RR^3_v))\supset
H^{3/2+2s+\epsilon}_s(\RR^6)\supset\cB^3_0(\RR^6),
$$
where $\epsilon$ is any small positive number. Thus the order of differentiation is equal to $3$. Note that
this is due to the nonlinearity in the operator $\Gamma(\cdot,\cdot)$. For the linear operators, the estimates given in
(\ref{3.2.8}) and (\ref{3.2.8+1}) do not involve this term.
\end{rema}

\begin{proof} For the proof of (\ref{3.2.7}), by using the
Leibniz formula (\ref{3.1.2}), we have
\begin{align*}
& -\Big(W^\ell \,\partial^\beta_{x, v} \cL_2(f),\,\,
h\Big)_{L^2(\RR^6_{x, v})}=\Big(W^\ell \,\partial^\beta_{x, v}
\Gamma(f,\, \sqrt\mu\,),\,\,
h\Big)_{L^2(\RR^6_{x, v})}\\
&= \sum C^{\beta_1}_{\beta_2, \beta_3} \Big(\cT(\partial^{\beta_1}
f,\, W^\ell \partial^{\beta_2} \sqrt\mu\,,\,
\mu_{\beta_3}\,) ,\,\, h\Big)_{L^2(\RR^6_{x, v})}\\
&\quad+\sum C^{\beta_1}_{\beta_2, \beta_3}
\Big(W^\ell\cT(\partial^{\beta_1} f\,,\, \partial^{\beta_2}
\sqrt\mu\,,\, \mu_{\beta_3}\,) - \cT(\partial^{\beta_1} f\,,\,
W^\ell
\partial^{\beta_2} \sqrt\mu\,,\, \mu_{\beta_3}\,) ,\,\,
h\Big)_{L^2(\RR^6_{x, v})}\\
&=E_1+E_2.
\end{align*}
Then (\ref{3.2.6+1}) implies
\begin{align*}
|E_2|&=\left|\sum C^{\beta_1}_{\beta_2, \beta_3}
\Big(W^\ell\cT(\partial^{\beta_1} f\,,\, \partial^{\beta_2}
\sqrt\mu\,,\, \mu_{\beta_3}\,) - \cT(\partial^{\beta_1} f\,,\,
W^\ell
\partial^{\beta_2} \sqrt\mu\,,\, \mu_{\beta_3}\,) ,\,\,
h\Big)_{L^2(\RR^6_{x, v})}\right|\\
&\leq C\|f\|_{H^{|\beta|}_\ell(\RR^6)}
\|\sqrt\mu\,\|_{H^{|\beta|}_\ell(\RR^3)} ||| h|||_{\cB^0_0(\RR^6)}
\leq C\|f\|_{H^{|\beta|}_\ell(\RR^6)} ||| h|||_{\cB^0_0(\RR^6)},
\end{align*}
and (\ref{3.2.4}) implies also,
\begin{align*}
|E_1|&=\left|\sum C^{\beta_1}_{\beta_2, \beta_3}
\Big(\cT(\partial^{\beta_1} f,\, W^\ell \partial^{\beta_2}
\sqrt\mu\,,\, \mu_{\beta_3}\,) ,\,\, h\Big)_{L^2(\RR^6_{x,
v})}\right|\\
&\leq C \|f\|_{H^{|\beta|}_\ell(\RR^6)} ||| h|||_{\cB^0_0(\RR^6)},
\end{align*}
where for the case when $\beta_1=\beta$, we have used
(\ref{3.2.6+4}).

For (\ref{3.2.8}), since $-\cL_1(g)=\Gamma(\sqrt\mu,\, g)$, by using
again the Leibniz formula (\ref{3.1.2}), we have
\begin{align*}
&-\Big(W^\ell \,\partial^\beta_{x, v}\,\cL_1( g)-
\cL_1(W^\ell \,\partial^\beta_{x, v}\,g),\,\, h\Big)_{L^2(\RR^6)}\\
&=\Big(W^\ell \,\partial^\beta_{x, v} \Gamma(\sqrt\mu\,,\,
,g)-\Gamma(\sqrt\mu\,,\,\,W^\ell \,\partial^\beta_{x, v},\, g),\,\,
h\Big)_{L^2(\RR^6)}\\
&= \sum_{|\beta_2|\leq |\beta|-1} C^{\beta_1}_{\beta_2, \beta_3}
\Big(\cT(\partial^{\beta_1} \sqrt\mu\,,\, W^\ell \partial^{\beta_2}
g,\,
\mu_{\beta_3}\,) ,\,\, h\Big)_{L^2(\RR^6_{x, v})}\\
&\qquad+\sum C^{\beta_1}_{\beta_2, \beta_3}
\Big(W^\ell\cT(\partial^{\beta_1} \sqrt\mu\,,\, \partial^{\beta_2}
g,\, \mu_{\beta_3}\,) - \cT(\partial^{\beta_1} \sqrt\mu\,,\, W^\ell
\partial^{\beta_2} g,\, \mu_{\beta_3}\,) ,\,\,
h\Big)_{L^2(\RR^6_{x, v})}\\
&=F_1+F_2.
\end{align*}
Then (\ref{3.2.6+1}) implies
\begin{align*}
|F_2|&= \left|\sum C^{\beta_1}_{\beta_2,
\beta_3}\Big(W^\ell\cT(\partial^{\beta_1} \sqrt\mu\,,\,
\partial^{\beta_2} g,\, \mu_{\beta_3}\,)
- \cT(\partial^{\beta_1}\sqrt\mu\,,\, W^\ell
\partial^{\beta_2} g,\, \mu_{\beta_3}\,)
,\,\, h\Big)_{L^2(\RR^6_{x, v})}\right|\\
&\leq C
\|\sqrt\mu\,\|_{H^{|\beta|}_\ell(\RR^3)}\|g\|_{H^{|\beta\,|}_\ell(\RR^6)}
|||h|||_{\cB^0_0(\RR^6_{x, v})} \leq C
\|g\|_{H^{|\beta\,|}_\ell(\RR^6)}|||h|||_{\cB^0_0(\RR^6_{x, v})}\, ,
\end{align*}
which also gives (\ref{3.2.8+1}) .

On the other hand, for $F_1$, (\ref{3.2.4}) implies that, when $ |\beta_2|\leq
|\beta|- 1$,
\begin{align*}
& \left|\Big(\cT(\partial^{\beta_1} \sqrt\mu\,,\, W^\ell
\partial^{\beta_2} g,\, \mu_{\beta_3}\,)
,\,\, h\Big)_{L^2(\RR^6_{x, v})}\right|\\
&\leq C \left(\int_{\RR^3_x} \|\partial^{\beta_1}
\sqrt\mu\,\|^2_{L^2_s(\RR^3_v)} \Big(|||W^\ell
 \partial^{\beta_2} g|||^2+||W^\ell
 \partial^{\beta_2} g||^2_{H^s(\RR^3_v)}\Big)dx\right.\\
&+ \left.\int_{\RR^3_x}\|W^\ell \partial^{\beta_2}
g\|^2_{L^2_s(\RR^3_v)} \Big(|||
\partial^{\beta_1} \sqrt\mu\,|||^2+||
\partial^{\beta_1} \sqrt\mu\,||^2_{H^s(\RR^3_v)}\Big)
dx\right)^{1/2}|||h|||_{\cB^0_0(\RR^6)} \\
&\leq C \| \sqrt\mu\,\|_{H^{|\beta_1|+s}_s(\RR^3_v)}
|||g|||_{\cB^{|\beta_2|}_\ell(\RR^6)} |||h|||_{\cB^0_0(\RR^6)}\\
&\leq C |||g|||_{\cB^{|\beta|-1}_\ell(\RR^6)}
|||h|||_{\cB^0_0(\RR^6)}.
\end{align*}
Then the proof of the proposition is completed.
\end{proof}

%%%%%%%%%%%%%%%%%%%%%%%%%%%%%%%%%%%%%%%%%%%%%%%%%%%%%%%%%%%%%%%%%%%%%%
%%%%%%%%%%%%%%%%%%%%%%%%%%%%%%%%%%%%%%%%%%%%%%%%%%%%%%%%%%%%%%%%%%%%%%
\section{Local existence}\label{section4}
\smallbreak

%%%%%%%%%%%%%%%%%%%%%%%%%%%%%%%%%%%%%%%%%%%%%%%%%%%%%%%%%%%%%%%%%%%%%%%
\subsection{Energy estimates for a linear equation}\label{section4.0}
\setcounter{equation}{0} \smallbreak We now consider the following
Cauchy problem for a linear Boltzmann equation with a given
function $f$,
\begin{equation}\label{4.0.1}
\partial_t g + v\,\cdot\,\nabla_x g + \cL_1 g = \Gamma (f,\,g) -\cL_2 f
,\qquad g|_{t=0} = g_0\,,
\end{equation}
which is equivalent to the problem:
\begin{equation*}%\label{4.0.2}
\partial_t G + v\,\cdot\,\nabla_x G=
Q(F,\,G) ,\qquad G|_{t=0} = G_0,
\end{equation*}
with $F=\mu+\sqrt\mu\,f$ and $G=\mu+\sqrt\mu\,g$.

We shall now study the
energy estimates on (\ref{4.0.1}) in the function space
$H^N_\ell$.
For $N\geq 3, \ell\geq 3$  and $\beta\in\NN^6,
|\beta|\leq N$, taking
$$
\varphi(t, x, v)= (-1)^{|\beta|}\Big(\partial^\beta_{x, v}
W^{2\ell}\partial^\beta_{x, v}g\Big)(t, x, v),
$$
as a test function on $\RR^3_x\times\RR^3_v$, we get
\begin{align*}
&\frac 12 \frac{d}{d t}\|\partial^\beta\,g\|^2_{L^2(\RR^6)}+
\Big(W^\ell\Big[\partial^\beta_{x, v},\, v\,\Big]\,\cdot\,\nabla_x
g,\, W^\ell\partial^\beta_{x, v}g\Big)_{L^2(\RR^6)}+
\Big(W^\ell\partial^\beta_{x, v} \cL_1(g),\,
W^\ell\partial^\beta_{x, v}g\Big)_{L^2(\RR^6)}\\
&=\Big(W^\ell\partial^\beta_{x, v} \Gamma(f, \, g),\,
W^\ell\partial^\beta_{x, v}g\Big)_{L^2(\RR^6)}-
\Big(W^\ell\partial^\beta_{x, v} \cL_2(f),\,
W^\ell\partial^\beta_{x, v}g\Big)_{L^2(\RR^6)},
\end{align*}
where we have used the fact that
$$
\left( v\,\cdot\,\nabla_x \Big(W^\ell\partial^\beta_{x, v}\,
g\Big),\, W^\ell\partial^\beta_{x, v}g\right)_{L^2(\RR^6)}=0\, .
$$
Applying now Propositions \ref{prop3.2.3} and
\ref{prop3.2.4}, we get for any  $3\leq k\leq N$ and
$|\beta|\leq k$,
\begin{align*}
&\frac 12 \frac{d}{d t}\|\partial^\beta\,g\|^2_{L^2_\ell(\RR^6)}+
\Big(\cL_1\Big(W^\ell\partial^\beta_{x, v} g\Big),\,
W^\ell\partial^\beta_{x, v}g\Big)_{L^2(\RR^6)}\\
&\leq C\Big\{ \|f\|_{H^k_\ell(\RR^6)} \,
|||g|||^2_{\cB^{k}_\ell(\RR^6)}+ \|g\|^2_{H^{k}_\ell(\RR^6)}+
\|f\|_{H^k_\ell(\RR^6)} \,
|||g|||_{\cB^{k}_\ell(\RR^6)}\\
&\qquad+  \Big(\|g\|_{H^{k}_\ell(\RR^6)}
+|||g|||_{\cB^{k-1}_\ell(\RR^6)}
 \,\Big)|||g|||_{\cB^{k}_\ell(\RR^6)}\Big\}.
\end{align*}
By taking summation over $|\beta|\leq k$, Lemma \ref{lemm2.1.1} together with
(\ref{3.1.0+1})
 and the Cauchy-Schwarz inequality imply that
\begin{align}\label{4.0.3}
&\frac{d}{d t}\|g\|^2_{H^k_\ell(\RR^6)}+\frac{C_0}{2}
|||g|||^2_{\cB^{k}_\ell(\RR^6)}\leq C_{k, \ell}
\|f\|_{H^k_\ell(\RR^6)} \,
|||g|||^2_{\cB^{k}_\ell(\RR^6)}\\
&\qquad\qquad+C_{k, \ell}\left( \|g\|^2_{H^{k}_\ell(\RR^6)}+
\|f\|^2_{H^k_\ell(\RR^6)} \, +|||g|||^2_{\cB^{k-1}_\ell(\RR^6)}
 \right)\,, \qquad 3\leq k\leq N\, .\nonumber
\end{align}
For $k=1, 2$, Proposition \ref{cor3.2.3} is used to get
\begin{align}\label{4.0.3+1}
&\frac{d}{d t}\|g\|^2_{H^k_\ell(\RR^6)}+\frac{C_0}{2}
|||g|||^2_{\cB^{k}_\ell(\RR^6)}\leq C_{k, \ell}
\|f\|_{H^3_\ell(\RR^6)} \,
|||g|||^2_{\cB^{3}_\ell(\RR^6)}\\
&\qquad\qquad+C_{k, \ell}\left( \|g\|^2_{H^{k}_\ell(\RR^6)}+
\|f\|^2_{H^3_\ell(\RR^6)} \, +|||g|||^2_{\cB^{k-1}_\ell(\RR^6)}
 \right)\,,\nonumber
\end{align}
while for $k=0$
\begin{align}\label{4.0.4}
&\frac{d}{d t}\|g\|^2_{L^2_\ell(\RR^6)}+\frac{C_0}{2}
|||g|||^2_{\cB^{0}_\ell(\RR^6)}\leq C_{0, \ell}
\|f\|_{H^3_\ell(\RR^6)} \,
|||g|||^2_{\cB^{0}_\ell(\RR^6)}\\
&\qquad\qquad\qquad+C_{0, \ell}\left( \|g\|^2_{L^{2}_\ell(\RR^6)}+
\|f\|^2_{H^3_\ell(\RR^6)} \,
 \right)\,, \nonumber
\end{align}
where $C_0$ is the constant in (\ref{3.1.0+1}), which is
independent on $k, \ell$ and $N$.

Take $N\ge 3$, when $k\ge 2$,  by taking a linear combination of
  (\ref{4.0.3}) and  (\ref{4.0.3+1}), we have
\begin{align*}
&\frac{d}{d t}\left(\|g\|^2_{H^{k-1}_\ell(\RR^6)}+\frac{C_0}{2 C_{k,
\ell}}\|g\|^2_{H^k_\ell(\RR^6)}\right)+\frac{C^2_0}{2^2 C_{k, \ell}}
|||g|||^2_{\cB^{k}_\ell(\RR^6)}\\
&\leq \frac{C_0}{2}\left( \|f\|_{H^N_\ell(\RR^6)} \,
|||g|||^2_{\cB^{N}_\ell(\RR^6)}+
\|g\|^2_{H^{N}_\ell(\RR^6)}+ \|f\|^2_{H^N_\ell(\RR^6)}\right)\\
&\qquad\qquad+ \frac{d}{dt}\|g\|^2_{H^{k-1}_\ell(\RR^6)}+\frac{C_0}{2}|||g|||^2_{\cB^{k-1}_\ell(\RR^6)}
\\
&\leq \frac{C_0}{2}\left( \|f\|_{H^N_\ell(\RR^6)} \,
|||g|||^2_{\cB^{N}_\ell(\RR^6)}+ \|g\|^2_{H^{N}_\ell(\RR^6)}+
\|f\|^2_{H^N_\ell(\RR^6)}\right)\\
&\qquad+C_{k-1, \ell}\left( \|f\|_{H^N_\ell(\RR^6)} \,
|||g|||^2_{\cB^{N}_\ell(\RR^6)} +\|g\|^2_{H^{N}_\ell(\RR^6)}+
\|f\|^2_{H^N_\ell(\RR^6)} \, +|||g|||^2_{\cB^{k-2}_\ell(\RR^6)}
 \right)
\, .
\end{align*}
By induction and by using (\ref{4.0.4}), we have the
following  estimate
\begin{align}\label{4.0.5}
&\frac{d}{d t}\left(\sum_{0\leq k\leq N}c_{k,
l}\,\,\|g\|^2_{H^k_\ell(\RR^6)}\right)+\widetilde{C_0}\,\,
|||g|||^2_{\cB^{N}_\ell(\RR^6)}\\
&\qquad\leq \widetilde C_{N, \ell} \left(\|f\|_{H^N_\ell(\RR^6)} \,
|||g|||^2_{\cB^{N}_\ell(\RR^6)}+\|g\|^2_{H^{N}_\ell(\RR^6)}+
\|f\|^2_{H^N_\ell(\RR^6)} \,
 \right)\,,\nonumber
\end{align}
for some positive constants $\widetilde{C_0}<C_0$, $c_{k, \ell}$ and $\widetilde C_{N, \ell}$. Notice that
\begin{equation}\label{4.0.6}
\|g\|^2_{H^N_\ell(\RR^6)}\sim \sum_{0\leq k\leq N}c_{k,
l}\|g\|^2_{H^k_\ell(\RR^6)}.
\end{equation}

We are now ready to prove the following theorem.

\begin{theo}\label{theo4.0.1}
Let $N\geq 3, \ell\geq 3$. Assume that $g_0\in H^N_\ell(\RR^6)$ and
$f\in L^\infty([0, T];\,H^N_\ell(\RR^6))$.  If $g\in L^\infty([0,
T];\,H^N_\ell(\RR^6))\bigcap L^2([0, T];\, \cB^N_\ell(\RR^6))$ is a
solution of Cauchy problem (\ref{4.0.1}), then there exists
$\epsilon_0>0$ such that if
\begin{equation*}\label{4.0.7}
\|f\|_{L^\infty([0, T];\,H^N_\ell(\RR^6))}\leq \epsilon_0,
\end{equation*}
we have
\begin{equation}\label{4.0.8}
\|g\|^2_{L^\infty([0, T];\,H^N_\ell(\RR^6))}+ ||g||^2_{L^2([0,
T];\,\cB^N_\ell(\RR^6))}\leq Ce^{C\,
T}(\|g_0\|^2_{H^N_\ell(\RR^6)}+\epsilon_0^2 T),
\end{equation}
for a constant $C>0$ depending only on $N, \ell$ .
\end{theo}

\begin{proof}
Choosing $\epsilon_0=\frac{\widetilde{C_0}}{2\widetilde C_{N, \ell}}$,
we have, {}from (\ref{4.0.5}),
\begin{align*}
&\frac{d}{d t}\left(\sum_{0\leq k\leq N}c_{k,
l}\,\,\|g\|^2_{H^k_\ell(\RR^6)}\right)+\frac{\widetilde{C_0}}{2}\,\,
|||g|||^2_{\cB^{N}_\ell(\RR^6)}\\
&\leq 2\widetilde C_{N, \ell} (\|g\|^2_{H^{N}_\ell(\RR^6)}+\epsilon_0^2)
 \leq C (\sum_{0\leq k\leq N}c_{k,
l}\,\,\|g\|^2_{H^k_\ell(\RR^6)}+\epsilon_0^2),
\end{align*}
and
\begin{align*}
\frac{d}{d t}\left(e^{-Ct
}\sum_{0\leq k\leq N}c_{k,
l}\,\,\|g\|^2_{H^k_\ell(\RR^6)}\right)+\frac{\widetilde{C_0}}{2}\,e^{-C\,t }\, |||g|||^2_{\cB^{N}_\ell(\RR^6)}\leq
C\epsilon_0^2e^{-Ct}\,.
\end{align*}
Thus we get (\ref{4.0.8}) for some constant $C>0$ and this completes the proof
of the theorem.
\end{proof}

%%%%%%%%%%%%%%%%%%%%%%%%%%%%%%%%%%%%%%%%%%%%%%%%%%%%%%%%%%%%%%%%%%%%%%%
\subsection{Existence for the linear equation}\label{section4.1}
\setcounter{equation}{0} \smallbreak

With the energy estimate given in the above subsection, we can now prove
the following local existence theorem by using the Hahn-Banach theorem.
\begin{theo}\label{theo4.1.1}
Let $\ell\geq 3, N\geq 3$ and $g_0\in H^{N}_\ell(\RR^6)$. There
exists $\epsilon_0>0$ such that if
$$
\|f\|_{L^\infty([0, T];\, H^{N}_\ell(\RR^6))}\leq\epsilon_0,
$$
then the Cauchy problem (\ref{4.0.1}) admits a unique solution
$$
g\in L^\infty([0, T];\, H^{N}_\ell(\RR^6))\cap L^2([0, T];\,
\cB^{N}_\ell(\RR^6)).
$$
\end{theo}

\begin{proof}
We consider the following Cauchy problem :
\begin{equation}\label{4.1.1}
{\mathcal P} g \equiv \partial_t g + v\,\cdot\,\nabla_x g + \cL_1 g
- \Gamma (f,\, g) = H , \ g(0) =g_0.
\end{equation}
For $h\in C^\infty ([0,T];\, {\mathcal S}(\RR^6_{x,v}) )$ with $h(T)
=0$, we define
$$
\Big(g,\,\, {\mathcal P}^*_{N,\,\ell} \, h\Big)_{L^2([0,\,T];
H^N_\ell(\RR^6))} = \Big({\mathcal P}\,g,\, h\Big)_{L^2([0,\,T];
H^N_\ell(\RR^6))}\,\, ,
$$
so that ${\mathcal P}^*_{N,\,\ell} $ is the adjoint of the linear operator
${\mathcal P}$ in the Hilbert space $L^2 ([0,\,T]; H^N_\ell(\RR^6)
)$.

Set
$$
{\mathbb W} = \left\{ w = {\mathcal P}^*_{N,\,\ell}\, h ;\,\,h\in
C^\infty ([0,T];\, {\mathcal S}(\RR^6_{x,v}) )\,\,\mbox{with}\,\,
h(T) =0 \right\},
$$
which  is a dense subspace of $L^2 ([0,\,T];
\,H^N_\ell(\RR^6) )$. And we also have
$$
{\mathcal P}^\ast_{N,\,\ell} (h) = -\partial_t  h +
(v\,\cdot\,\nabla_x )^\ast h + \cL^\ast_1\, h + \Gamma^\ast (f,\,
h)\,.
$$
Then
\begin{align*}
\Big(h,\,\, {\mathcal P}^*_{N,\,\ell} \, h\Big)_{H^N_\ell(\RR^6)}
&=-\frac 12 \frac{d}{dt}|| h (t) ||^2_{H^N_\ell(\RR^6)} +
\Big(v\,\cdot\,\nabla_x\, h ,\, h\Big)_{
H^N_\ell(\RR^6)}\\
& + \Big(\cL_1  (h) , h\Big)_{H^N_\ell(\RR^6)}- \Big(\Gamma (f,\, h)
,\, h\Big)_{H^N_\ell(\RR^6)}.
\end{align*}
Same as Theorem \ref{theo4.0.1}, for
$|| f||_{L^\infty ([0, T];\, H^N_\ell(\RR^6))} \leq \epsilon_0$, we
have
\begin{align*}
\int^T_t e^{2C (s-t)}\Big|\Big(h,\,\, {\mathcal P}^*_{N,\,\ell} \,
h\Big)_{H^N_\ell(\RR^6)}\Big| dt &\geq || h(t)
||^2_{H^N_\ell(\RR^6)} + \int^T_t C e^{2C (s-t)} ||| h(s)
|||^2_{\cB^N_\ell(\RR^6)} ds\, .
\end{align*}
Thus, for all $0<t<T$,
\begin{align*}
&|| h(t)||^2_{H^N_\ell(\RR^6)}+C|| h ||^2_{L^2([t,
T];\,\,\cB^N_\ell(\RR^6))} \leq C\Big(h,\,\, {\mathcal
P}^*_{N,\,\ell} \,
h\Big)_{L^2([t,\,T];\, H^N_\ell(\RR^6))}\\
&\leq C || {\mathcal P}^*_{N,\,\ell} (h) ||_{L^2([t,\,T];
H^N_\ell(\RR^6))} || h||_{L^2([t,\,T]; H^N_\ell(\RR^6))}\,.
\end{align*}
Hence, we get
\begin{equation}\label{4.1.2+0}
|| h||_{L^2 ([0,\,T];\, H^N_\ell(\RR^6))}+|| h||_{L^\infty
([0,\,T];\, H^N_\ell(\RR^6))}  \leq CT ||  {\mathcal P}^*_{N,\,\ell}
(h) ||_{L^2 ([0,\,T];\, H^N_\ell(\RR^6))}\,.
\end{equation}
Since
$$
|| h||_{L^2([t,\,T]; H^N_\ell(\RR^6))}\leq C || h ||_{L^2([t,
T];\,\,\cB^N_\ell(\RR^6))},
$$
 we also have
\begin{equation}\label{4.1.2}
 || h ||_{L^2([0,
T];\,\,\cB^N_\ell(\RR^6))}\leq C ||  {\mathcal P}^*_{N,\,\ell} (h)
||_{L^2 ([0,\,T];\, H^N_\ell(\RR^6))}\,.
\end{equation}

Next, we define a functional ${\mathcal G}$ on ${\mathbb W}$ as
follows
$$
{\mathcal G}(w) = (H,\,h)_{L^2 ([0,\,T];\, H^N_\ell(\RR^6))} + ( g_0
, h(0))_{H^N_\ell(\RR^6)}.
$$
Then, if $H\in L^2 ([0,\,T];\, H^N_{\ell-s}(\RR^6))$,
(\ref{3.1.0+2})
gives
\begin{align*}
 |{\mathcal G}(w)| &\leq \|H\|_{L^2 ([0,\,T];\, H^N_{\ell-s}(\RR^6))}\|h\|_{L^2
([0,\,T];\, H^N_{\ell+s}(\RR^6))}
+\|g_0\|_{H^N_\ell(\RR^6)}\|h(0)\|_{H^N_\ell(\RR^6)}\\
&\leq C \|H\|_{L^2 ([0,\,T];\, H^N_{\ell-s}(\RR^6))}
|| h ||_{L^2([0, T];\,\,\cB^N_\ell(\RR^6))}
+\|g_0\|_{H^N_\ell(\RR^6)}\|h(0)\|_{H^N_\ell(\RR^6)}\\
&\leq C || {\mathcal P}^*_{N,\,\ell} (h) ||_{L^2 ([0,\,T];\,
H^N_\ell(\RR^6))}\leq C|| w||_{L^2 ([0,\,T];\, H^N_\ell(\RR^6))}\, ,
\end{align*}
where we have used (\ref{4.1.2+0})   and (\ref{4.1.2}).

Thus,
${\mathcal G}$ is a continuous linear functional on $\Big({\mathbb
W};\,\|\,\cdot\,\|_{L^2 ([0,\,T];\, H^N_\ell(\RR^6))}\Big)$. Now, there
exists $g\in L^2 ([0,\,T];\, H^N_\ell(\RR^6))$ such that for any
$w\in{\mathbb W}$,
$$
{\mathcal G}(w)= \big(g,\,w\big)_{L^2 ([0,\,T];\, H^N_\ell(\RR^6))},
$$ by Hahn-Banach Theorem. For any $h\in C^\infty ([0,T]; {\mathcal S}
(\RR^6_{x,v}) )$ with $h(T) =0$, we have
$$
\Big(g,\,\, {\mathcal P}^*_{N,\,\ell} \, h\Big)_{L^2([0,\,T];
H^N_\ell(\RR^6))} = \big(H,\, h\big)_{L^2([0,\,T]; H^N_\ell(\RR^6))}
+ \big(g_0 ,\, h(0)\big)_{H^N_\ell(\RR^6)},
$$
and by the definition of the operator ${\mathcal P}^*_{N,\,\ell}$, we
have also
$$
\Big({\mathcal P}\,g,\,\, \tilde h\Big)_{L^2([0,\,T]; L^2(\RR^6))} =
\big(H,\, \tilde h\big)_{L^2([0,\,T]; L^2(\RR^6))} + \big(g_0 ,\,
\tilde h(0)\big)_{L^2(\RR^6)},
$$
where
$$
\tilde h  = \Lambda^N W^{2\ell} \Lambda^N h \in C^\infty ([0, T];\,
{\mathcal S}(\RR^6) )\,\,\, \mbox{with}\,\,\, \tilde h (T )=0,
$$
where $\Lambda=(1-\Delta_{x,v})^{\frac 12}$.
Since $\Lambda^N W^{2\ell} \Lambda^N $ is an isomorphism on
$\big\{h:h\in C^\infty ([0, T];\, {\mathcal S}(\RR^6) )$
with $ h (T )=0\big\}$, we have shown that if $H\in L^2
([0,\,T];\!\! H^N_{\ell-s}(\RR^6))$, then $g\in L^2 ([0,\,T];\!\!
H^N_\ell(\RR^6))$ is a solution of the Cauchy problem (\ref{4.1.1}).

It remains to take
$$
H=-\cL_2(f)=\Gamma (f, \sqrt\mu),
$$
to get
$$
\left|(H,\,h)_{L^2 ([0,\,T];\, H^N_\ell(\RR^6))}\right|\leq C
 \|f\|_{L^2 ([0,\,T];\, H^N_{\ell}(\RR^6))}
|| h ||_{L^2([0, T];\,\,\cB^N_\ell(\RR^6))}.
$$
Then ${\mathcal G}\,$ is also continuous on ${\mathbb W}$.
 And this completes the proof of Theorem \ref{theo4.1.1}.
\end{proof}

%%%%%%%%%%%%%%%%%%%%%%%%%%%%%%%%%%%%%%%%%%%%%%%%%%%%%%%%%%%%%%%%%
\subsection{Convergence of approximate solutions}\label{section4.2}
\setcounter{equation}{0} \smallbreak

In this subsection, we prove the local existence theorem.
\begin{theo}\label{theo4.2.1}
Let $N\geq 3, \ell\geq 3$. There exist $\epsilon_1,\, T>0$ such that if $g_0\in H^N_\ell(\RR^6)$ and
$$
\|g_0\|_{H^N_\ell(\RR^6)}\leq \epsilon_1\, ,
$$
then the Cauchy problem \eqref{cauchy-problem} admits a solution
$$
g \in L^\infty([0, T];\, H^N_\ell(\RR^6))\cap L^2([0, T];\,
\cB^N_\ell(\RR^6)).
$$
\end{theo}

\begin{rema}\label{rema4.4}
{}By the  equation in \eqref{cauchy-problem}, we have, for
 $0<s<1/2$,
By using the equation \eqref{cauchy-problem}, we have, for $0<s<1/2$,
\[
\partial_t g, \ v\cdot\nabla_x g \in L^2([0,T]; H^{N-1}_{\ell-1}(\RR^6)).%, %\cL g, %\Gamma(g,g)\in L^2([0,T]; H^{N-1}_{\ell-1}(\RR^6)).
\]
Moreover, if we go back to the equation (1.1), we have that
\[
f=\mu+ \mu^{1/2}g\in H^N_\ell([0,T]\times\Omega\times \RR^3)),
\]
for any $\ell\in \NN$ and any bounded domain $\Omega\subset \RR^3_x$, and
thus the Sobolev embedding implies
that $f$ is
a classical solution of equation (1.1) if $N>7/2+1$. We will use this properties for the
smoothing effect of Theorem 1.1.

% $$
% \partial_t g\in L^\infty([0, T];\, H^{N-1}_{\ell-1}(\RR^6)),
% $$
% {\red because we have
% \begin{align*}
% &v\cdot \nabla_x g\in H^{N-1}_{\ell-1}(\RR^6),
% \\&
% \|\cL g\|
% \end{align*}
% }
% Thus
% $$
% g \in H^N_{\ell-1}([0, T]\times\RR^6)),
% $$
% which implies that $g$ is a classical solution.
\end{rema}

For the proof of Theorem \ref{theo4.2.1}, we consider the  sequence of approximate
solutions defined by the following Cauchy problem, $n\in\NN$,
$$
\partial_t f^{n+1} + v\,\cdot\,\nabla_x f^{n+1} =Q (f^n, f^{n+1})
\,,\qquad\qquad f^{n+1}|_{t=0}=f_0,
$$
where $f^n = \mu + \mu^{1/2} g^n$ and $f^0=f_0$. Note that it is also equivalent to
\begin{equation}\label{4.2.1}
\partial_t g^{n+1} + v\,\cdot\,\nabla_x g^{n+1} + \cL_1 g^{n+1}
- \Gamma (g^n , g^{n+1})=-\cL_2 g^n\,,\qquad g^{n+1}|_{t=0}=g_0.
\end{equation}

\begin{prop}\label{prop4.2.1}
Let $N\geq 3, \ell\geq 3$. There exist $\epsilon_1,\,\,T>0$
such that if $g_0\in H^N_\ell(\RR^6)$ and
\begin{equation*}\label{4.2.2}
\|g_0\|_{H^N_\ell(\RR^6)}\leq \epsilon_1\, ,
\end{equation*}
the Cauchy problem (\ref{4.2.1}) admits a sequence of solutions
$$
\{g^n, n\in\NN\}\subset L^\infty([0, T];\, H^N_\ell(\RR^6))\cap
L^2([0, T];\, \cB^N_\ell(\RR^6)).
$$
Moreover, for all $n\in\NN$,
\begin{equation}\label{4.2.3}
\|g^n\|_{L^\infty([0, T];\, H^N_\ell(\RR^6))}+\|g^n\|_{L^2([0, T];\,
\cB^N_\ell(\RR^6))}\leq \epsilon_0,
\end{equation}
where $\epsilon_0$ is the constant in
Theorem \ref{theo4.0.1}.
\end{prop}

\begin{proof}

(\ref{4.2.3}) will be proven by induction on $n$. Firstly, consider the equation
\begin{equation*}
\partial_t g^{1} + v\,\cdot\,\nabla_x g^{1} + \cL_1 g^{1}
- \Gamma (g_0 , g^{1})=-\cL_2 g_0\,,\qquad g^{1}|_{t=0}=g_0.
\end{equation*}
When $\epsilon_1<\epsilon_0$, the existence of $g^1$ is given by Theorem \ref{theo4.1.1} satisfying
$$
g^1\in L^\infty([0, T];\, H^N_\ell(\RR^6))\cap L^2([0, T];\,
\cB^N_\ell(\RR^6)).
$$
{}From Theorem \ref{theo4.0.1}, we can deduce
$$
\|g^1\|_{L^\infty([0, T];\, H^N_\ell(\RR^6))}+\|g^1\|_{L^2([0, T];\,
\cB^N_\ell(\RR^6))}\leq  Ce^{C \, T} \|g_0\|_{H^N_\ell(\RR^6)}.
$$
Thus (\ref{4.2.3}) holds when $\epsilon_1$
is chosen to be small compared to $\epsilon_0$.

For $n\ge 1$,
under the assumption  that
$$
\|g^n\|_{L^\infty([0, T];\, H^N_\ell(\RR^6))}+\|g^n\|_{L^2([0, T];\,
\cB^N_\ell(\RR^6))}\leq \epsilon_0,
$$
Theorem \ref{theo4.1.1} yields the existence of
$$
g^{n+1}\in L^\infty([0, T];\, H^N_\ell(\RR^6))\cap L^2([0, T];\,
\cB^N_\ell(\RR^6)).
$$
{}From Theorem \ref{theo4.0.1}, we can deduce
$$
\|g^{n+1}\|^2_{L^\infty([0, T];\,
H^N_\ell(\RR^6))}+\|g^{n+1}\|^2_{L^2([0, T];\, \cB^N_\ell(\RR^6))}\leq
Ce^{C\, T} (\|g_0\|^2_{H^N_\ell(\RR^6)} +\epsilon_0^2 T).
$$
and this gives
$$
\|g^{n+1}\|_{L^\infty([0, T];\,
H^N_\ell(\RR^6))}+\|g^{n+1}\|_{L^2([0, T];\, \cB^N_\ell(\RR^6))}\leq \epsilon_0,
$$
when $ T>0$ is sufficiently small,

Thus we  prove (\ref{4.2.3}) for all $n\in\NN$, and this completes
the proof of the proposition.
\end{proof}

It remains to prove the convergence. Set $w^n=g^{n+1}-g^n$ and
deduce {}from \eqref{4.2.1} that
\begin{equation*}%\label{4.2.1-unique}
\partial_t w^{n} + v\,\cdot\,\nabla_x w^{n} + \cL_1 w^{n}
- \Gamma (g^n , w^{n})=\Gamma(w^{n-1}, g^n)-\cL_2 w^{n-1}\,,\qquad w^{n}|_{t=0}=0.
\end{equation*}
Similar to the computation for \eqref{4.0.4}, we obtain
\begin{align*}\label{4.0.4-u}
&\frac{d}{d t}\|w^n\|^2_{L^2_\ell(\RR^6)}+\frac{C_0}{2}
|||w^n||^2_{\cB^{0}_\ell(\RR^6)}\leq C_{0, \ell}
\|g^n\|_{H^3_\ell(\RR^6)} \,
|||w^n|||^2_{\cB^{0}_\ell(\RR^6)}\\
&\qquad\qquad\qquad+C_{0,\ell}
\|w^{n-1}\|_{L^2_\ell(\RR^6)}(|||g^n|||_{\cB^3_\ell(\RR^3)}
+1)|||w^n|||_{\cB^0_\ell(\RR^6)}. \nonumber
\end{align*}
If $\epsilon_0$ is sufficiently small, this yields,
\begin{align*}
&\frac{d}{d t}\|w^n\|^2_{L^2_\ell(\RR^6)}+C_1
|||w^n|||^2_{\cB^{0}_\ell(\RR^6)}\leq C_2
\|w^{n-1}\|_{L^2_\ell(\RR^3)},
 \nonumber
\end{align*}
which, in turn, gives, if $T$ is sufficiently small,
\[
\|w^n\|_{L^\infty([0, T]; L^2_\ell(\RR^6))}\le \lambda\|w^{n-1}\|_{L^\infty([0, T]; L^2_\ell(\RR^6))},
\]
for some $\lambda\in(0,1).$
Thus we conclude that the sequence $\{g^n\}$ is a Cauchy sequence in
$L^\infty([0, T]; L^2_\ell(\RR^6))$. Let $g$ be the limit function.

By interpolation with the uniform estimates \eqref{4.2.3}, we see that
the sequence is strongly convergent in
\[
L^\infty([0, T];\, H^{N-\delta}_\ell(\RR^6))\cap
L^2([0, T];\, \cB^{N-\delta}_\ell(\RR^6))
\]
for any $\delta>0$.
Furthermore, by using equation
\eqref{4.2.1} and Proposition
\ref{prop4.2.1}, we see that $\{\p_tg^n\}$ is uniformly bounded in 
$L^\infty([0,T]; H^{N-1}_{\ell-1})$, so that  it is a
compact set in  the function space
\[
C^{1-\delta} (]0, T_*[\,;\,\,H^{N-1-2\delta}_{l-1}(\Omega
\times\RR^3_v))
\]
for any bounded domain $\Omega\subset\RR^3_x$.
Now we can take the limit in equation \eqref{4.2.1}
and thus $g$ is a solution of Cauchy problem (\ref{cauchy-problem}).

Finally,
by a standard weak compactness argument, we can extract a subsequence
of approximate solutions such that
\begin{align*}
&g^n\rightarrow g\in L^\infty([0, T];\, H^N_\ell(\RR^6))
\quad\text{weakly*},
\\&
g^n\rightarrow g\in
L^2([0, T];\, \cB^N_\ell(\RR^6)) \quad\text{weakly},
\end{align*}
which shows that
\[
g\in L^\infty([0, T];\, H^{N}_\ell(\RR^6))\cap
L^2([0, T];\, \cB^{N}_\ell(\RR^6)).
\]
Now the proof of Theorem \ref{theo4.2.1}
 is complete.

%%%%%%%%%%%%%%%%%%%%%%%%%%%%%%%%%%%%%%%%%%%%%%%%%%%%%%%%%%%%%%%%%%%%%%
%%%%%%%%%%%%%%%%%%%%%%%%%%%%%%%%%%%%%%%%%%%%%%%%%%%%%%%%%%%%%%%%%%%%%%
\section{Qualitative study on the solutions}\label{section5}
\smallbreak

In this section, we will prove two main qualitative properties of the
solutions to the problem considered in this paper, that is, the uniqueness
and non-negativity.
%%%%%%%%%%%%%%%%%%%%%%%%%%%%%%%%%%%%%%%%%%%%%%%%%%%%%%%%%%%%%%%%%%%%%%%
\subsection{Uniqueness}\label{section5.1}
\setcounter{equation}{0} \smallbreak
The uniqueness of solutions can be proved in a larger function space as stated in
Theorem \ref{theo2}. To obtain this theorem, we now first prove
 two preliminary results in  the following   lemmas.

Set $\varphi(v,x) = \la v, x \ra^2 = 1+|v|^2 +|x|^2$ and
\[
W_{\varphi,l} = \frac{\la v \ra^l}{\varphi(v,x)}= \frac{(1+|v|^2)^{l/2}}{1+|v|^2 + |x|^2} .
\]

\begin{lemm}\label{weight} For $ l \geq 4$, we have
\begin{align}\label{diff-sing}
|W_{\varphi,l} - W_{\varphi,l} '| &\leq C
\sin \left(\frac{\theta}{2} \right)
\left( \frac{W_l' + W_{l-3}'  W_{3,*}'}{\varphi(v_*',x)}
+ \sin^{l-3} \left(\frac{\theta}{2} \right)
W_{\varphi,l, *}'\right)  \notag \\
&\leq C \left(\theta W_l' W_{\varphi, 3,*}'   + \theta^{l-2} W_{\varphi,l, *}' \right ),
\end{align}
where $W_{\varphi,l,*}' =\displaystyle  \frac{W_{l,*}'}{\varphi(v_*',x)} $, and also for $l\ge 1$,
\begin{equation}\label{diff-reg}
|W_{\varphi,l} - W_{\varphi,l} '| \leq C
\sin \left(\frac{\theta}{2} \right)
 \frac{W_l' +  W_{l,*}'}{\varphi(v',x)} \leq C \theta \frac{W_l'  W_{l,*}'}{\varphi(v',x)}.
\end{equation}
\end{lemm}

\begin{proof} For $k \geq 0,  a \geq 0$, set
\[
F_k(\lambda)= \frac{\lambda^k}{\lambda + a}.
\]
Then for $\lambda \in [1,\infty[$, we have $\frac{d}{d\lambda}F_k(\lambda) \geq 0$ if $k \geq 1$
and $\frac{d^2}{d\lambda^2}F_k(\lambda) \geq 0$ if $k \geq 2$.
Since $\frac{d}{d\lambda}F_k(\lambda)$ is positive and increasing on
$[1,\infty[$ if $k \geq 2$, it follows {}from the mean
value theorem that for $\lambda, \lambda' \geq 1$
\[
|F_k(\lambda) -F_k(\lambda')| \leq \frac{d}{d\lambda}F_k (\lambda + |\lambda-\lambda'|) |\lambda-\lambda'|.
\]
Setting $\lambda = \la v \ra^2, \lambda' = \la v' \ra^2$, we have
\begin{align*}
|F_k(\la v \ra^2) -F_k(\la v' \ra^2)| &\leq
\frac{d}{d\lambda}F_k (2(\la v \ra^2 + |v-v'|^2))\big( 2|v|+ |v-v'| \big )|v-v'|\\
&\leq 2k F_{k-1/2}(2(\la v \ra^2 + |v-v'|^2))|v-v'|,
\end{align*}
because $ |\lambda-\lambda'| \leq 2|v-v'||v| + |v-v'|^2 \leq |v|^2 + 2 |v-v'|^2$
and $\sqrt \lambda \frac{d}{d\lambda}F_k(\lambda) \leq k F_{k-1/2}(\lambda)$.
Therefore, we obtain, choosing $a = |x|^2$
\begin{align}\label{important-w-ineq}
|W_{\varphi,l} - W_{\varphi,l} '| &\leq
C_l \left( \frac{\la v \ra ^{l-1}|v-v'|}
{\la v \ra^2 + |v-v'|^2 +a }
+ \frac{|v-v'|^l}
{\la v \ra^2 + |v-v'|^2 + a}\right) \notag \\
&\leq
C_l \enskip |v-v'| \la v \ra^{l-3} F_1(\la v \ra^2)
+ C_l \frac{|v-v'|^l}
{\la v \ra^2 + |v-v'|^2 + a}\\
&= I + II. \notag
\end{align}
Note that $\la v \ra^2 \leq 2\la v_* \ra^2 + 2|v-v_*|^2$.
Since $F_1$ is increasing, we have
\begin{align*}
I &\leq C_l \enskip |v-v'| \la v \ra^{l-3} \frac{\la v_* \ra^2
+ |v-v_*|^2}{\la v_* \ra^2 + |v-v_*|^2+|x|^2}\\
&\leq C_l \sin \left(\frac{\theta}{2} \right)
\frac{W_l + W_{l-3}W_{3,*}}{\varphi(v_*,x)}.
\end{align*}
On the other hand
\begin{align*}
II &\leq C_l \frac{|v-v_*|^l  \sin^l\left(\frac{\theta}{2} \right)}
{1 + \big (1-\sin^2\left(\frac{\theta}{2} \right) \big )|v|^2 +
\frac{1}{2}\sin^2\left(\frac{\theta}{2} \right)|v_*|^2 +|x|^2} \\
&\leq C_l \sin^{l-2}\left(\frac{\theta}{2} \right)\frac{W_l + W_{l,*}}{\varphi(v_*,x)}.
\end{align*}
Since $v$ and $v'$ are symmetric, we get the first conclusion.
The second one is a direct consequence of the first inequality of \eqref{important-w-ineq}.
\end{proof}

\begin{lemm}\label{moment-commutator-with-x}
Let $l \in \NN$. If \, $0<s <1/2$, there exists $C >0$ such
that
\begin{align}\label{usual}
&\left|\Big(\big(W_{\varphi,l}  \,\, Q(f,\,\,g)-Q(f,\,\,W_{\varphi,l}  \,\,
g)\big),\,\,\, h\Big)_{L^2(\RR^6)}\right| \\
&\leq C
\|f\|_{L^\infty(\RR_x^3; L^1_{l}(\RR^3_v))} \|W_{\varphi, l} g\|_{L^{2}
(\RR^6)} \|h\|_{L^{2}(\RR^6)}. \notag
\end{align}

Moreover, if $l \geq 5 $
then
\begin{align}\label{singular}
&\left|\Big(\big(W_{\varphi, l}  \,\, Q(f,\,\,g)-Q(f,\,\,W_{\varphi,l}  \,\,
g)\big),\,\,\, h\Big)_{L^2(\RR^6)}\right| \\
&\leq C
\|W_{\varphi, l}f\|_{L^2(\RR^6)}\|g\|_{L^\infty(\RR^3_x; L^2_{l}(\RR^3_v))}
\|h\|_{L^2(\RR^6)}. \notag
\end{align}
\end{lemm}

\begin{proof}
It follows {}from \eqref{diff-reg} that
\begin{eqnarray*}
&&\Big|\Big(\big(W_{\varphi,l}\,\,Q(f,\,\,g)-Q(f,\,\,W_{\varphi,l}\,\, g)\big),\,\,\,
h\Big)_{L^2(\RR^6)}\Big| \\
&=&\Big|\iiiint
b\,  \,f'_* g'(W'_{\varphi, l}-W_{\varphi,l})\, h\, dv dv_* d\sigma dx\Big|\\
& \leq& C \iiiint  b\, |\theta|\,\,
|(W_{l} f)'_*|\,\,|(W_{\varphi,l}g)'|\,\,|h|\,\,dvdv_*d\sigma dx\\
& =&C\iiiint  b\, |\theta| |(W_{l}f)_*|\,\,
|(W_{\varphi, l}g)|\,\,|h'|\, dvdv_*d\sigma dx\\
& \leq& C \Big(\iiiint  b\, |\theta|\,\, |(W_{l}\,
 f)_*|\,\,|(W_{\varphi, l}\, g)|^2 dv dv_*
d\sigma dx \Big)^{1/2}\\
&&\times \Big(\iiiint  b\, |\theta|\,\, |(W_{l}\, f)_*|\,
|h'|^2\, dvdv_*d\sigma dx\Big)^{1/2}\\
 &=&C J_1 \times J_2.
\end{eqnarray*}
Clearly, one has
\[
J^2_1\leq C \|f\|_{L^\infty(\RR^3_x; L^1_{l}(\RR^3_v))}
\|W_{\varphi, l} g\|^2_{L^2(\RR^6)}\int_{\SS^2}b(\cos\theta)\,|\theta|\, d\sigma \leq C
\|f\|_{L^\infty(\RR^3_x; L^1_{l}(\RR^3_v))}
\|W_{\varphi, l} g\|^2_{L^2(\RR^6)}.
\]
Next, by the regular change of variables $v\to v'$, cf.
\cite{al-1, al-3}, we have
\[
J^2_2=\iiint   D_0(v_*,v')|(W_{l}f)_*||h'|^2dv_*dv'dx,
\]
where
$$
D_0(v,v')=2\int_{\SS^2} \frac{\theta(v_*, v',
\sigma)}{\cos^2(\theta(v_*, v', \sigma)/2)} b(\cos\theta(v_*,
v',\sigma))d\sigma \leq C\int_0^{\pi/4}\psi^{-1-2s}\sin\psi\,\,
d\psi,
$$
and
$$
\cos\psi=\frac{v'-v_*}{|v'-v_*|}\cdot \sigma, \quad \psi=\theta/2,
\qquad d\sigma=\sin\psi d\psi d\phi.
$$
Thus,
\[
J^2_2\leq C \|f\|_{L^\infty(\RR_x^3;L^1_{l}(\RR^3_v))}\|h\|^2_{L^2(\RR^6)},
\]
and this together with the estimate on
$J_1$ give (\ref{usual}).

\smallbreak
We now prove \eqref{singular} by using \eqref{diff-sing}
instead of \eqref{diff-reg}. For this purpose,
when $l\ge 5$, we write
\begin{eqnarray*}
&&\Big|\Big(\big(W_{\varphi,l}\,\,Q(f,\,\,g)-Q(f,\,\,W_{\varphi,l}\,\, g)\big),\,\,\,
h\Big)_{L^2(\RR^6)}\Big| \\
& \leq& C \left \{ \iiiint  b\, |\theta|\,\,
|(W_{\varphi,3} f)'_*|\,\,|(W_{l}g)'|\,\,|h|\,\,dvdv_*d\sigma dx
\right.\\
&&\left . + \iiiint  b\, |\theta|^{l-2}\,\,
|(W_{\varphi,l} f)'_*|\,\,|g'|\,\,|h|\,\,dvdv_*d\sigma dx \right \}\\
%&& +\left. \iiiint  b\, |\theta|\,\,
%|(W_{\varphi,0}f)'_*|\,\,|(W_{l}g)'|\,\,|h|\,\,dvdv_*d\sigma dx\right\}\\
& =&\cM_1 + \cM_2 .
\end{eqnarray*}
The estimation on  $\cM_1$ can be obtained following the proof of  (\ref{usual}) except for the $x$ variable.
Indeed,
\begin{align*}
\cM_1
&\leq C \int \Big(\iiint  b\, |\theta|\,\, |(W_{\varphi,3}\,
 f)_*|\,\,|(W_{l}\, g)|^2 dv dv_*
d\sigma \Big)^{1/2}\\
&\hskip1cm \times \Big(\iiint  b\, |\theta|\,\, |(W_{\varphi,3}\, f)_*|\,
|h'|^2\, dvdv_*d\sigma \Big)^{1/2} dx\\
&\leq C \|g\|_{L^\infty(\RR^3_x; L^2_{l}(\RR^3_v))} \int
\|W_{\varphi,3}\,f\|_{L^1(\RR_v^3)}\|h\|_{L^2(\RR_v^3)}\,dx\\
&\leq C\|g\|_{L^\infty(\RR^3_x; L^2_{l}(\RR^3_v))} \left(\int
\|W_{\varphi,5}\,f\|_{L^2(\RR_v^3)}^2 dx
\right)^{1/2}\left(\int\|h\|_{L^2(\RR^3_v)}dx
\right)^{1/2}\\
&\leq C\|W_{\varphi,5}\,f\|_{L^2(\RR^6)}\|g\|_{L^\infty(\RR^3_x; L^2_{l}(\RR^3_v))}
\|h\|_{L^2(\RR^6)}.
\end{align*}
 $\cM_2$ can be estimated as follows. Firstly, we have
\begin{align*}
\cM_2^2 =& C^2  \left(\iiiint  b\, |\theta|^{l-2}
|(W_{\varphi,l} f)_*||g|\,\,|h'|\,\,dvdv_*d\sigma dx \right)^2  \\
 \le &C^2  \iiiint b \, |\theta|^{\, l-2-\frac{3}{2}}
 |g||(W_{\varphi,l} f)_*|^2 dv dv_* d\sigma dx \\
& \times \iiiint b \, |\theta|^{\,l-2+\frac{3}{2}}|g||h'|^2 dv dv_* d \sigma dx\\
= & \cM_{2,1} \times \cM_{2,2}.
\end{align*}
Then, if $l-2-\frac 32-2s-1>-1$, that is,  $l>2s+\frac 32 +2$, we have
$$\cM_{2,1} \leq C  \|g\|_{L^\infty(\RR_x^3; L^1(\RR_v^3))} \|W_{\varphi, l}f\|_{L^2(\RR^6)}^2.$$
On the other hand, for $\cM_{2,2}$ we need to apply
the singular change of
variables $v_* \rightarrow v'$. The Jacobian of this transform
is
\begin{equation*}%\label{jacobian2}
\Big|\frac{\pa v_*}{\pa v'}\Big|=\frac{8}{ \Big|I- \vk\otimes
\sigma\Big|}=\frac{8}{|1-\vk\cdot\sigma|} =\frac{4}{\sin^2
(\theta/2)}\le 16\theta^{-2}, \ \ \theta\in[0,\pi/2].
\end{equation*}
Notice that this gives rise to an additional singularity
in the angle $\theta$ around $0$. Actually, the
situation is even worse in the following sense.
Recall that $\theta$ is no longer
a legitimate polar angle. In this case, the best choice of the pole is
$\vk''=(v'-v)/|v'-v|$ for which polar angle $\psi$ defined by
$\cos\psi=\vk''\cdot\sigma $  satisfies (cf. \cite[Fig. 1]{al-1})
\[
\psi=\frac{\pi-\theta}{2}, \qquad d\sigma=\sin\psi d\psi d\phi,
\qquad \psi\in[\frac{\pi}{4},\frac{\pi}{2}].
\]
This measure does not cancel any of the singularity of $b(\cos\theta)$,
unlike the case in the usual polar coordinates.
Nevertheless, this singular change of variables yields
\begin{align*}
\cM_{1,2} &= C \iiiint  b(\cos\theta)\, |\theta|^{l+ \frac{3}{2}}
|(g)|\,\,|h'|^2\,\,dvdv_*d\sigma dx \\
&\leq C\iiint D_1(v, v') |(g)|\,\,|h'|^2dv dv' dx,
\end{align*}
when $l-2>\frac 32+2s$ because
\[
D_1(v,v')=\int_{\SS^2}\theta^{l-2+\frac 32-2}b(\cos\theta)d\sigma \le
C \int_{\pi/4}^{\pi/2}(\frac{\pi}{2}-\psi)^{-2-2s+l-2+\frac 32-2}d\psi
\le C.
\]
Therefore,
\begin{equation*}
\cM_{2,2} \leq C \|g\|_{L^\infty(\RR^3_x;L^1(\RR_v^3))} \|h\|_{L^2(\RR^6)}^2.
\end{equation*}
Now the proof of \eqref{diff-sing} is completed by using the imbedding estimate
for $l>\frac 32$,
$$
\|g\|_{L^1(\RR_v^3)}\leq C\|g\|_{L^2_l(\RR_v^3)}.
$$
And this completes the proof of the lemma.
\end{proof}

We are now ready to conclude the proof of the uniqueness theorem.\\

\noindent{\bf Proof of Theorem \ref{theo2} : }
Set $F=f_1-f_2$. Then we have
\begin{equation}\label{E-Cauchy}
\left\{\begin{array}{l}
F_t+v\cdot\nabla_x F =Q(f_1,\, F)+Q(F,\, f_2)\, ,\\
F|_{t=0}=0.
\end{array}\right.
\end{equation}
Let $S(\tau) \in C_0^\infty(\RR)$ satisfy $0 \leq S \leq 1$ and
$$S(\tau)=1, \enskip |\tau| \leq 1 \enskip; \enskip S(\tau)=0, \enskip |\tau| \geq 2. $$
Set $S_N(D_x) = S(2^{-2N}|D_x|^2)$ and multiply $W_{\varphi, l} S_N(D_x)^2 W_{\varphi, l} F$
to \eqref{E-Cauchy}. Integrating and letting $N \rightarrow \infty$, we have
\begin{align*}%\label{4.4.1}
\frac 12 \frac{d}{d t} \|  W_{\varphi, l} F(t)\|^2_{L^2(\RR^6)}&=\Big(W_{\varphi,l}\,
Q(f_1,\, F)+W_{\varphi, l}\, Q(F,\, f_2)\, , W_{\varphi,l} F\Big)_{L^2(\RR^6)}\\
&- ( v\cdot \nabla_x (\varphi^{-1}) W_l F, W_{\varphi,l} F)_{L^2(\RR^6)}, %\notag
\end{align*}
because
$( v\cdot \nabla_x S_N(D_x)W_{\varphi, l }F, S_N(D_x)W_{\varphi,l} F)_{L^2(\RR^6)} =0$.
The second term on the right hand side is estimated by $\|W_{\varphi,l} F\|^2_{L^2(\RR^6)}$.
Since $f_1\geq 0$, {}from the coercivity  of $-(Q(f_1, g), g)$ it follows that
$$
\Big(Q(f_1,\, W_{\varphi, l }F)\, ,W_{\varphi,l }F\Big)_{L^2(\RR^6)}\leq C
\|f_1(t)\|_{L^\infty(\RR^3_x; \,L^1(\RR^3_{v}))}
\|W_{\varphi,l} F(t)\|^2_{L^2(\RR^6_{x, \, v})}.
$$
By Lemma \ref{moment-commutator-with-x},
we have %Using \eqref{3.11-+} with $\gamma^+=0$ gives
\begin{align*}%\label{usual}
&\left|\Big(\big(W_{\varphi,l}  \,\, Q(f_1,\,\,F)-Q(f_1,\,\,W_{\varphi,l}  \,\,
F)\big),\,\,\, W_{\varphi,l}  \,\,
F\Big)_{L^2(\RR^6)}\right| \\
&\leq C
\|f_1\|_{L^\infty(\RR_x^3; L^1_{l}(\RR^3_v))} \|W_{\varphi, l} F\|_{L^{2}
(\RR^6)} ^2, \notag
\end{align*}
and
\begin{align*}%\label{singular}
&\left|\Big(\big(W_{\varphi, l}  \,\, Q(F,\,\,f_2)-Q(F,\,\,W_{\varphi,l}  \,\,
f_2)\big),\,\,\, W_{\varphi, l} F      \Big)_{L^2(\RR^6)}\right| \\
&\leq C \|f_2\|_{L^\infty(\RR^3_x; L^2_{l}(\RR^3_v))}
\|W_{\varphi, l}F\|^2_{L^2(\RR^6)}\,.
\end{align*}
Finally, for $l>7/2+2s$, we have
\begin{eqnarray*}
&&\left|\Big(Q(F,\, W_{\varphi, l } f_2)\, ,W_{\varphi, l} F\Big)_{L^2(\RR^6)}\right|\leq C
\|Q( F,\, W_{\varphi,l} f_2)\|_{L^2(\RR^6)}\|W_{\varphi, l}F(t)\|_{L^2(\RR^6)}\\
&& \leq C \|W_{\varphi, l}F(t)\|_{L^2(\RR^6)} \left(\int_{\RR^3_x}
\|\frac{F(t, x,\cdot\,)}{\la x \ra^2}\|^2_{L^{1}_{2s}(\RR^3_{v})}
\|\frac{\la x \ra^2}{\varphi}f_2(t, x, \cdot\,
)\|^2_{H^{2s}_{l+2s}(\RR^3_{v})} dx\right)^{1/2}\\
&&\leq C \|W_{\varphi,l}F(t)\|^2_{L^2(\RR^6)}\|f_2(t)\|_{L^\infty(\RR^3_x;\;
H^{2s}_{l+2s}(\RR^3_{v}))},
\end{eqnarray*}
because $\la x \ra^{-2} \leq W_{\varphi,2}$ and $\la x \ra^2/\varphi$ is a bounded
operator on $H^{2s}$ uniformly with respect to $x$.
Thus, we have, for any $0<t< T$
\begin{align*}
&\frac{d}{d t} \|W_{\varphi,l} F(t)\|^2_{L^2_l(\RR^6)} \\
&\leq C \left(
\|f_1\|_{L^\infty(]0, T[\times \RR_x^3 ;\; L^1_{l}(\RR^3_{v}))}+
\|f_2\|_{L^\infty(]0, T[  \times \RR_x^3;\; H^{2s}_{l+2s}(\RR^6_{v}))}
\right)\|W_{\varphi,l}F(t)\|^2_{L^2_{l}(\RR^6)}.
\end{align*}
Therefore, $\| W_{\varphi,l}F(0)\|_{L^2(\RR^6)}=0$ which implies $\|
W_{\varphi,l}F(t)\|_{L^2(\RR^6)}=0$ for all $t\in [0, T[$.
And this gives $f_1=f_2$, and thus completes the proof of Theorem \ref{theo2}.\\

\begin{rema}
For the function space considered in Theorem \ref{theo1}, the uniqueness of
solutions
is a direct consequence of Theorem \ref{theo2} if there exists a non-negative
solution. It is
because $g_0 \in H^k_{\ell}$ and $ g \in L^\infty(]0,T[\times H_{\ell}^k)$ with $k, \ell \geq 3$
imply
$f_0 = \mu + \sqrt \mu g_0 \in L^\infty(\RR_x^3; H^{2s}_m)$ and $f = \mu + \sqrt \mu g
\in L^\infty(]0,T[\times \RR_x^3; H_m^{2s})$ for any $m$, respectively, and
$k > 3/2 + 2s$.
\end{rema}

%%%%%%%%%%%%%%%%%%%%%%%%%%%%%%%%%%%%%%%%%%%%%%%%%%%%%%%%%%%%%%%%%%%%%%%
\subsection{Non-negativity}\label{section4.4}
\setcounter{equation}{0}
In this subsection, we will prove the non-negativity of the solution
obtained in Theorem \ref{theo1}.
\begin{theo}\label{theo4.4.1}
Let $N\geq 3, \ell\geq 3$. There exists $\epsilon_1>0$ such that if
 $g_0\in H^N_\ell(\RR^6)$ with $\mu + \mu^{1/2} g_0\geq 0$ and
$ \|g_0\|_{H^N_\ell(\RR^6)}\leq \epsilon_1\,$, and $ g \in
L^\infty([0, T];\, H^N_\ell(\RR^6))$ is a solution of Cauchy problem
(\ref{cauchy-problem}), then we have $\mu + \mu^{1/2} g\geq 0$ on $
[0, T]\times\RR^6$.
\end{theo}

\begin{proof} By applying the Remark 5.3 on the uniqueness to
 the Cauchy problem (\ref{cauchy-problem}), it is
enough to prove the non-negativity of the approximate solutions given by
Proposition \ref{prop4.2.1}, that is,
\begin{equation}\label{4.4.1+100}
f^n = \mu + \mu^{1/2} g^n \geq 0\,,\qquad n\in\NN.
\end{equation}
Again, this will be proved by induction. It is clearly true for $n=0$
by taking $g^0=g_0$, and so we now
assume that it is true for some $n$ and will prove that (\ref{4.4.1+100}) is true
for $n+1$.

{}From (\ref{4.2.1}), $f^{n+1}= \mu + \mu^{1/2} g^{n+1}$ is the
solution of the following Cauchy problem :
\begin{equation}\label{4.4.3}
\left\{\begin{array}{l} \partial_t f^{n+1} + v\,\cdot\,\nabla_x
f^{n+1} =Q (f^n, f^{n+1}), \\ f^{n+1}|_{t=0} = f_0 =\mu + \mu^{1/2}
g_0\geq 0.
\end{array} \right.
\end{equation}

Take the convex function 
\[
\beta (s) = \frac 1 2 (s^- )^2= \frac 1 2
s\,(s^- )
\]
with $s^-=\min\{s, 0\}$, and notice that 
\[
\beta_s(s)=\frac{d\beta(s)}{ds}=s^-.
\]  
Setting
$\phi(x,v) = (1+|x|^2+|v|^2)^{-2}$ and noticing that
$$
\beta_s (f^{n+1}) \phi(x, v)
% \left(\frac{d}{ds}\,\,\beta
% \right)(f^{n+1}) \phi(x, v)
=\min\{f^{n+1}, 0\}\phi(x, v)\in
L^\infty([0,T];L^1(\RR_x^3; L^2(\RR^3_v)),
$$
we have by \eqref{4.4.3},
\begin{align*}
&\frac{d}{dt} \int_{\RR^6} \beta (f^{n+1}) \phi dxdv = \int_{\RR^6}
Q(f^n,\, f^{n+1})\,\, \beta_s(f^{n+1}) \phi \,\,dxdv \\
&\qquad\qquad- \int_{\RR^6} { v\,\cdot\, \nabla_x \enskip (\beta
(f^{n+1}) \phi) }dxdv - \int_{\RR^6} { (\phi^{-1}\,\,v\, \cdot\,
\nabla_x \,\phi ) }\enskip \beta (f^{n+1})\phi dxdv,
\end{align*}
where the first term on the right hand side is well defined by
Theorem \ref{theo2.1}, because $f^n$ belongs to $L^\infty([0,T]\times \RR_x^3;
L^1_{2s} \cap H_{2s}^{2s}(\RR_v^3))$. Since the second term vanishes
and $|v\, \cdot\, \nabla_x\, \phi | \leq C \phi$,
we obtain
\begin{align*}
\frac{d}{dt} \int_{\RR^6} \beta (f^{n+1}) \phi dxdv  \leq
\int_{\RR^6} Q(f^n, f^{n+1} ) \beta_s(f^{n+1}) \phi dxdv + C
\int_{\RR^6}  \beta (f^{n+1}) \phi dxdv.
\end{align*}

For the first term on the right hand side, we have
\begin{align*}
&\int_{\RR^6} Q(f^n, f^{n+1} )\, \beta_s(f^{n+1})\phi  dxdv\\
& =  \int_{\RR^6_{x, v}}\int_{\RR^3_{v_*}\times\SS^2_\sigma}
b(\cos\theta) \Big( f^{n '}_\ast f^{n+1 '}
 - f^n_\ast f^{n+1} \Big) \beta_s (f^{n+1})\phi \\
&= \int_{\RR^6_{x, v}}\int_{\RR^3_{v_*}\times\SS^2_\sigma}
b(\cos\theta) f^{n '}_\ast \Big( f^{n+1 '}
 -f^{n+1}\Big) \beta_s(f^{n+1}) \phi \\
 &\qquad\qquad+  \int_{\RR^6_{x, v}}\beta_s(f^{n+1} )
 f^{n+1} \phi \int_{\RR^3_{v_*}\times\SS^2_\sigma}
 b(\cos\theta) \Big(f^{n '}_\ast -f^n_\ast \Big)\\
&=I+II \,\, .
\end{align*}
{}From (\ref{4.2.3}), we have, for any $n\in\NN$,
\begin{align*}
&\|f^n\|_{L^\infty([0, T]\times\RR^3_x;\, L^1(\RR^3_v))} \leq 1 +
\|\sqrt\mu\,\, g^n\|_{L^\infty([0, T]\times\RR^3_x;\,
L^1(\RR^3_v))}\\
&\qquad\leq 1 + C\|g^n\|_{L^\infty([0, T]\times\RR^3_x;\,
L^2(\RR^3_v))}\leq 1+C\epsilon_0\, ,
\end{align*}
so that the cancellation lemma {}from \cite{al-1} implies that
$$\int_{\RR^3_{v_*}\times\SS^2_\sigma}
 b(\cos\theta) \Big(f^{n '}_\ast -f^n_\ast \Big) =C \int_{\RR^3_{v}}
 f^n (t,x,v) dv \leq C || f^n||_{L^\infty([0, T]\times\RR^3_x;\, L^1(\RR^3_v))}
 \leq C\, ,
 $$
while $\beta_s (s) s =2 \beta (s)$ implies that
$$
|II|\leq C \int_{\RR^6}  \beta (f^{n+1}) \phi dxdv\, .
$$
On the other hand, by the convexity of $\beta$, that is,
$$
\beta_s(a)(b-a)\leq \beta(b)-\beta(a)\,,
$$
and the assumption that ${f_*^n}' \geq 0$, we get
\begin{align*}
I&=\int_{\RR^6_{x, v}}\int_{\RR^3_{v_*}\times\SS^2_\sigma}
b(\cos\theta) f^{n '}_\ast \Big( f^{n+1 '}
 -f^{n+1}\Big) \beta_s(f^{n+1}) \phi \\
&\leq \int_{\RR^6_{x, v}}\int_{\RR^3_{v_*}\times\SS^2_\sigma}
b(\cos\theta) f^{n '}_\ast \Big( \beta (f^{n+1 '})
 - \beta (f^{n+1}) \Big) \phi \\
 &\leq \int_{\RR^6_{x, v}}\int_{\RR^3_{v_*}\times\SS^2_\sigma}
b(\cos\theta) \Big( f^{n'}_\ast \beta (f^{n+1 '}) - f^n_\ast \beta
(f^{n+1})\Big) \phi\\
&\qquad\qquad- \int_{\RR^6_{x,
v}}\int_{\RR^3_{v_*}\times\SS^2_\sigma} b(\cos\theta)\beta
(f^{n+1}) \Big( f^{n'}_\ast -f^n_\ast \Big)\phi\\
& \leq  \int_{\RR^6_{x,
v}}\int_{\RR^3_{v_*}\times\SS^2_\sigma} b(\cos\theta)  f^{n}_\ast
\beta (f^{n+1 })
 \Big ( \phi' - \phi \Big)
+C \int_{\RR^6}  \beta (f^{n+1}) \phi dxdv \\
&=I_1+I_2 \,.
\end{align*}
Applying Taylor's formula to the first term gives
$$
 I_1 = \int_0^1 d \tau  \int_{\RR^6_{x,
v}}\int_{\RR^3_{v_*}\times\SS^2_\sigma} b(\cos\theta)  f^{n}_\ast
\beta (f^{n+1 })\,\big(v'-v\big)\, \cdot\,  \nabla_v \phi  \Big(v+
\tau (v'-v))\Big) \,.
$$
Since
\begin{equation}\label{4.2.8}
|v' -v|=|v-v_*|\sin\,\Big(\frac\theta 2\Big) \leq  \langle v \rangle
\,\,\langle v_* \rangle \sin\,\Big(\frac\theta 2\Big),
\end{equation}
by setting $v_\tau = v+ \tau (v'-v)$, $0<\tau<1$, $0\leq\theta\leq
\pi/2$, we have
$$
|v| \leq |v_\tau| + |v'-v|\leq |v_\tau| +  \sin\,\Big(\frac\theta
2\Big)\Big(|v|+|v_*|\Big)\leq |v_\tau| + \frac{\sqrt 2
}{2}|v|+|v_*|.
$$
Then
$$
(1+|x|^2+|v|^2) \leq C(1+|x|^2+|v_\tau|^2)(1 + |v_*|^2),
$$
which implies
$$
|\nabla_v \phi (x,\, v_\tau)| \leq (1+|x|^2+|v_\tau|^2)^{-5/2}\leq C
\phi(x,v) \frac{\langle v_* \rangle^5}{\langle v \rangle}\,\, .
$$
So we obtain
$$|I_1| \leq C ||f^n||_{L^\infty([0,\, T]\times\RR_x^3;\, L^1_{6}(\RR^3_v))}
\int_{\RR^6}  \beta (f^{n+1}) \phi dxdv.
$$
Again {}from (\ref{4.2.3}), we have, for any $n\in\NN$,
\begin{equation*}
\|f^n\|_{L^\infty([0, T]\times\RR^3_x;\, L^1_6(\RR^3_v))} \leq
\|\mu^{1/2}\|_{L^1_6(\RR^3)} + \|\mu^{1/2} g^n\|_{L^\infty([0,
T]\times\RR^3_x;\, L^1_6(\RR^3_v))}\leq C (1+\epsilon_0).
\end{equation*}
Finally, we have obtained, for $0<t<T$,
$$
 \frac{d}{dt} \int_{\RR^6}\beta
(f^{n+1})\phi\, dx\,dv \leq C \int_{\RR^6} \beta (f^{n+1})\phi
\,dx\,dv ,\qquad \beta (f^{n+1})|_{t=0} = 0.
$$
Therefore, for $0<t<T$, and by continuity,
$$
\int_{\RR^6} \beta(f^{n+1}(t))\phi \,dx\,dv=0
$$
which implies that, $f^{n+1}(t, x, v) \geq 0$ for $(t, x, v)\in [0,
T]\times\RR^6_{x, v}$. Therefore,  the proof of Theorem \ref{theo4.4.1}
is completed.
\end{proof}

\begin{rema}
Note that the above analysis can be extended to the strong singularity
case.

Indeed, by writing
\begin{align*}
 I_1& =   \int_{\RR^6_{x,
v}}\int_{\RR^3_{v_*}\times\SS^2_\sigma} b(\cos\theta)  f^{n}_\ast
\beta (f^{n+1 })\,\big(v'-v\big)\, \cdot\,  \nabla_v \phi  (v) \\
&+ \frac{1}{2} \int_0^1 d \tau  \int_{\RR^6_{x,
v}}\int_{\RR^3_{v_*}\times\SS^2_\sigma} b(\cos\theta)  f^{n}_\ast
\beta (f^{n+1 })\,\big(v'-v\big)^2 \, \,  \nabla_v^2 \phi  \Big(v+
\tau (v'-v))\Big) \\
&= I_{11} + I_{12},
\end{align*}
since we have
$$
|\nabla_v^2 \phi (x,\, v_\tau)| \leq (1+|x|^2+|v_\tau|^2)^{-3}\leq C
\phi(x,v) \frac{\langle v_* \rangle^6}{\langle v \rangle^2}\,\, ,
$$
it follows {}from \eqref{4.2.8} that
$$|I_{12}| \leq C ||f^n||_{L^\infty([0,\, T]\times\RR_x^3;\, L^1_{8}(\RR^3_v))}
\int_{\RR^6}  \beta (f^{n+1}) \phi dxdv.
$$
On the other hand, setting $ {\bf k} = \frac{v-v_*}{|v-v_*|}$ and
writing
$$v' -v = \frac{1}{2}|v-v_*|\Big( \sigma - (\sigma \cdot {\bf k})  {\bf k} \Big)
+ \frac{1}{2} ((\sigma \cdot {\bf k})-1)( v-v_*),$$ we have
$$
 I_{11} =   \frac{1}{2}\int_{\RR^6_{x,
v}}\int_{\RR^3_{v_*}\times\SS^2_\sigma} b(\cos\theta)  f^{n}_\ast
\beta (f^{n+1 })\,\big(\cos \theta -1 \big)(v-v_*) \cdot  \nabla_v
\phi  (v),
$$
where we have used the symmetry that $\int_{\SS^2} b(\sigma
\cdot {\bf k})\Big( \sigma - (\sigma \cdot {\bf k})  {\bf k} \Big) d
\sigma =0$. Therefore, we have
$$|I_{11}| \leq C ||f^n||_{L^\infty([0,\, T]\times\RR_x^3;\, L^1_{6}(\RR^3_v))}
\int_{\RR^6}  \beta (f^{n+1}) \phi dxdv,
$$
and the same estimation holds also in the
strong singularity case.
\end{rema}

%%%%%%%%%%%%%%%%%%%%%%%%%%%%%%%%%%%%%%%%%%%%%%%%%%%%%%%%%%%%%%%%%%%%%%
%%%%%%%%%%%%%%%%%%%%%%%%%%%%%%%%%%%%%%%%%%%%%%%%%%%%%%%%%%%%%%%%%%%%%%
\section{Full regularity }\label{section6}
\smallbreak

%%%%%%%%%%%%%%%%%%%%%%%%%%%%%%%%%%%%%%%%%%%%%%%%%%%%%%%%%%%%%%%%%%%%%%%

We now prove the smoothness effect of the Cauchy problem  for the
non-cutoff Boltzmann equation. More precisely, the main result of this section is given by

\begin{theo}\label{theo4.5.1}
Assume that $0<s<1/2$. There exists $\epsilon_1>0$ such that if
 $g_0\in H^3_3(\RR^6)$ with $\mu + \mu^{1/2} g_0\geq 0$,
$ \|g_0\|_{H^3_3(\RR^6)}\leq \epsilon_1\,$,  and $ g \in L^\infty([0,
T];\, H^3_3(\RR^6))$ is the solution of Cauchy problem
(\ref{cauchy-problem}), then we have $g\in C^\infty(]0,
T[\times\RR^6)$.
\end{theo}

Let us recall that $ H^k_\ell  (\RR^7_{t, x,}) $, $H^k_\ell(\RR^6_{x,\,
v})$ and $H^k_\ell(\RR^3_{v})$ denote some weighted Sobolev spaces
 with the weight defined in  the variable $v$. Since the regularity result to be proved is local in space and time,  for convenience, we define the corresponding
 local version of weighted Sobolev spaces. For
$0\leq T_1<T_2<+\infty$, and any given open domain
$\Omega\subset\RR^3_x$, define
\begin{eqnarray*}
{\mathcal H}^m_l  (]T_1, \, T_2[\times\Omega\times\RR^3_v)&=&
\Big\{f\in \cD^{\, '}(]T_1, \, T_2[\times\Omega\times\RR^3_v);\,\, \\
&&\varphi(t)\psi(x) f\in H^m_l(\RR^7)\,,\,\,\forall\,\varphi\in
C^\infty_0(]T_1, \, T_2[),\,\,\psi\in C^\infty_0(\Omega)\, \Big\}.
\end{eqnarray*}

The proof of Theorem \ref{theo4.5.1} will be divided into several steps.

\subsection{Formulation of the problem}\label{section6.1}
\setcounter{equation}{0}

First of all, we recall  the main result  in \cite{amuxy-nonlinear-3}  given below.

\begin{theo}\label{theo0.1}
Assume that $0<s<1$,\, $0\leq T_1<T_2<+\infty$,
$\Omega\subset\RR^3_x$ is an open domain. Let $f$ be a non-negative
solution of the Boltzmann equation (\ref{1.1}) satisfying $f\in
{\mathcal H}^5_\ell  (]T_1, \, T_2[\times\Omega\times\RR^3_v)$ for
all $\ell \in\NN$. Moreover, assume that $f$ satisfies the non-vacuum
condition
\begin{equation}\label{0.1.4}
\|f(t, x, \cdot)\|_{L^1(\RR^3_v)}>0,
\end{equation}
for all $(t, x)\in ]T_1, \, T_2[\times\Omega$. Then we have
$$
f\in {\mathcal H}^{+\infty}_\ell(]T_1, \,
T_2[\times\Omega\times\RR^3_v)\subset C^{+\infty}(]T_1, \,
T_2[\times\Omega\times\RR^3_v),
$$
for any $\ell\in\NN$.
\end{theo}

To apply this result, let us firstly note that, by Theorem \ref{theo4.2.1} and Theorem
\ref{theo4.4.1}, under the assumption of Theorem \ref{theo4.5.1},
the unique solution of the Cauchy problem (\ref{cauchy-problem})
satisfies
$$
\|g\|_{L^\infty([0,\, T];\, H^3_3(\RR^6))}\leq\,\epsilon_0,\qquad
\mu+ \sqrt\mu\, g\geq 0.
$$
Therefore, $f=\mu + \sqrt \mu\, g \geq 0$ is a solution of Boltzmann
equation (\ref{1.1}). On the other hand, we can choose $\epsilon_0>0$
small enough such that
$$
|| \sqrt\mu\,g||_{L^\infty([0, T]\times\RR^3_{x};\, L^1(\RR^3_v
))}\leq C\|g\|_{L^\infty([0,\, T];\, H^2(\RR^6))}\leq\,C\epsilon_0<1
$$
where $C$ is the Sobolev constant of the imbedding $H^2(\RR^3_x)\subset
L^\infty(\RR^3_x)$. Thus, for any $(t, x)\in ]0, T[\times\RR^3_x$,
\begin{equation}\label{5.3.1+0}
\int_{\RR^3_v} f(t,x,v) dv =1+ \int_{\RR^3_v} \sqrt\mu\,\, g(t, x,
v) dv \geq 1- || \sqrt\mu\,g||_{L^\infty([0, T]\times\RR^3_{x};\,
L^1(\RR^3_v ))}>0\,,
\end{equation}
so that $f=\mu + \sqrt \mu\, g$ satisfies the condition (\ref{0.1.4}).

By using equation (\ref{1.1}) and Remark \ref{rema4.4}, we
have also, for any $\ell\in\NN$,  $0< T_1<T_2< T$ and
bounded open domain $\Omega\subset\RR^3_x$,
$$
f=\mu + \sqrt \mu\, g \in {\mathcal H}^3_\ell  (]T_1, \,
T_2[\times\Omega\times\RR^3_v).
$$

 Note that we can not apply directly Theorem \ref{theo0.1}
  because we now only know that $f$ has regularity just in ${\mathcal H}^3_\ell  (]T_1, \,
T_2[\times\Omega\times\RR^3_v)$.
To overcome this, we prove the following theorem.

\begin{theo}\label{theo5.3.1}
Under the assumptions of Theorem \ref{theo4.5.1}, we have, for any
$0<T_1<T_2<T$ and bounded open domain $\Omega\subset\RR^3_x$,
$$
f=\mu + \sqrt \mu\, g \in {\mathcal H}^5_l  (]T_1, \,
T_2[\times\Omega\times\RR^3_v),
$$
holds for all $\ell\in\NN$.
\end{theo}

The proof of this theorem is similar but easier than
that of Theorem \ref{theo0.1} which was proved in \cite{amuxy-nonlinear-3}. In fact, since we have
$$
g^n\,\rightarrow \,g,
$$
by mollifying the initial data and using the uniqueness of solution,
we do not need at all to mollify the solution as in  \cite{amuxy-nonlinear-3}.
It follows that to obtain the above regularization result, we
 only need to prove the \`a priori estimate on smooth solution, which can deduced {}from \cite{amuxy-nonlinear-3}.

To make the paper self-contained, we give a proof here.   Let us recall that
 here we consider the Maxwellian molecule type cross-sections with the mild singularity.

Here and below,  $\phi$ denotes a cutoff  function satisfying
$\phi\in C^\infty_0$ and $0\leq \phi\leq 1$. Notation
$\phi_1\subset\subset\phi_2$ stands for two cutoff functions such
that $\phi_2=1$ on the support of $\phi_1$.

Take the smooth cutoff functions $\varphi,\, \varphi_2, \varphi_3\in
C^\infty_0(]T_1, T_2[)$ and $ \psi,\, \psi_2, \psi_3\in
C^\infty_0(\Omega)$  such that $\varphi\subset\subset
\varphi_2\subset\subset \varphi_3$ and
$\psi\subset\subset\psi_2\subset\subset\psi_3$. Set
$f_1=\varphi(t)\psi(x) f$, $f_2=\varphi_2(t) \psi_2(x) f$ and
$f_3=\varphi_3(t) \psi_3(x) f$. For $\alpha\in\NN^6, |\alpha|\leq
3$, define
$$
g=\partial^\alpha_{x, v} (\varphi(t)\psi(x) f)\in L^\infty(]T_1,
T_2[;\, L^2_\ell (\RR^6)).
$$
Then the Leibniz formula  yields the following equation :
\begin{equation}\label{5.3.1}
 g_t + v\,\cdot\,\partial_x {g } = Q(f_2,\,\, g)+ G,
\enskip (t, x, v) \in \RR^7,
\end{equation}
where
\begin{eqnarray}\label{5.3.2}
G&=&\sum_{\alpha_1+\alpha_2=\alpha,\,\, 1\leq
|\alpha_1|}C^{\alpha_1}_{\alpha_2} Q\Big(\partial^{\alpha_1}
{f_2},\,\,
\partial^{\alpha_2} f_1\Big)\\
&+&\partial^\alpha\Big( \varphi_t \psi(x) f+ v\,\cdot\,\psi_x(x)
\varphi(t) f\Big) +[\partial^\alpha,\,\,\,
v\,\cdot\,\partial_x](\varphi(t)\psi(x) f)\nonumber \\
&\equiv & A+B+C.\nonumber
\end{eqnarray}
Then we can take $W^{2\ell}\,g $ as a test function for equation (\ref{5.3.1}). It follows by integration by parts on
$\RR^7=\RR^1_t\times \RR^3_x\times\RR^3_v$ that
\begin{equation*}\label{5.3.3}
0=\Big(W^{\ell}\,Q({f_2},  g),\, W^{\ell}\,g\Big)_{L^2(\RR^7)}+
\Big(G,\, W^{2\ell}\,g \Big)_{L^2(\RR^7)},
\end{equation*}
which is sufficient for obtaining the required initial regularity.

\subsection{Gain of regularity in velocity variable}\label{section6.2}
\setcounter{equation}{0}
The next step is to show the gain of regularity in the velocity variable
by using the coercivity of the collision operator.

\begin{prop}\label{prop5.3.1}
Under the assumption of Theorem \ref{theo4.5.1}, for any
$0<T_1<T_2<T$ and  bounded open domain $\Omega\subset\RR^3_x$,
one has,
\begin{equation*}\label{5.3.4}
\Lambda^s_v f_1\in L^2(\RR_t;\, H^3_\ell  (\RR^6)),
\end{equation*}
for any $\ell\in\NN$, where $\Lambda_v=(1-\Delta_v)^{\frac 12}$,
$f_1=\varphi(t)\psi(x) f$ with
$\varphi\in C^\infty_0(]T_1, T_2[), \psi\in C^\infty(\Omega)$.
\end{prop}

\begin{proof} Firstly,  the local positive lower bound
(\ref{5.3.1+0}) implies that
 $$
 \inf_{(t,  x)\in \mbox{Supp}\,\,\varphi \times \mbox{Supp}\,\psi_1}
 \|{f_2}(t, x, \cdot)\|_{L^1(\RR^3_v)}=c_0>0.
 $$
Thus, the coercivity estimate \eqref{2.1.1} gives
\begin{eqnarray*}
&&-\Big(Q({f_2},\,  W^{\ell}\, g),\,\,
 W^{\ell}\, g\Big)_{L^2(\RR^7)}=-\int_{t\in
 \mbox{Supp}\,\varphi }\int_{x\in
 \mbox{Supp}\,\psi_1 }\Big(Q({f_2},\,  W^{\ell}\, g),\,\,
W^{\ell}\, g\Big)_{L^2(\RR^3_v)} dx dt
\\
&& \geq \int_{\RR_t}\int_{\RR^3_x}\Big( C_0\| W^{\ell}\, g(t, x,
\cdot)\|^2_{H^s(\RR^3_v)}
 -C \|{f_2}(t, x, \cdot)\|_{L^1(\RR^3_v)}
 \| W^{\ell}\, g(t, x, \cdot)\|^2_{L^2(\RR^3_v)}\Big)dx dt \\
&&\geq C_0\|\Lambda^{s}_v W^{\ell}\, g\|^2_{L^2(\RR^7)}- C T\|
{f_2}\|_{L^\infty(\RR^4_{t, x};\,\,
L^1(\RR^3_v))}\|W^{\ell}\,g\|^2_{L^\infty([0, T];\, L^2(\RR^6))}.
\end{eqnarray*}
Since
$$
\|{f_2}\|_{L^\infty(\RR^4_{t, x};\,\, L^1(\RR^3_v))}\leq C
\|{f_2}\|_{L^\infty([0, T];\,\,
H^{3/2+\epsilon}_{3/2+\epsilon}(\RR^6))},
$$
and
$$
\|W^{\ell}\,g\|^2_{L^\infty([0, T];\, L^2(\RR^6))}\leq C
\|f_1\|^2_{L^\infty([0, T];\, H^3_\ell(\RR^6))},
$$
for $l>3/2$, we have
\begin{eqnarray}\label{5.3.6}
\|\Lambda^{s}_v  W^{\ell}\, g\|^2_{L^2(\RR^7)}&\leq& C
\|{f_2}\|^2_{L^\infty([0, T];\,\, H^{3}_{\ell}(\RR^6))}
+\left|\big(G,\, W^{2\ell}\, g\big)_{L^2(\RR^7)}\right|\\
&&+\Big|\big(  W^{\ell}\,  Q({f_2}, g)-Q({f_2},\,  W^{\ell}\,  g),
\,\,  W^{\ell}\, g\big)_{L^2(\RR^7)}\Big|\, .\nonumber
\end{eqnarray}

By applying Lemma \ref{lemm3.2.1}, the third term on the right hand side
 of \eqref{5.3.6} can be estimated as follows:
\begin{align*}
&\left|\Big(\big(W^\ell\, Q(f_2,\,\,g)-Q(f_2,\,\,W^\ell\,
g)\big),\,\,W^{\ell}\, g\Big)_{L^2(\RR^7)}\right|\\
& \leq C T \|f_2\|_{L^\infty([0, T];\,
H^{3/2+\epsilon}_{\ell+3}(\RR^6))}\|g\|_{L^\infty([0, T];\,
L^2_{\ell} (\RR^6))} \|g\|_{L^2_\ell(\RR^7)}\\
& \leq C T^2 \|f_2\|_{L^\infty([0, T];\,
H^{3/2+\epsilon}_{\ell+3}(\RR^6))}\|f_1\|^2_{L^\infty([0, T];\,
H^3_{\ell} (\RR^6))}.
\end{align*}
For the second term in \eqref{5.3.6}, we shall prove the following claim:

 For $0<s<1/2$, one has
\begin{equation}\label{5.3.7}
\left|\Big(G,\, W^{2\ell}\, g\Big)_{L^2(\RR^7)}\right|\leq
C\Big(\|f_3\|_{L^\infty([0, T];\,\, H^{3}_{\ell}(\RR^6))}+
\|f_2\|^2_{L^\infty([0, T];\,\, H^{3}_{\ell}(\RR^6))}\Big)
\|\Lambda^s_v W^{\ell}\, g\|_{L^2(\RR^7_{t, x, v})}.
\end{equation}
In fact, recalling the expression \eqref{5.3.2}, it is easy to get
$$
\|B\|^2_{L^2_\ell(\RR^7)}+\|C\|^2_{L^2_\ell(\RR^7)}\leq C T
\|f_3\|^2_{L^\infty([0, T];\,\,H^3_{\ell+1}(\RR^6))}.
$$
For the term $A$, firstly recall that $\alpha_1+\alpha_2=\alpha$,
$|\alpha|\leq 3$ and $|\alpha_2|\leq 2$. In the following,
 we will apply Theorem \ref{theo2.1} with $m=-s$. We separate the discussion
 on $A$ into
two cases.

\noindent{\bf Case 1.} If $|\alpha_1|=1 $, we have
\begin{eqnarray*}
&&\int_{\RR_t}\int_{\RR^3_x}\|Q(\partial^{\alpha_1}
{f_2},\,\,\partial^{\alpha_2} f_1)(t, x,
\cdot)\|^2_{H^s_\ell(\RR^3_v)}dx dt\\
&\leq& C\int_{\RR_t}\int_{\RR^3_x}\|\partial^{\alpha_1} {f_2}(t, x,
\cdot)\|^2_{L^{1}_{\ell+2s}(\RR^3_v)} \|
\partial^{\alpha_2} f_1(t, x,
\cdot)\|^2_{H^{s}_{\ell+2s}(\RR^3_v)}dx dt\\
&\leq& C\|\partial^{\alpha_1} {f_2}\|^2_{L^\infty(\RR^4_{t, x};\,
L^{1}_{\ell+2s}(\RR^3_v))} \int_{\RR_t}\int_{\RR^3_x} \|
\partial^{\alpha_2} f_1(t, x,
\cdot)\|^2_{H^{s}_{\ell+2s}(\RR^3_v)}dx dt\\
&\leq& C T\|f_2\|^2_{L^\infty([0, T];\,\,H^3_{\ell+3}(\RR^6))}
\|f_1\|^2_{L^\infty([0, T];\,\,H^3_{\ell+2s}(\RR^6))}.
\end{eqnarray*}
\noindent{\bf Case 2.} If $|\alpha_1|\geq 2$, then $|\alpha_2|\leq
1$, it follows that
\begin{eqnarray*}
&&\int_{\RR_t}\int_{\RR^3_x}\|\partial^{\alpha_1} {f_2}(t, x,
\cdot)\|^2_{L^{1}_{\ell+2s}(\RR^3_v)} \|
\partial^{\alpha_2} f_1(t, x,
\cdot)\|^2_{H^{s}_{l+2s}(\RR^3_v)}dx dt\\
&\leq& C \|\partial^{\alpha_2} {f_1}\|^2_{L^\infty(\RR^4_{t,x}; \,
H^{s}_{l+2s}(\RR^3_v))} \int_{\RR_t}\int_{\RR^3_x} \|
\partial^{\alpha_1} f_2(t, x,
\cdot)\|^2_{L^{2}_{l+3/2+\delta+2s}(\RR^3_v)}dx dt\\
&\leq& C T\|f_1\|^2_{L^\infty([0,
T];\,\,H^{1+3/2+\epsilon}_{\ell+2s}(\RR^6))} \|f_2\|^2_{L^\infty([0,
T];\,\,H^3_{\ell+3}(\RR^6))}.
\end{eqnarray*}
By combining these two cases,  we have
\[
\left|\Big(G,\, W^{2\ell}\, g\Big)_{L^2(\RR^7)}\right|\leq C T\Big(
\|f_2\|^2_{L^\infty([0, T];\,\,H^3_{\ell+1}(\RR^6))}+
\|f_2\|^2_{L^\infty([0, T];\,\,H^3_{\ell+3}(\RR^6))} \|\Lambda^s_v
W^{\ell}\, g\|_{L^2(\RR^7_{t, x, v})} \Big).
\]
Therefore, we obtain
$$
\|\Lambda^{s}_v  W^{\ell}\, g\|^2_{L^2(\RR^7)}\leq C
\Big(1+\|{f_3}\|_{L^\infty([0, T];\,\, H^{3}_{\ell}(\RR^6))}\Big)^4.
$$
The proof of  the proposition  is then completed.
\end{proof}

\subsection{Gain of regularity in space variable}\label{section6.3}
\setcounter{equation}{0}
 With the gain of regularity in  the variable $v$ given in the
 above subsection, we now want to prove the gain of  regularity in the variable $x$. Here, the hypo-elliptic nature of the  equation will be used.

For this purpose, we introduce a more general framework. First of all, let us consider a transport equation in the form of
\begin{equation}\label{5.3.8}
 f_t + v\cdot\nabla_x f = g \in \cD '({\RR}^{7}) ,
\end{equation}
where $(t,x,v) \in  \RR^{7}$.

In
\cite{amuxy-nonlinear-b}, by using a generalized uncertainty
principle, we proved the following hypo-elliptic estimate.

\begin{lemm}\label{lemm5.2.1+0}
Assume that $ g \in H^{-s'} (\RR^{7})$, for some $0\leq s' <1$.
Let $f\in L^2 (\RR^{7}) $ be a weak solution of the transport
equation (\ref{5.3.8}), such that $\Lambda^s_v\, f \in L^2
(\RR^{7})$ for some $0<s\leq 1$. Then it follows that
\[
\Lambda_x^{s (1-s')/(s +1)}f\in L^2_{-\frac{s s'}{s +1}}(\RR^{7})
, \enskip \enskip \Lambda_t^{s (1-s')/(s +1)}f\in L^2_{-\frac{s}{s
+1}}(\RR^{7}),
\]
where $\Lambda_\bullet=(1-\Delta_\bullet)^{1/2}$.
\end{lemm}

Of course $g$ is typically linked with the Boltzmann collision operator. Through the above uncertainty principle, we have the following result on the gain of regularity in the
variable $x$

\begin{prop}\label{prop5.4.1}
Under the hypothesis of Theorem \ref{theo4.5.1}, one has
\begin{equation}\label{5.3.10}
\Lambda^{s_0}_{x}\, f_1\in L^2(\RR_t;\, H^3_\ell  (\RR^6)),
\end{equation}
for any $\ell \in\NN$ and $0<s_0=\frac{s(1-s)}{(s+1)}$.
\end{prop}

\begin{proof} For any $\ell\in\NN$, it follows {}from
Proposition \ref{prop5.3.1} that $
\Lambda^s_v W^\ell g \in L^2(\RR^7)$, while using the upper bound estimation given by Theorem \ref{theo2.1} with
$m=-s$, we get
$$
W^\ell Q( f_2,\,\, g)\in L^2(\RR^4_{t, x};\,\, H^{-s}(\RR^3_v)).
$$
On the other hand, \eqref{5.3.7} gives
$$
W^\ell G\in L^2(\RR^4_{t, x};\,\, H^{-s}(\RR^3_v)).
$$
By using equation (\ref{5.3.1}), it follows that
\begin{equation*}\label{5.3.11}
\partial_t(W_l g) + v\,\cdot\,\partial_x {(W_l g )} =W_l Q(f_2,\,\, g)+
W_l G \in H^{-s}(\RR^7).
\end{equation*}
Finally, by using Lemma \ref{lemm5.2.1+0} with $s'=s$, we can
conclude (\ref{5.3.10}) and this completes the proof of the
proposition.
\end{proof}

Therefore,  under the hypothesis $f\in L^\infty([0, T];\, H^3_\ell
(\RR^6))$ for all $\ell  \in\NN$, it follows that for any $ \ell
\in\NN$ we have
\begin{equation}\label{5.3.12}
\Lambda^{s}_{v}(\varphi(t)\psi(x) f) \, \in L^2(\RR_t;\, H^3_\ell
(\RR^6)),\,\, \hskip 0.5cm \Lambda^{s_0}_{x}(\varphi(t)\psi(x) f) \,
\in L^2(\RR_t;\, H^3_\ell  (\RR^6))\, .
\end{equation}

\smallbreak We now improve this partial regularity in the variable $x$. Since fractional derivatives will be involved, it is not surprising that a Leibniz type formula for fractional derivatives in the variable $x$ is needed. We shall use the following one, taken {}from \cite{amuxy-nonlinear-3}.
Let $0<\lambda<1$.
Then there exists a positive constant $C_\lambda\neq 0$ such that
\begin{align}\label{leibniz-x}
|D_{x}|^{\lambda}Q\big({f},\,\, g\big)&= Q\big(|D_{
x}|^{\lambda}{f},\,\, g\big)+ Q\big({f},\,\, |D_{
x}|^{\lambda}g\big)\\
&\qquad+C_\lambda\int_{\RR^3} |h|^{-3-\lambda}Q\big(f_h,\,\,
g_h\big) d h\,,\nonumber
\end{align}
with
$$
f_h(t, x, v)={f}(t, x, v)-{f}(t, x+h, v),\,\,\,\,\,\, h\in\RR^3_{
x}\, .
$$

With this preparation, we need a preliminary step, given by
\begin{prop}\label{prop5.4.2}
Let $0<\lambda< 1$ and $f\in L^\infty([0, T];\, \mathcal{H}^3_\ell (\RR^6))$
be a solution of (\ref{1.1}). Assume that, for all
$\ell \in\NN$, we have
\begin{equation*}\label{5.3.13}
\Lambda^{s}_{v} f_1 \, \in H^3_\ell(\RR^7),\,\,
\, \Lambda^{\lambda}_{x} f_1\in
H^3_\ell(\RR^7).
\end{equation*}
Then, one has for any $\ell  \in\NN$,
\begin{align*}%\label{5.3.14}
&\|\Lambda^{s}_v \Lambda^{\lambda}_{x}  f_1 \|_{
L^2(\RR_t;\, H^3_\ell  (\RR^6))}\leq C \|f_2\|_{L^\infty([0,
T];\,H^{3}_{\ell+3}(\RR^6))}\\
&\qquad\times\Big(\|\Lambda^{s}_v f_1 \|_{
H^3_{\ell+2s}(\RR^7)}+\|\Lambda^{\lambda}_{x}  f_1
\|_{ H^3_{\ell+2s}(\RR^7)}\Big)\, .%\nonumber
\end{align*}
\end{prop}
\begin{proof} Set $ g=\partial^\alpha_{x, v} f_1$ and $\alpha\in\NN^6, |\alpha|\leq 3$. We choose
$W^{\ell}\,|D_{x}|^{\lambda}\,\psi^2_2(x)\, |D_{x}|^{\lambda}\,
W^{\ell}\,g $ as a test function for equation (\ref{5.3.1}). Then
\begin{align*}%\label{5.3.16}
&\Big( \big(v\,\cdot\,\partial_x\psi_2(x)\big) |D_{x}|^{\lambda}\,
W^{\ell}\,g,\,\,\psi_2(x)\,
|D_{x}|^{\lambda}\, W^{\ell}\,g \Big)_{L^2(\RR^7)} \\
=&\Big(\psi_2(x)\, |D_{x}|^{\lambda}\, W^{\ell}\,Q(f_2,\,\,
g),\,\,\psi_2(x)\, |D_{x}|^{\lambda}\, W^{\ell}\,g
\Big)_{L^2(\RR^7)}\\%\nonumber\\
&\qquad+ \Big( \psi_2(x)\, |D_{x}|^{\lambda}\,
W^{\ell}\,G,\,\,\psi_2(x)\, |D_{x}|^{\lambda}\, W^{\ell}\,g
\Big)_{L^2(\RR^7)}. %\nonumber
\end{align*}
One has
$$
\Big| \Big( \big(v\,\cdot\,\partial_x\psi_2(x)\big)
|D_{x}|^{\lambda}\, W^{\ell}\,g,\,\,\psi_2(x)\, |D_{x}|^{\lambda}\,
W^{\ell}\,g \Big)_{L^2(\RR^7)}\Big| \leq C \|\,\Lambda_{x}^{\lambda}
\partial^\alpha f_1\|_{L^2_{\ell+1}(\RR^7)}\, ,
$$
and the same estimation for the linear term of $G$ in \eqref{5.3.2}
$$
\Big| \Big( \psi_2(x) |D_{x}|^{\lambda}\,
W^{\ell}\,\big(B+C\big),\,\,\psi_2(x)\, |D_{x}|^{\lambda}\,
W^{\ell}\,g \Big)_{L^2(\RR^7)}\Big| \leq C \|\,\Lambda_{x}^{\lambda}
\partial^\alpha f_1\|_{L^2_{\ell+1}(\RR^7)}\, .
$$
For the nonlinear terms of $G$ in \eqref{5.3.2},  we shall use  the
Leibniz formula \eqref{leibniz-x}. First of all, the coercivity
estimate \eqref{2.1.1} gives, as in (\ref{5.3.6}),
\begin{eqnarray*}
&&-\Big(Q({f_2},\, \psi_2(x)|D_{x}|^{\lambda} W^{\ell}\, g),\,\,
\psi_2(x)|D_{x}|^{\lambda} W^{\ell}\,
g\Big)_{L^2(\RR^7)} \\
&&\geq C_0\|\Lambda^{s}_v \psi_2(x)|D_{x}|^{\lambda} W^{\ell}\,
g\|^2_{L^2(\RR^7)}- C T\| {f_2}\|_{L^\infty(\RR^4_{t, x};\,\,
L^1(\RR^3_v))}\|\psi_2(x)|D_{x}|^{\lambda}W^{\ell}\,g\|^2_{L^\infty([0,
T];\, L^2(\RR^6))}\,.
\end{eqnarray*}
On the other hand, the upper bound estimate of Theorem \ref{theo2.1} with
$m=-s$ gives,
\begin{eqnarray*}
&&\Big|\Big(Q(|D_{x}|^{\lambda}f_2,\, \psi_1(x) W^{\ell}\, g),\,\,
\psi_1(x)|D_{x}|^{\lambda} W^{\ell}\, g\Big)_{L^2(\RR^7)}\Big|
\\
&\leq& C\||D_{x}|^{\lambda} f_2\|_{L^\infty(\RR^4_{t, x},\,
L^1_{2s}(\RR^3_v))}
 { \|\psi_2(x)\Lambda^{s}_v W^{\ell}\, g\|_{L^2_{2s}(\RR^7)}}
\|\psi_2(x)|D_{x}|^{\lambda}\Lambda^{s}_v W^{\ell}\,
g\|_{L^2(\RR^7)}\\
&\leq& C\| f_2\|_{L^\infty([0,
T];\,\,H^{3/2+\lambda+\epsilon}_{3/2+2s+\epsilon}(\RR^6))}
 { \|\psi_2(x)\Lambda^{s}_v W^{\ell}\, g\|_{L^2_{2s}(\RR^7)}}
\|\psi_2(x)|D_{x}|^{\lambda}\Lambda^{s}_v W^{\ell}\,
g\|_{L^2(\RR^7)}\\
&\leq& \delta \|\psi_2(x)|D_{x}|^{\lambda}\Lambda^{s}_v W^{\ell}\,
g\|^2_{L^2(\RR^7)} +C_\delta \| f_2\|^2_{L^\infty([0,
T];\,\,H^{3}_{3}(\RR^6))}
 { \|\Lambda^{s}_v g\|^2_{L^2_{\ell+2s}(\RR^7)}} .
\end{eqnarray*}
For the cross term coming {}from the decomposition
\eqref{leibniz-x}, by using again Theorem \ref{theo2.1} with $m=-s$,
we get
\begin{eqnarray*}
&&\int_{\RR^3} |h|^{-3-\lambda}\Big|\Big(Q((f_2)_h,\, (W^{\ell}\,
g)_h),\,\, \psi^2_2(x)|D_{x}|^{\lambda} W^{\ell}\,
g\Big)_{L^2(\RR^7)}d h\Big|
\\
&\leq& |C_\lambda| \|\psi_2(x)|D_{x}|^{\lambda}\Lambda^{s}_v
W^{\ell}\, g\|_{L^2(\RR^7)}\\
&&\,\,\,\,\,\,\,\,\,\times \int_{\RR^3} |h|^{-3-\lambda} \|
(f_2)_h\| _{L^\infty(\RR^4_{t, x},\, L^1_{2s}(\RR^3_v))}
\|\Lambda^{s}_v (W^{\ell}\, g)_h\|_{L^2_{2s}(\RR^7)}
 d h
\\
&\leq& \delta \|\psi_1(x)|D_{x}|^{\lambda}\Lambda^{s}_v W^{\ell}\,
g\|^2_{L^2(\RR^7)} +C_\delta \| f_2\|^2_{L^\infty([0,
T];\,\,H^{3}_{3}(\RR^6))}
 { \|\Lambda^{s}_v g\|^2_{L^2_{\ell+2s}(\RR^7)}}.
\end{eqnarray*}
Hence, we have
\begin{eqnarray*}
&&\Big|\Big(|D_{x}|^{\lambda}Q(f_2,\, \psi_2(x) W^{\ell}\, g)-
Q(|D_{x}|^{\lambda}f_2,\, \psi_2(x) W^{\ell}\, g),\,\,
\psi_2(x)|D_{x}|^{\lambda} W^{\ell}\, g\Big)_{L^2(\RR^7)}\Big|
\\
&\leq& \delta \|\psi_2(x)|D_{x}|^{\lambda}\Lambda^{s}_v W^{\ell}\,
g\|^2_{L^2(\RR^7)} +C_\delta \| f_2\|^2_{L^\infty([0,
T];\,\,H^{3}_{3}(\RR^6))}
 { \|\Lambda^{s}_v g\|^2_{L^2_{\ell+2s}(\RR^7)}} .
\end{eqnarray*}

In conclusion,  we get
\begin{eqnarray*}
&&\|\psi_2(x)|D_{x}|^{\lambda}\Lambda^{s}_v W^{\ell}\,
g\|^2_{L^2(\RR^7)}\\
&\leq& C \| f_2\|^2_{L^\infty([0, T];\,\,H^{3}_{3}(\RR^6))}\Big(
\|\,\,|D_{x}|^{\lambda} g\|^2_{L^2_{\ell+2s}(\RR^7)}+\|\Lambda^{s}_v
g\|^2_{L^2_{\ell+2s}(\RR^7)}
\Big)\nonumber\\
& +&\left|\left( |D_{x}|^{\lambda} \Big(W^{\ell}\, Q\big(f_2,\,\,
g\big)-Q\big(f_2,\,\,W^{\ell}\, g\big)\Big),\,\, \psi^2_2(x)|D_{
x}|^{\lambda}W^{\ell}\, g
\right)_{L^2(\RR^7)}\right| \nonumber \\
&+&\left|\left(|D_{x}|^{\lambda} W^{\ell}\,\, A,\,\,
\psi^2_2(x)|D_{x}|^{\lambda} \,
W^{\ell}\,g\right)_{L^2(\RR^7)}\right|
\nonumber\\
&=& {\rm I}+{\rm II}+{\rm III}\, . \nonumber
\end{eqnarray*}
For the term II,  again, formula (\ref{leibniz-x}) yields,
\begin{eqnarray*}
&&\left( |D_{x}|^{\lambda} \Big(W^{\ell}\, Q\big(f_2,\,\,
g\big)-Q\big(f_2,\,\,W^{\ell}\, g\big)\Big),\,\, \psi^2_2(x)|D_{
x}|^{\lambda}W^{\ell}\, g
\right)_{L^2(\RR^7)}\\
&&=\Big(\big( W^{\ell}\, Q\big(|D_{x}|^{\lambda}f_2,\,\,
g\big)-Q\big(|D_{x}|^{\lambda}f_2,\,\, W^{\ell}\, g\big)\big) ,\,\,
\psi^2_2(x)|D_{x}|^{\lambda} W^{\ell}\,g\Big)_{L^2(\RR^7)}\\
&&+\Big(\big( W^{\ell}\, Q\big(f_2,\,\, |D_{t, x}|^{\lambda}
g\big)-Q\big(f_2,\,\, W^{\ell}\, |D_{x}|^{\lambda} g\big)\big) ,\,\,
\psi^2_2(x)|D_{x}|^{\lambda} W^{\ell}\,g\Big)_{L^2(\RR^7)}
\\
&&+C_\lambda\int_{\RR^3} |h|^{-3-\lambda}\Big( \big( W^{\ell}\,
Q\big((f_2)_h,\,\, g_h\big)-Q\big((f_2)_h,\,\, W^{\ell}\,
g_h\big)\big) ,\,\, \psi^2_2(x)|D_{x}|^{\lambda}
W^{\ell}\,g\Big)_{L^2(\RR^7)}d h.
\end{eqnarray*}
Since $0<s<1/2$ , Lemma \ref{lemm3.2.1} implies
\begin{eqnarray*}
&&\Big|\Big(\big( W^{\ell}\, Q\big(|D_{x}|^{\lambda}f_2,\,\,
g\big)-Q\big(|D_{x}|^{\lambda}f_2,\,\, W^{\ell}\, g\big)\big) ,\,\,
\psi^2_2(x)|D_{x}|^{\lambda} W^{\ell}\,g\Big)_{L^2(\RR^7)}\Big|\\
&&\leq C\|\,|D_{x}|^{\lambda}f_2\|_{L^\infty(\RR^4_{t, x}\,, \,\,
L^1_{\ell}(\RR^3_v))} \|g\|_{L^2(\RR^4_{t, x}\,, \,\,
L^2_{\ell}(\RR^3_v))} \|\,|D_{x}|^{\lambda} g\|_{L^2_{\ell
}(\RR^7)}\\
&&\leq C\|f_2\|_{L^\infty ([0, T];\,
H^{3/2+\lambda+\epsilon}_{\ell+3/2+\epsilon}(\RR^6))}
\|g\|_{L^2_{\ell}(\RR^3)} \|\,|D_{x}|^{\lambda} g\|_{L^2_{\ell
}(\RR^7)},
\end{eqnarray*}
and
\begin{eqnarray*}
&&\Big|\Big(\big( W^{\ell}\, Q\big(f_2,\,\,
|D_{x}|^{\lambda}g\big)-Q\big(f_2,\,\, W^{\ell}\,|D_{x}|^{\lambda}
g\big)\big) ,\,\,
\psi^2_2(x)|D_{x}|^{\lambda} W^{\ell}\,g\Big)_{L^2(\RR^7)}\Big|\\
&&\leq C\|f_2\|_{L^\infty(\RR^4_{t, x}\,, \,\, L^1_{\ell}(\RR^3_v))}
\||D_{x}|^{\lambda}g\|_{L^2(\RR^4_{t, x}\,, \,\,
L^2_{\ell}(\RR^3_v))} \|\,|D_{x}|^{\lambda} g\|_{L^2_{\ell
}(\RR^7)}\\
&&\leq C\|f_2\|_{L^\infty ([0, T];\,
H^{1+3/2+\epsilon}_{\ell+3/2+\epsilon}(\RR^6))}
\||D_{x}|^{\lambda}g\|^2_{L^2_{\ell}(\RR^3)}\, .
\end{eqnarray*}
For the cross term,
\begin{eqnarray*}
&&\Big|\int_{\RR^3} |h|^{-3-\lambda}\Big( \big( W^{\ell}\,
Q\big((f_2)_h,\,\, g_h\big)-Q\big((f_2)_h,\,\, W^{\ell}\,
g_h\big)\big) ,\,\, \psi^2_2(x)|D_{x}|^{\lambda}
W^{\ell}\,g\Big)_{L^2(\RR^7)}d h\Big|\\
&&\leq C\|f_2\|_{L^\infty ([0, T];\,
H^{1+3/2+\epsilon}_{\ell+3/2+\epsilon}(\RR^6))}
\|g\|_{L^2_{\ell}(\RR^3)} \|\,|D_{x}|^{\lambda} g\|_{L^2_{\ell
}(\RR^7)}.
\end{eqnarray*}
Thus, we have
$$
{\rm II}\leq C\|f_2\|_{L^\infty ([0, T];\, H^{3}_{\ell+2}(\RR^6))}
\||D_{x}|^{\lambda}g\|^2_{L^2_{\ell}(\RR^3)}\, .
$$

We now consider  the last term $III$. Recall that $A$ stands for
the nonlinear terms {}from $G$
\begin{eqnarray*}
A=\sum_{\alpha_1+\alpha_2=\alpha\,\, \alpha_1\neq
0}C^{\alpha_1}_{\alpha_2} Q\Big(\partial^{\alpha_1} f_2,\,\,
\partial^{\alpha_2} f_1\Big).
\end{eqnarray*}
We have
\begin{align*}
&\left|\left(|D_{x}|^{\lambda} \, \Big(Q\big(\partial^{\alpha_1}
f_2,\,\,\partial^{\alpha_2} f_1\big)\Big),\,\, W^{\ell}\,
\psi^2_2(x)|D_{x}|^{\lambda} \,  W^{\ell}\,g\right)_{L^2(\RR^7)}\right|\\
&\leq C\|\Lambda^{s}_v  \psi_1(x)|D_{x}|^{\lambda}
W^{\ell}\,g\|_{L^2(\RR^7)}\times\\
&\times\Big\{\big\|Q\big(|D_{t, x}|^\lambda\partial^{\alpha_1}
f_2,\,\,
\partial^{\alpha_2} f_1\big)\big\|_{L^2(\RR^4_{t, x}; H^{-s}_{\ell}(\RR^3_v))}
+\big\|Q\big(\partial^{\alpha_1} f_2,\,\, |D_{x}|^\lambda
\partial^{\alpha_2} f_1\big)\big\|_{L^2(\RR^4_{t, x};
H^{-s}_{\ell}(\RR^3_v))}\\
&+\big\|\int h^{-3-\lambda}Q\big(\partial^{\alpha_1} (f_2)_h,\,\,
\partial^{\alpha_2}(f_1)_h\big) dh\big\|_{L^2(\RR^4_{t, x};
H^{m}_{l+|\gamma|/2}(\RR^3_v))}\Big\}.
\end{align*}

We divide the discussion into two cases.

\noindent{\bf Case 1.}  $|\alpha_1|=1$ (then $|\alpha_2|\leq 2$).
Applying \eqref{2.1.2} with $m=-s$. We have
\begin{align*}
&\Big\|Q\Big(|D_{x}|^\lambda\partial^{\alpha_1} f_2,\,\,
\partial^{\alpha_2} f_1\Big)\Big\|_{L^2(\RR^4_{t, x}; H^{-s}_{\ell}(\RR^3_v))}
\\
&\leq C\|\,\Lambda_{x}^{\lambda}
\partial^{\alpha_1} f_2\|_{L^2(\RR_t;\, (L^\infty(\RR^3_{x};
L^1_{\ell +2s}(\RR^3_v)))}\|\,\Lambda_v^{s}
\partial^{\alpha_2} f_1\|_{L^\infty(\RR_t;\, L^2_{\ell+2s}(\RR^6))}
\\
&\leq C\|\,\Lambda_{x}^{\lambda}f_2\|_{H^{1+3/2+
\epsilon}_{\ell+3/2+\epsilon+2s}(\RR^7)}\|\Lambda_v^s
f_1\|_{H^{2+1/2+\epsilon}_{\ell +2s}(\RR^7)},
\\
&\Big\|Q\Big(\partial^{\alpha_1} f_2,\,\, |D_{ x}|^\lambda
\partial^{\alpha_2} f_1\Big)\Big\|_{L^2(\RR^4_{t, x};
H^{-s}_{\ell}(\RR^3_v))}\\
&\leq C\|
\partial^{\alpha_1} f_2\|_{L^\infty(\RR^4_{t, x};
L^1_{\ell +2s}(\RR^3_v))}\|\,\Lambda_v^{s} \Lambda_{x}^{\lambda}
\partial^{\alpha_2} f_1\|_{L^2_{\ell+2s}(\RR^7)}
\\
&\leq C\|f_2\|_{L^\infty([0,
T];\,H^{1+3/2+\epsilon}_{\ell+3/2+\epsilon+2s}(\RR^6))}\|\Lambda_v^s
f_1\|_{H^{3}_{\ell +2s}(\RR^7)},
\end{align*}
and
\begin{eqnarray*}
&&\Big\|\int_{\RR^3}h^{-3-\lambda}Q\Big(\partial^{\alpha_1}
(f_2)_h,\,\,
\partial^{\alpha_2} (f_1)_h\Big) dh\Big\|_{L^2(\RR^4_{t, x};
H^{-s}_{\ell}(\RR^3_v))}\\
&&\leq C\int |h|^{-3-\lambda}\|\,\partial^{\alpha_1}
(f_2)_h\|_{L^\infty(\RR^4_{t, x};\, L^1_{\ell
+2s}(\RR^3_v))}\|\,\Lambda_v^{s}
\partial^{\alpha_2} (f_1)_h\|_{L^2_{\ell+2s}(\RR^7)}dh\\
&&\leq C\|\,\partial^{\alpha_1} f_2\|_{L^\infty(\RR^4_{t, x};
L^1_{\ell+2s}(\RR^3_v))}\|\,\Lambda_v^{s}
\partial^{\alpha_2} \nabla_{x}f_1\|_{L^2_{\ell +2s}(\RR^7)}
\\
&&\leq C\|f_2\|_{L^\infty([0,
T];\,H^{1+3/2+\epsilon}_{\ell+3/2+\epsilon+2s}(\RR^6))}\|\Lambda_v^s
f_1\|_{H^{3}_{\ell +2s}(\RR^7)}\,.
\end{eqnarray*}

\noindent{\bf Case 2.}  $|\alpha_1|\geq 2$. By the same argument as
above, one has
\begin{eqnarray*}
&&\Big\|Q\Big(|D_{x}|^\lambda\partial^{\alpha_1} f_2,\,\,
\partial^{\alpha_2} f_1\Big)\Big\|_{L^2(\RR^4_{t, x}; H^{-s}_{\ell}(\RR^3_v))}
+\Big\|Q\Big(\partial^{\alpha_1} f_2,\,\, |D_{x}|^\lambda
\partial^{\alpha_2} f_1\Big)\Big\|_{L^2(\RR^4_{t, x};
H^{-s}_{\ell}(\RR^3_v))}\\
&&\leq C\|f_2\|_{L^\infty([0,
T];\,H^{1+3/2+\epsilon}_{\ell+3/2+\epsilon+2s}(\RR^6))}\Big(\|\Lambda_v^s
f_1\|_{H^{3}_{\ell +2s}(\RR^7)}+\|\Lambda_x^\lambda
f_1\|_{H^{3}_{\ell +2s}(\RR^7)}\Big)\,.
\end{eqnarray*}
When $|\alpha_1|= 2$, we have
\begin{align*}
&\Big\|\int_{\RR^3}h^{-3-\lambda}Q\Big(\partial^{\alpha_1}
(f_2)_h,\,\,
\partial^{\alpha_2} (f_1)_h\Big) dh\Big\|_{L^2(\RR^4_{t, x};\,
H^{-s}_{\ell}(\RR^3_v))}\\
&\leq C\int |h|^{-3-\lambda}\|\,\partial^{\alpha_1}
(f_2)_h\|_{L^\infty(\RR_t;\, (L^2(\RR^3_{x};\, L^1_{\ell
+2s}(\RR^3_v)))}\|\,\Lambda_v^{s}
\partial^{\alpha_2} (f_1)_h\|_{L^2(\RR_t;\, (L^\infty(\RR^3_{x};\,
L^2_{\ell +2s}(\RR^3_v)))}dh\\
&\leq C\|\,\nabla_{x}\partial^{\alpha_1} f_2\|_{L^\infty(\RR_t;\,
(L^2(\RR^3_{x};\, L^1_{\ell +2s}(\RR^3_v)))}\|\,\Lambda_v^{s}
\partial^{\alpha_2} f_1\|_{L^2(\RR_t;\, (L^\infty(\RR^3_{x};\, L^2_{\ell +2s}(\RR^3_v)))}
\\
&\leq C\|f_2\|_{L^\infty([0, T];\, H^{3}_{\ell+3/2+\epsilon
+2s}(\RR^6))} \|\,\Lambda_v^{s}f_1\|_{H^{3}_{\ell +2s}(\RR^7)},
\end{align*}
while when $|\alpha_1|= |\alpha|=3$, we have
\begin{eqnarray*}
&&\Big\|\int_{\RR^3}h^{-3-\lambda}Q\Big(\partial^{\alpha}
(f_2)_h,\,\,
(f_1)_h\Big) dh\Big\|_{L^2(\RR^4_{t, x}; H^{-s}_{\ell}(\RR^3_v))}\\
&&\leq C\int |h|^{-3-\lambda}\|\,\partial^{\alpha}
(f_2)_h\|_{L^2(\RR^4_{t, x}; L^1_{\ell
+2s}(\RR^3_v))}\|\,\Lambda_v^{s}
(f_1)_h\|_{L^\infty(\RR^4_{t, x};\, L^2_{\ell+2s}(\RR^3_v))}dh\\
&&\leq C\|\,\partial^{\alpha} f_2\|_{L^2(\RR^4_{t, x};\,
L^1_{\ell+2s}(\RR^3_v))}\|\,\Lambda_v^{s} \nabla_{x}
f_1\|_{L^\infty(\RR^4_{t, x};\, L^2_{\ell+2s}(\RR^3_v))}
\\
&&\leq C\|f_2\|^2_{L^\infty([0, T];\,
H^3_{\ell+3/2+\epsilon+2s}(\RR^6))}.
\end{eqnarray*}
Thus, by the Cauchy-Schwarz inequality, we obtain
\begin{align*}
{\rm III}\leq\delta\|\Lambda^{s}_v \psi_1(x)|D_{ x}|^{\lambda}
W^\ell\,g\|^2_{L^2(\RR^7)}+ C_\delta \|f_2\|^2_{L^\infty([0,
T];\,H^{3}_{\ell+3}(\RR^6))}\Big(\|\Lambda_v^s f_1\|^2_{H^{3}_{\ell
+2s}(\RR^7)}+\|\Lambda_x^\lambda f_1\|^2_{H^{3}_{\ell
+2s}(\RR^7)}\Big).
\end{align*}
Finally, we get
$$
\|\Lambda^{s}_v \psi_1(x)|D_{x}|^{\lambda}
W^\ell\,g\|^2_{L^2(\RR^7)}\leq C_\delta \|f_2\|^2_{L^\infty([0,
T];\,H^{3}_{\ell+3}(\RR^6))}\Big(\|\Lambda_v^s f_1\|^2_{H^{3}_{\ell
+2s}(\RR^7)}+\|\Lambda_x^\lambda f_1\|^2_{H^{3}_{\ell
+2s}(\RR^7)}\Big),
$$
and the proof of the proposition is completed.
\end{proof}

\smallbreak We are now ready to show that the gain of at least order $1$
 regularity
in the variable $x$.
\begin{prop}\label{prop5.4.3}
Under the hypothesis of Theorem \ref{theo4.5.1}, one has
\begin{equation}\label{5.5.1}
\Lambda^{1+\varepsilon}_{x}\, (\varphi(t)\psi(x) f)\in H^3_\ell
(\RR^7),
\end{equation}
for any $\ell  \in\NN$ and some $\varepsilon>0$.
\end{prop}

\begin{proof}  By fixing $s_0=\frac{s(1-s)}{(s+1)}$, then
\eqref{5.3.12} and Proposition \ref{prop5.4.2} with $\lambda=s_0$
imply
$$
\Lambda^{s}_{v}\,\Lambda^{s_0}_{x} g\in H^{3}_\ell(\RR^7).
$$
It follows that,
\begin{equation*}\label{4.1+17}
 (\Lambda^{s_0}_{x} g)_t + v\,\cdot\,\partial_x {(\Lambda^{s_0}_{x} g)}
  =\Lambda^{s_0}_{x} Q(f_2,\,\, g)+ \Lambda^{s_0}_{x}  G\in
H^{-s}_\ell(\RR^7).
\end{equation*}
By applying Lemma \ref{lemm5.2.1+0} with $s'=s$, we can deduce that
\begin{equation*}
\Lambda^{s_0+s_0}_{x} (\varphi(t)\psi(x) f)\in H^3_\ell  (\RR^7),
\end{equation*}
for any $\ell  \in\NN$. If $2s_0<1$, by using again Proposition
\ref{prop5.4.2} with $\lambda=2s_0$ and Lemma \ref{lemm5.2.1+0} with
$s'=s$, we have
$$
\Lambda^{s}_{v} (\varphi(t)\psi(x) f),\,\,\Lambda^{2s_0}_{x}
(\varphi(t)\psi(x) f)\in H^3_\ell (\RR^7)\,\,\Rightarrow\,
\Lambda^{3s_0}_{x} (\varphi(t)\psi(x) f)\in H^3_\ell (\RR^7).
$$
Choose $k_0\in\NN$ such that
$$
k_0 s_0<1,\,\,\,\,\,\,\,\,\, (k_0+1) s_0=1+\varepsilon>1.
$$
Finally, (\ref{5.5.1}) follows {}from Proposition \ref{prop5.4.2} with
$\lambda=k_0s_0$ by induction. And
this completes the proof of the proposition.
\end{proof}

%%%%%%%%%%%%%%%%%%%%%%%%%%%%%%%%%%%%%%%%%%%%%%%%%%%%%%%%%%%%%%%%%%%%%%%
%%%%%%%%%%%%%%%%%%%%%%%%%%%%%%%%%%%%%%%%%%%%%%%%%%%%%%%%%%%%%%%%%%%%%%%
\subsection{Higher order regularity}\label{section6.4}
\setcounter{equation}{0}
\smallskip
The proof of Theorem \ref{theo5.3.1} will now be concluded with
the above preparation.

{}From Proposition \ref{prop5.3.1},
Proposition \ref{prop5.4.3} and equation \eqref{1.1}, it follows
that for any $\ell\in\NN$,
\begin{equation*}%\label{5.1}
\Lambda^{s}_{v}\, (\varphi(t)\psi(x) f),\,\,\,\,\,\,\, \nabla_{x} \,
(\varphi(t)\psi(x) f) \in H^3_\ell(\RR^7).
\end{equation*}

This fact will be used to show higher order regularity in
the variable $v$ by using the following

\begin{prop}\label{prop6.5.1}
Let $0<\lambda< 1$. Assume that, for any cutoff functions
$\varphi\in C^\infty_0(]T_1, T_2[), \psi\in C^\infty_0(\Omega)$ and
all $\ell\in\NN$,
\begin{equation*}%\label{5.2+0}
\Lambda^{\lambda}_{v}\, (\varphi(t)\psi(x) f),\,\,\,\,\,
\nabla_{x}\, (\varphi(t)\psi(x) f) \in H^3_\ell  (\RR^7).
\end{equation*}
Then, for any cutoff function and any $\ell\in\NN$,
\begin{equation*}%\label{5.2}
\Lambda^{\lambda+s}_v\, (\varphi(t)\psi(x) f)\in H^3_\ell(\RR^7).
\end{equation*}
\end{prop}

For the proof, we choose $W^{\ell}\,\Lambda^{2 \, \lambda}_v \, W^{\ell}\,g $ as a
test function for (\ref{5.3.1}), and then proceed as in the proof of Proposition \ref{prop5.4.2}. The only difference is in the
estimation on the  commutator with the convection:
$$
\Big|\Big( \big[\Lambda^{\lambda}_v, \,\, v\big]\,\cdot\,\partial_x
\, W^\ell\,g ,\,\,  \Lambda^{\lambda}_v W^\ell\,g\Big)_{L^2(\RR^7)}
\Big| \leq C
\|\Lambda^{\lambda}_{v}\,g\|_{L^2_\ell(\RR^7)}\|\nabla_x
\,g\|_{L^2_\ell(\RR^7)} ,
$$
since
$$
\big[\Lambda^{\lambda}_v, \,\, v\big]\,\cdot\,\partial_x =\lambda
\Lambda^{\lambda-2}_v \, \partial_v\,\cdot\,\partial_x,
$$
and $\Lambda^{\lambda-2}_v \, \partial_v$ are bounded operators in
$L^2$. This is the reason why we need in the first
step the gain of the regularity of order $1$ in the variable $x$.

To complete the proof of the full
regularization result, firstly, exactly as Proposition \ref{5.5.1}, we can get
\begin{equation*}%\label{5.8}
\Lambda^{1+\varepsilon}_{v} (\varphi(t)\psi(x) f)\in H^3_\ell
(\RR^7),
\end{equation*}
for any $\ell  \in\NN$ and some $\varepsilon>0$.

Therefore, we obtain that there exists $\epsilon>0$ such
that for any $\ell\in\NN$, and any cutoff functions $\varphi(t)$ and
$\psi(x)$,
$$
\varphi(t)\psi(x) f\in H^{4+\epsilon}_l(\RR^7).
$$
Notice that the estimate in Proposition \ref{prop5.4.2} can be
modified as follows. In fact, we can obtain
\begin{align*}%\label{5.3.14-1}
&\|\Lambda^{s}_v \Lambda^{\lambda}_{x}  f_1 \|_{
L^2(\RR_t;\, H^4_\ell  (\RR^6))}\leq C \| f_2 \|_{{H^{4+\epsilon}_{\ell+3}(\RR^7)}}\\
&\qquad\times\Big(\|\Lambda^{s}_v f_1 \|_{
H^4_{\ell+2s}(\RR^7)}+\|\Lambda^{\lambda}_{x}  f_1
\|_{ H^4_{\ell+2s}(\RR^7)}\Big)\, .%\nonumber
\end{align*}
By using this, the proof of
$$
f\in \cH^{4+\epsilon}_\ell(]T_1,
T_2[\times\Omega\times\RR^3_v),\,\,\forall
\ell\in\NN\,\,\Longrightarrow \,\, f\in \cH^{5}_\ell(]T_1,
T_2[\times\Omega\times\RR^3_v),\,\,\forall \ell\in\NN,
$$
 is direct and this completes  the proof of Theorem
\ref{theo5.3.1} by the bootstrapping argument.

%%%%%%%%%%%%%%%%%%%%%%%%%%%%%%%%%%%%%%%%%%%%%%%%%%%%%%%%%%%%%%%%%%%%%%
%%%%%%%%%%%%%%%%%%%%%%%%%%%%%%%%%%%%%%%%%%%%%%%%%%%%%%%%%%%%%%%%%%%%%%
\section{Global existence}\label{section7}
\smallbreak

We shall establish a global energy estimate for the Cauchy problem
\eqref{cauchy-problem}. For ease of exposition, and unless necessary, generic constants will be dropped out {}from the estimates in this section. Finally, all in all, we shall follow and adapt the method initiated by Guo \cite{guo-1} on the estimation on the
macroscopic components. Here we point out
that his method works also for the non-cutoff case 
but only when the estimations of the nonlinear and related terms are carried out
in terms of the non-isotorpic norms (2.2.1). We also note that his  method 
has been generalized to various  directions. Among them, a few are 
  the external force case \cite{duan,duan-ukai-yang}, the Vlassov-Maxwell-Boltzmann equation
\cite{strain}, the soft potential case \cite{strain-guo,h-yu} and the Landau equation \cite{guo,h-yu}.
Independently of his method which is based on the macro-micro decomposition near a global Maxwellian,
the energy method based on the macro-micro decomposition around a local Maxwellian
has
also been developed with application to the classical fluid dynamical equations 
\cite{liu-1,liu-2,liu-yang-yu-zhao}. Further references are found in 
these paper.

Introduce the macro-micro  decomposition near the absolute Maxwellian $\mu$:
\begin{align*}
g&=\pP g+(\iI-{\pP})g=g_1+g_2,
\quad
\pP g =g_1= \big(a+ b\cdot v+ c| v|^2 \big) \mu^{1/2},
\end{align*}

In this section, the following result on the energy estimate will be proved.
\begin{theo}\label{aprioriestimate}
For $N,\ell\ge 3$, let $T>0$ and suppose that
$g$ is a classical solution to the Cauchy problem \eqref{cauchy-problem} on $[0, T]$.
There exist constants $M_0, M_1>0$ such that
if
\[
\max_{0\le t\le T}\mathcal{E}(t)\le M_0,
\]
then  $g$ enjoys  the estimate
\[
\mathcal{E}(t)+
\int_0^t\mathcal{D}(\tau)d\tau\le M_1 \mathcal{E}(0),
\]
for any $t\in[0,T]$, where
\begin{equation*}
\mathcal{E}=\| g\|^2_{H^N_\ell(\RR^6)}\sim\|(a,b,c)\|_{H^N(\RR^3_x)}^2
+\| g_2\|^2_{H^N_\ell(\RR^6)},
\end{equation*}
is the  instant energy functional, and
\[
\mathcal{D}=\|\nabla_x(a,b,c)\|_{H^{N-1}(\RR^3_x)}^2
+||| g_2|||^2_{\cB^N_\ell(\RR^6)}
\]
the total dissipation rate.
\end{theo}

The proof of this theorem is divided into two parts, that is, the estimation
on the macroscopic component and the microscopic component respectively.

%%%%%%%%%%%%%%%%%%%%%%%%%%%%%%%%%%%%%%%%%%%%%%%%%%%%%%%%%%%%%%%%%%%%%
\subsection{Macroscopic Energy Estimate}\label{section7.1}
\setcounter{equation}{0}

By the macro-micro decomposition, the equation in
\eqref{cauchy-problem} is reduced to

\begin{align*}
\partial_t\Big\{ a + &b\cdot v+ c| v|^2 \Big\} \mu^{1/2}
+v\cdot\nabla_x\Big\{ a + b\cdot v+|v|^2 c\Big\} \mu^{1/2}
\\&\hspace*{2cm}=
-(\p_t+v\cdot \nabla_x)g_2+\mathcal{L} g_2+\Gamma(g,g).\notag
\end{align*}
Using
$$v\cdot\nabla_x b\cdot v=\sum_{i,j}v_iv_j\partial_i b_j
=\sum_{i}v_i^2 \partial_i b_i+\sum_{i>j}v_iv_j(\partial_i b_j+\partial_jb_i),
$$
we deduce
\begin{equation}\label{macroeq}
\left\{
\begin{array}{lrlrl}
\text{(i)}\quad& v_i | v|^2 \mu^{1/2} :&\quad&\nabla_x c &= -\partial_t r_c+l_c + h_c,
\\
\text{(ii)}\quad &v^2_i \mu^{1/2}:&\quad&\partial_t c +\partial_ib_i &= -\partial_t r_i+l_i + h_i ,
\\
\text{(iii)}\quad& v_iv_j \mu^{1/2}:& &\partial_ib_j + \partial_j b_i &
= -\partial_t r_{ij}+l_{ij} + h_{ij} , \quad i\neq j,
\\
\text{(iv)}\quad& v_i \mu^{1/2} :&&\partial_t b_i + \partial_i a &
= -\partial_t r_{bi}+l_{bi} + h_{bi},
\\
\text{(v)}\quad& \mu^{1/2} :&& \partial_t a &= -\partial_t r_a+l_a + h_a,
\end{array}
\right.
\end{equation}
where
\begin{align*}
&r=(g_2,e)_{L^2_v},\qquad l=-(v\cdot\nabla_x g_2+\mathcal{L} g_2,e)_{L^2_v},
\qquad h=(\Gamma(g,g),e)_{L^2(\RR^3_v)},
\end{align*}
stand for $r_c, \cdots, h_a$, while
\begin{align*}
 e\in \text{span}
\lbrace v_i | v|^2 \mu^{1/2} , v^2_i \mu^{1/2} ,
v_iv_j \mu^{1/2}, v_i \mu^{1/2} , \mu^{1/2} \rbrace.
\end{align*}
\begin{lemm}\label{abc2}
Let $\p^\alpha=\p^\alpha_x$, $\alpha\in\NN^3, |\alpha|\le m, m\ge 3$. Then,
\begin{equation*}%\label{dabc}
\|\p^\alpha (a,b,c)^2\|_{L^2(\RR^3_x)}\le \|\nabla_x (a,b,c)
\|_{H^{m-1}(\RR^3_{x})}\|(a,b,c)\|_{H^{m-1}(\RR^3_{x})}.
\end{equation*}
\end{lemm}

\begin{proof} Let $k=|\alpha|$. Then, one has for $k=0$ that
\begin{align*}
&\|(a,b,c)^2\|_{L^2(\RR^3_x)}\le
\|(a,b,c)\|_{L^6(\RR^3_x)}\|(a,b,c)\|_{L^3(\RR^3_x)}
\le \|\nabla_x (a,b,c)\|_{L^2(\RR^3_x)}\|(a,b,c)\|_{H^2(\RR^3_x)},
\\ \intertext{since}&
\|ab\|_{L^2(\RR^3_x)}\le \|a\|_{L^6(\RR^3_x)}\|b\|_{L^3(\RR^3_x)}\le
\|\nabla_x a\|_{L^2(\RR^3_x)}\|b\|_{L^\infty(\RR^3_x)}^{1/3}\|b\|_{L^2(\RR^3_x)}^{2/3}
\le \|\nabla_x a\|_{L^2(\RR^3_x)}\|b\|_{H^{2}(\RR^3_{x})}.
\end{align*}
Also for $k=1$, we have
\begin{align*}
&\|\p (a,b,c)^2\|_{L^2(\RR^3_x)}\le \|(a,b,c)\p (a,b,c)\|_{L^2(\RR^3_x)}
\le \|(a,b,c)\|_{L^\infty(\RR^3_x)}\|\nabla_x (a,b,c)\|_{L^2(\RR^3_x)}
\\&\hspace*{3cm}
\le\|(a,b,c)\|_{H^{m-1}(\RR^3_x)}\|\nabla_x (a,b,c)\|_{L^2(\RR^3_x)},
\end{align*}
and for $2\le k\le m$,
\begin{align*}
&\|\p^\alpha (a,b,c)^2\|_{L^2(\RR^3_x)}\le \sum_{k'\le \frac k 2}
\|\p_x^{k'}(a,b,c)\p_x^{k-k'}(a,b,c)\|_{L^2(\RR^3_x)}
\\&\le \sum\|\p_x^{k'}(a,b,c)\|_{L^\infty(\RR^3_x)}
\|\p_x^{k-k'} (a,b,c)\|_{L^2(\RR^3_x)}
\le \|(a,b,c)\|_{H^{m-1}(\RR^3_x)}
\|\nabla_x (a,b,c)\|_{H^{m-1}(\RR^3_x)}.
\end{align*}
And this completes the proof of the lemma.
\end{proof}

\begin{lemm} \label{rlh}\text{\bf (Estimate of $r,l,h$).}

Let $\p^\alpha=\p^\alpha_x  , \p_i=\p_{x_i}$, $|\alpha|\le N-1, N\ge 3$. Then, one has
\begin{align}\label{Dr}
&\|\p_i\p^\alpha r\,\|_{L^2(\RR^3_x)}+\|\p^\alpha   l\,\|_{L^2_{x}}
\le
\|g_2\|_{H^N(\RR^3_x,\, L^2(\RR^3_v))} \equiv A_1,
\\&\label{Dh}
\|\p^\alpha   h\|_{L^2(\RR^3_x)}
\le \|\nabla_x (a,b,c)\|_{H^{N-2}(\RR^3_x)}\|(a,b,c)\|_{H^{N-1}(\RR^3_x)}
\\  &\quad+\|  (a,b,c)\|_{H^{N-1}(\RR^3_x)}\|g_2\|_{H^{N-1}(\RR^3_x,\, L^2(\RR^3_v))}
+
\|g_2\|_{H^{N-1}(\RR^3_x,\, L^2(\RR^3_v))}^2\equiv A_2. \notag
\end{align}
\end{lemm}
\noindent
\begin{proof}
\eqref{Dr} follows {}from
\[
\|\p_i\p^\alpha r\,\|_{L^2(\RR^3_x)}
\le
\| (\p_i\p^\alpha g_2, e)_{L^2(\RR^3_v)}\|_{L^2(\RR^3_x)}
\le
\|\p_i\p^\alpha  g_2\|_{L^2(\RR^6_{x, v})},
\]
and
\[
\|\p^\alpha   l\,\|_{L^2_{x}}\le
\| (\nabla_x\p^\alpha g_2, ve)_{L^2(\RR^3_v)}+(\p^\alpha g_2, \mathcal{L}^*e)_{L^2(\RR^3_v)}\|_{L^2(\RR^3_x)}
\le
\| \p^{\alpha}_xg_2\|_{H^{1}(\RR^3_{x},\, L^2(\RR^3_v))}.
\]

We prove \eqref{Dh} as follows.
\begin{align*}
 h&=\iiint b(\cos\theta)\mu_*^{1/2}(g'_*g'-g_*g)edvdv_*d\sigma
\\&=\iiint b(\cos\theta)g g_*\Big((\mu^{1/2})_*'e'-\mu_*^{1/2}e\Big)dvdv_*d\sigma
\\&=\iiint b(\cos\theta)(\mu^{1/2}g)(\mu^{1/2} g)_*\Big(q(v')-q(v)\Big)dvdv_*d\sigma
\\&\equiv \Phi(g,g)=\sum_{i,j=1}^2\Phi(g_i,g_j)=\sum_{i,j=1}^2\Phi^{(ij)}, %h_1+h_2+h_3+h_4,
\end{align*}
where $q=q(v)$ is some polynomial of $v$.
Firstly, we have
\[
\Phi^{(11)}=\sum_{\eta_j, \eta_k\in \{a,b,c\}} \eta_j\eta_k\Phi(\psi_j,\psi_k),
\]
where $
\psi_j(v)\in \mathcal{N}$.
Clearly, $|\Phi(\psi_j,\psi_k)|<\infty$, so that by virtue of Lemma \ref{abc2},
\begin{align*}
 \|\p^\alpha  \Phi^{(11)}\|_{L^2(\RR^3_x)}\le \|\p^\alpha  (a,b,c)^2\|_{L^2(\RR^3_x)}
 \le \|\nabla_x (a,b,c)\|_{H^{N-2}(\RR^3_x)}\|(a,b,c)\|_{H^{N-1}(\RR^3_x)}.
\end{align*}
On the other hand,
\[
\|\Phi(g,f)\|_{L^2(\RR^3_x)}\le \|\mu^{1/2}g\|_{{L^2(\RR^3_x; L^1_3(\RR_v^3))}}
\|\mu^{1/2}f\||_{{L^2(\RR^3_x; L^1_3(\RR_v^3))}}
\le \|g\|_{L^2(\RR^6_{x, v})}\|f\|_{L^2(\RR^6_{x, v})}
\]
which yields for $|\alpha|\le N-1$,
\begin{align*}
 &
  \|\p^\alpha   \Phi^{(12)}\|_{L^2(\RR^3_x)}\le \|\p^\alpha  (a,b,c)\|_{L^2(\RR^3_x)}
  \|\p^\alpha  g_2\|_{L^2(\RR^6_{x, v})}
  \le \|  (a,b,c)\|_{H^{N-1}(\RR^3_x)}\|g_2\|_{H^{N-1}(\RR^3_x, L^2(\RR^3_v))},
 \\
 & \|\p^\alpha   \Phi^{(21)}\|_{L^2(\RR^3_x)}\le \|\p^\alpha  g_2\|_{L^2(\RR^6_{x, v})}
 \|\p^\alpha  (a,b,c)\|_{L^2(\RR^3_x)}
\le \|g_2\|_{H^{N-1}(\RR^3_x, L^2(\RR^3_v))}\|  (a,b,c)\|_{H^{N-1}(\RR^3_x)},
 \\
 &\|\p^\alpha   \Phi^{(22)}\|_{L^2(\RR^3_x)}\le \|\p^\alpha  g_2\|_{L^2(\RR^6_{x, v})}^2
\le \|g_2\|_{H^{N-1}(\RR^3_x, L^2(\RR^3_{v}))}^2.
 \end{align*}
Now the proof of \eqref{Dh} is completed.
\end{proof}

\begin{lemm}\text{\bf (Macro-energy estimate)}
Let $|\alpha|\le N-1$.
\begin{align}\label{pabc}
\|\nabla_x \p^\alpha (a, b, c)\|_{L^2(\RR^3_x)}^2
\le
-\frac{d}{dt}\Big\{(\p^\alpha r,&\nabla_x \p^\alpha (a, - b, c))_{L^2(\RR^3_x)}
+(\p^\alpha b, \nabla_x \p^\alpha a)_{L^2(\RR^3_x)}\Big\}
\\&\quad +\notag
 \|g_2\|_{H^{N}(\RR^3_x, L^2(\RR^3_v))}^2+
\mathcal D_1\mathcal E_1\, ,
\end{align}
where
\[
\mathcal{D}_1=
\|\nabla_x(a,b,c)\|_{H^{N-1}(\RR^3_x)}^2+\|g_2\|_{H^{N}(\RR^3_x, L^2(\RR^3_v))}^2
\]
is a dissipation rate and
\[
\mathcal{E}_1=
\|(a,b,c)\|^2_{H^{N-1}(\RR^3_x)}+\|g_2\|_{H^{N-1}(\RR^3_x, L^2(\RR^3_v))}^2=\|g\|_{H^{N-1}(\RR^3_x, L^2(\RR^3_v))}^2
\]
is an instant energy functional.
\end{lemm}
\noindent
\begin{proof}
(a) Estimate of $\nabla_x\p^\alpha  a$. Let $A_1,A_2$ be as in Lemma \ref{rlh}.
{}From \eqref{macroeq} (iv),
\begin{align*}
 \|\nabla_x \p^\alpha a\|_{L^2(\RR^3_x)}^2
&=(\nabla_x \p^\alpha a,\nabla_x \p^\alpha a)_{L^2(\RR^3_x)}
\\&= (\p^\alpha (-\p_t b-\p_tr+l+h),\nabla_x \p^\alpha a)_{L^2(\RR^3_x)}
\\&\le R_1+C_\eta (A_1^2+A_2^2)+\eta \|\nabla_x \p^\alpha a\|_{L^2(\RR^3_x)}^2.
\end{align*}
\begin{align*}
R_1&=-(\p^\alpha \p_tb+\p^\alpha \p_tr,\nabla_x \p^\alpha a)_{L^2(\RR^3_x)}
\\&=-\frac{d}{dt}(\p^\alpha (b+r),\nabla_x \p^\alpha a)_{L^2(\RR^3_x)}+
(\nabla_x\p^\alpha (b+r), \p_t \p^\alpha a)_{L^2(\RR^3_x)}
\\&
\le
-\frac{d}{dt}(\p^\alpha (b+r),\nabla_x \p^\alpha a)_{L^2(\RR^3_x)}+
C_\eta ( \|\nabla_x\p^\alpha  b\|_{L^2(\RR^3_x)}^2+A_1^2)
+\eta
\|\p_t\p^\alpha  a\|^2_{L^2(\RR^3_x)},
\end{align*}
(b) Estimate of $\nabla_x\p^\alpha  b$. {}From \eqref{macroeq} (iii) and (ii),
\begin{align*}
\Delta_x&\p^\alpha b_i+\p^2_{i}\p^\alpha   b_i=
{\sum_{j \ne i} \p_j \p^{\alpha}( \p_j b_i + \p_i b_j)
+ \p_i \p^\alpha( 2 \p_i b_i - \sum_{j\ne i} \p_i b_j)
}
\\
&\hskip2cm =
\p_i\p^\alpha  (-\p_t r +l+h),
\\
&\|\nabla_x\p^\alpha  b\|_{L^2(\RR^3_x)}^2+\|\p_i\p^\alpha  b\|_{L^2(\RR^3_x)}^2=
-(\Delta_x\p^\alpha b_i+\p^2_{i}\p^\alpha   b_i, \p^\alpha  b)_{L^2(\RR^3_x)}
= R_2+R_3+R_4,
\end{align*}
where
\begin{align*}
R_2&=(\p_t\p^\alpha   r,\p_i\p^\alpha  b)_{L^2(\RR^3_x)}
=\frac{d}{dt}(\p^\alpha   r,\p_i\p^\alpha  b)_{L^2(\RR^3_x)}+
(\p_i\p^\alpha   r,\p_t\p^\alpha  b)_{L^2(\RR^3_x)}
\\&\hspace*{1cm}\le
\frac{d}{dt}(\p^\alpha   r,\p_i\p^\alpha  b)_{L^2(\RR^3_x)}+
C_\eta A_1^2
+\eta
\|\p_t\p^\alpha  b\|^2_{L^2(\RR^3_x)},
\\
R_3&=-(\p^\alpha   l,\p_i\p^\alpha  b)_{L^2(\RR^3_x)}\le
C_\eta A_1^2%\|\p^\alpha   l\|^2_{L^2_{x}}
+
\eta
\|\p_i\p^\alpha  b\|^2_{L^2(\RR^3_x)},
\\R_4&=-(\p^\alpha   h,\p_i\p^\alpha  b)_{L^2(\RR^3_x)}\le
C_\eta A_2^{{2}}
+
\eta
\|\p_i\p^\alpha  b\|^2_{L^2(\RR^3_x)}.
\end{align*}
(c) Estimate of $\nabla_x\p^\alpha  c$. {}From \eqref{macroeq} (i),
 \begin{align*}
 \|\nabla_x \p^\alpha &c\|_{L^2(\RR^3_x)}^2
=(\nabla_x \p^\alpha c,\nabla_x \p^\alpha c)_{L^2(\RR^3_x)}
= (\p^\alpha (-\p_tr+l+h),\nabla_x \p^\alpha c)_{L^2(\RR^3_x)}
\\&
\le R_5 %(-\p_t\p^\alpha r,\nabla_x \p^\alpha c)
+C_\eta(A_1^2+A_2^2)+\eta \|\nabla_x \p^\alpha c\|_{L^2(\RR^3_x)}^2
\end{align*}
where
\begin{align*}
R_5&=-(\p^\alpha \p_tr,\nabla_x \p^\alpha c)_{L^2(\RR^3_x)}=
-\frac{d}{dt}(\p^\alpha r,\nabla_x \p^\alpha c)_{L^2(\RR^3_x)}+
(\nabla_x\p^\alpha r, \p_t \p^\alpha c)_{L^2(\RR^3_x)}
\\&\quad
\le
-\frac{d}{dt}(\p^\alpha r,\nabla_x \p^\alpha c)_{L^2(\RR^3_x)}+
C_\eta A_1^2 % \|\nabla_x\p^\alpha  (I-P)\|^2_{L^2(\RR^6_{x, v})}
+\eta
\|\p_t\p^\alpha  c\|^2_{L^2(\RR^3_x)},
\end{align*}
(d) Estimate of $\p_t\p^\alpha  (a,b,c)$. \
\begin{align}\label{ptabc}
\|\p_t\p^\alpha & (a,b,c)\|_{L^2(\RR^3_x)}=\|\p^\alpha \p_t \pP g\|_{L^2(\RR^6_{x, v})}
\\& \notag
=\|\p^\alpha  \pP\Big(-v\cdot\nabla_xg-\mathcal{L}g+\Gamma(g,g)\Big)\|_{L^2(\RR^6_{x, v})}
\\&=\|\p^\alpha  \pP(v\cdot\nabla_xg)\|_{L^2(\RR^6_{x, v})}
\le \|\nabla_x\p^\alpha  (a,b,c)\|_{L^2(\RR^3_x)}+
\|\nabla_x\p^\alpha g_2\|_{L^2(\RR^6_{x, v})}.\notag
\end{align}
Combining all the above estimates and
taking $\eta>0$ sufficiently small, we deduce
\begin{align*}
\|\nabla_x \p^\alpha (a, b,  c)\|_{L^2(\RR^3_x)}^2
&\le
-\frac{d}{dt}\Big\{(\p^\alpha r,\nabla_x \p^\alpha (a, -b,c))_{L^2(\RR^3_x)}
+(\p^\alpha b, \nabla_x \p^\alpha a)_{L^2(\RR^3_x)}\Big\}
\\&\quad +\notag
A_1^2+A_2^2+
 \eta\|\nabla_x\p^\alpha g_2\|_{L^2(\RR^6_{x, v})}^2.
\end{align*}
Finally, choosing $|\alpha|\le N-1$, and using Lemma \eqref{rlh}, we obtain
\begin{align*}
A_1^2&+A_2^2+
 \eta\|\nabla_x\p^\alpha g_2\|_{L^2(\RR^6_{x, v})}^2
%\\&
\le \mathcal D_1\mathcal E_1 +\eta \|g_2\|_{H^{N}_x(L^2(\RR^3_v))}^2,
\end{align*}
which  completes the proof of Lemma \ref{pabc}.
\end{proof}

%%%%%%%%%%%%%%%%%%%%%%%%%%%%%%%%%%%%%%%%%%%%%%%%%%%%%%%%%%%
\subsection{Microscopic Energy Estimate}\label{section7.2}
\setcounter{equation}{0}
In this subsection,
We shall use Lemma \ref{lemm2.2.3} and Proposition \ref{prop3.2.3} to
estimate the microscopic component.

\noindent
{\bf Step 1.}
 Let $\alpha\in\NN^3$, $|\alpha|\le N$, and apply $\p^\alpha=\p^\alpha_x$ to \eqref{cauchy-problem} to have,
\begin{equation*}%\label{equation1}
\p_t(\p^\alpha g)+ v\cdot\nabla_x (\p^\alpha g)
+ \cL(\p^\alpha g)=\p^\alpha\Gamma(g,\,
g),
\end{equation*}
and take the $L^2(\RR^6_{x, v})$ inner product with $\p^\alpha g$. By Lemma \ref{lemm2.2.3}, we have
\begin{equation}\label{microenergy1}
\frac 12 \frac{d}{dt}\|\p^\alpha g\|^2_{L^2(\RR^6_{x, v})}
+D_1\le J,
\end{equation}
where $D_1$ is a dissipation rate
\begin{align*}
D_1&=\int_{\RR^3}||| \p^\alpha g_2 |||^2dx=||| \p^\alpha g_2 |||^2_{\cB^0_0(\RR^6)},
\end{align*}
and $J$ is given by
\begin{align*}
J&=(\p^\alpha\Gamma(g,\,g), \p^\alpha g)_{L^2(\RR^6)}
=\sum_{i,j=1}^2(\p^\alpha\Gamma(g_i,\,g_j), \p^\alpha g_2)_{L^2(\RR^6)}
\\&=\sum_{i,j=1}^2 J^{(ij)}.
\end{align*}

First, consider $J^{(11)}$. We have, with $\psi_j\in\mathcal{N}$,
\begin{align}\label{J11}
|J^{(11)}|&
\le
\|\p^\alpha\Gamma(g_1,\,g_1)\|_{L^2(\RR^6)}    \|\p^\alpha g_2\|_{L^2(\RR^6)}
\\&
\notag
\hspace*{2cm}\le \|\Big(\|\p^\alpha \Gamma(g_1,\,g_1)\|_{L^2(\RR^3_v)}
\Big)\|_{L^2(\RR^3_x)}    \|\p^\alpha g_2\|_{L^2(\RR^6)},
\\ \notag
&\| \p^\alpha\Gamma(g_1,\,g_1)\|_{L^2(\RR^3_v)}\le
\sum_{\eta_j,\eta_k\in \{a,b,c\}} |\p^\alpha(\eta_j\eta_k)|
\|\Gamma(\psi_j,\psi_k)\|_{L^2(\RR^3_v)},
\\ \notag
&\|\Gamma(\psi_j,\psi_k)\|_{L^2(\RR^3_v)}^2=\int\Big(\int\int b(\cos\theta)\mu_*^{1/2}
\{ (\psi_j)_*' \psi_k'-(\psi_j)_*\psi_k\}dv_*d\sigma\Big)^2dv \notag
\\&\qquad \notag
=\int\mu\Big(\int\int b(\cos\theta)\mu_*
\{ (p_j)_*' p_k'-(p_j)_*p_k\}dv_*d\sigma\Big)^2dv<\infty,
\end{align}
where $p_j\in\{1,v,|v|^2\}$.
Consequently, by virtue of Lemma \ref{abc2},
\begin{align*}
|J^{(11)}|&
\le \|\p^\alpha (a,b,c)^2\|_{L^2(\RR^3_x)} \|\p^\alpha g_2\|_{L^2(\RR^6)}
\\&
 \le
\|\nabla_x (a,b,c)\|_{H^{N-2}(\RR^3_x)}\|(a,b,c)\|_{H^{N-1}(\RR^3_x)}\|g_2\|_{H^N(\RR^3_x, L^2(\RR^3_v))}
\\
&\le \|(a,b,c)\|_{H^{N-1}(\RR^3_x)}\Big(\|\nabla_x (a,b,c)\|_{H^{N-2}(\RR^3_x)}^2+
|||g_2|||_{\cB^{N}_0(\RR^6)}^2\Big).
\end{align*}

On the other hand, using Proposition \ref{prop3.2.3} gives
\begin{align*}
|J^{(12)}|&\le \|g_1\|_{H^N_0(\RR^6)}|||g_2|||_{\cB^{N}_0(\RR^6)}
|||g_2|||_{\cB^{N}_0(\RR^6)}
\\&\quad\quad\leq ||(a,b,c)||_{H^N(\RR^3_x)}\,\,
||| g_2|||_{\cB^{N}_0(\RR^6)}^2,
\\
|J^{(21)}|&\le \|g_2\|_{H^N_0(\RR^6)}|||g_1|||_{\cB^{N}_0(\RR^6)}
|||g_2|||_{\cB^{N}_0(\RR^6)}
\\&\quad\quad\leq  ||g_2||_{H^N_0(\RR^6)}||(a,b,c)||_{H^N(\RR^3_x)}\,\,
||| g_2|||_{\cB^{N}_0(\RR^6)},
\\
|J^{(22)}|&\le \|g_2\|_{H^N_0(\RR^6)}|||g_2|||_{\cB^{N}_0(\RR^6)}
|||g_2|||_{\cB^{N}_0(\RR^6)}
\\&\quad\quad\leq ||g_2||_{H^N_0(\RR^6)}\,\,
||| g_2|||_{\cB^{N}_0(\RR^6)}^2.
\end{align*}
Taking the summation of \eqref{microenergy1} for $|\alpha|\le N, N\ge 3$, we have

\begin{lemm} Let $ N\ge 3$. Then,
\begin{equation}\label{microenergy11}
\frac{d}{dt}\| g\|^2_{H^N(\RR^3_x, L^2(\RR^3_v))}+
|||  g_2 |||^2_{\cB^N_0(\RR^6)}\le
\mathcal{D}_2\mathcal{E}_2^{{1/2}},
\end{equation}
where
\begin{align*}
\mathcal{D}_2&=
\|\nabla_x(a,b,c)\|_{H^{N-1}(\RR^3_x)}^2+|||g_2|||_{\cB^{N}_0(\RR^6)}^2,
\\
\mathcal{E}_2&=\|g\|_{H^{N}(\RR^3_x, L^2(\RR^3_v))}^2\sim
\|(a,b,c)\|^2_{H^{N}(\RR^3_x)}+\|g_2\|_{H^{N}(\RR^3_x, L^2(\RR^3_v))}^2.
\end{align*}
\vspace{0.5cm}
\end{lemm}
\noindent
{\bf Step 2.} Let
$
\p^\alpha=\p^\alpha_x, 1\le |\alpha|\le N, N\ge 3,
$
and apply $W^\ell\p^\alpha$ to \eqref{cauchy-problem}.
We have
\begin{equation}\label{g1-cauchy-problem}
 \p_t (W^\ell\p^\alpha g)+ v\cdot\nabla_x
(W^\ell\p^\alpha g) + \cL (W^\ell\p^\alpha g)=S_1+S_2,
\end{equation}
where
\[
S_1=W^\ell\p^\alpha\Gamma(g,\, g),
\qquad
S_2=[\cL,\, W^\ell ](\p^\alpha g).
\]
Take the $L^2(\RR^6_{x, v})$ inner product with $W^\ell\p^\alpha g$ to deduce
\begin{equation*}%\label{energy1}
\frac 12 \frac{d}{dt}\|W^\ell\p^\alpha g\|^2_{L^2(\RR^6)}+D_2\le G,
\end{equation*}
where $D_2$ is a dissipation rate
\begin{align*}
 D_2&=\int_{\RR^3}||| (\iI-\pP)W^\ell\p^\alpha g |||^2dx
\\&\ge \frac 12||| \p^\alpha g |||^2_{\cB^{0}_\ell(\RR^6)}-
(\|\nabla_x (a,b,c)\|_{H^{N-1}(\RR^3_x)}+|| g_2 ||^2_{H^N(\RR^3_x, L^2(\RR^3_v))}).
\end{align*}
Here we have used for $1\le |\alpha|\le N$,
\begin{align*}
\int_{\RR^3}&|||\pP W^\ell\p^\alpha g|||^2 dx
\le ||\p^\alpha g||_{L^2(\RR^6_{x, v})}^2
\\&\le
 \|\nabla_x (a,b,c)\|^2_{H^{N-1}(\RR^3_x)}+||g_2 ||^2_{H^N(\RR^3_x, L^2(\RR^3_v))}.
\end{align*}
On the other hand, $G$ is defined by
\[
G=G_1+G_2, \quad
G_i=(S_i, W^\ell\p^\alpha g)_{L^2(\RR^6)}, \qquad i=1,2.
\]
The estimation on $G_1$ and $G_2$ will be given in the following lemmas.
\begin{lemm}\label{G1}
Let $N\ge 3, \ell\ge 3$. Then, for $\mathcal{E}$ and $\mathcal{D}$ defined in   Theorem \ref{aprioriestimate}, we have
\begin{align*}
G_1\le \mathcal{E}^{1/2} \mathcal{D},
\end{align*}
\end{lemm}
\begin{proof} First, write
\begin{align*}
G_1&=\sum_{i,j=1,2}(W^\ell\p^\alpha \Gamma(g_i,\,g_j), W^\ell\p^\alpha g)_{L^2(\RR^6)}
=\sum_{i,j=1,2}G_1^{(ij)}.
\end{align*}
Proceeding as in \eqref{J11},
\begin{align}\label{G11}
|G^{(11)}|&\le
\|W^{2\ell}\p^\alpha\Gamma(g_1,\,g_1)\|_{L^2(\RR^6)}    \|\p^\alpha g\|_{L^2(\RR^6)}
\\&
\notag
\le \|\Big(\|W^{2\ell}\p^\alpha \Gamma(g_1,\,g_1)\|_{L^2(\RR^3_v)}\Big)\|_{L^2(\RR^3_x)}
\|\p^\alpha g\|_{L^2(\RR^6)},
\\& \notag
\sim \|\p^\alpha\{(a,b,c)^2\}\|_{L^2(\RR^3_x)}\|W^{2\ell}\Gamma(\psi_j,\psi_k)\|_{L^2(\RR^3_v)}
\|\p^\alpha g\|_{L^2(\RR^6)},
\\ \notag
\|W^{2\ell}\Gamma&(\psi_j,\psi_k)\|_{L^2(\RR^3_v)}^2=
\int\Big(W^{2\ell}\int\int b(\cos\theta)\mu_*^{1/2}
\{ (\psi_j)_*' \psi_k'-(\psi_j)_*\psi_k\}dv_*d\sigma\Big)^2dv \notag
\\&\qquad \notag
=\int\mu W^{4\ell}\Big(\int\int b(\cos\theta)\mu_*
\{ (p_j)_*' p_k'-(p_j)_*p_k\}dv_*d\sigma\Big)^2dv<\infty.
\end{align}
Since $1\le|\alpha|\le N$,
\begin{align*}%\label{G111}
|G^{(11)}|\le \|&\nabla_x(a,b,c)\|_{H^{N-1}(\RR^3_x)}\|(a,b,c)\|_{H^{N-1}(\RR^3_x)}
\\&\times %\notag
(\|\nabla_x(a,b,c)\|_{H^{N-1}(\RR^3_x)}+|| g_2 ||_{H^N(\RR^3_x, L^2(\RR^3_v))})\,.
\end{align*}
On the other hand, we have
\begin{align*}
||| g_1|||_{\cB^{N}_\ell(\RR^6)}^2&=
\sum_{|\beta|\le N}\int_{\RR^3_x} |||W^\ell
\p^\beta_{x,v} g_1(x,\, \cdot\,)|||^2dx
\le \|(a,b,c)\|^2_{H^{N}(\RR^3_x)}.
\end{align*}
This fact and  Proposition \ref{prop3.2.3} yield
 \begin{align*}
|G_1^{(12)}|\le &||g_1||_{H^N_\ell(\RR^6)}\,\,
||| g_2|||_{\cB^{N}_\ell(\RR^6)}\,\, ||| W^\ell\p^\alpha g|||_{\cB^0_0(\RR^6)}
\\&
 \leq  ||(a,b,c)||_{H^N(\RR^3_x)}\,\,
 ||| g_2|||_{\cB^{N}_\ell(\RR^6)}||| \p^\alpha g|||_{\cB^0_\ell(\RR^6)},
\\
|G_1^{(21)}|\le&||g_2||_{H^N_\ell(\RR^6)}\,\,
||| g_1|||_{\cB^{N}_\ell(\RR^6)}\,\, ||| W^\ell\p^\alpha g|||_{\cB^0_0(\RR^6)}
\\&
 \leq \|g_2\|_{H^N_\ell(\RR^6)} ||(a,b,c)||_{H^N(\RR^3_x)}\,\,
 ||| \p^\alpha g|||_{\cB^0_\ell(\RR^6)},
 \\
 |G_1^{(22)}|\le&||g_2||_{H^N_\ell(\RR^6)}\,\,
||| g_2|||_{\cB^{N}_\ell(\RR^6)}\,\,
||| W^\ell\p^\alpha g|||_{\cB^0_0(\RR^6)}
\\&
 \leq \|g_2\|_{H^N_\ell(\RR^6)}
 ||| g_2|||_{\cB^{N}_\ell(\RR^6)}||| \p^\alpha g|||_{\cB^0_\ell(\RR^6)}.
 \end{align*}
Noticing that for
 $1\le|\alpha|\le N$,
\begin{align*}
|||  \p^\alpha g|||_{\cB^{0}_\ell(\RR^6)}^2&\le
\int_{\RR^3_x} |||W^\ell
\p^\alpha g_1(x,\, \cdot\,)|||^2dx
+\int_{\RR^3_x} |||W^\ell
\p^\alpha g_2(x,\, \cdot\,)|||^2dx
\\&\le \|\nabla_x(a,b,c)\|^2_{H^{N-1}(\RR^3_x)}+
|||g_2|||_{\cB^N_\ell(\RR^6)},
\end{align*}
we  now conclude the proof of the lemma.
\end{proof}

We shall evaluate $G_2$. In view of Proposition \ref{prop3.2.4},
\begin{align*}
|G_2|&\le
\left|\Big([\cL_1,W^\ell \,] \p^\alpha g,W^\ell
\,\p^\alpha
g)_{L^2(\RR^6)}\right|
\\&\hspace*{0.5cm}+
\left|\Big(W^\ell \,\cL_2(\p^\alpha g),\,\,
W^\ell\p^\alpha  g\Big)_{L^2(\RR^6)}\right|+
\left|\Big( \cL_2(W^\ell \,\p^\alpha g),\,\,
W^\ell\p^\alpha  g\Big)_{L^2(\RR^6)}\right|
\\&\le
 ||\p^\alpha
g||_{L^2_\ell(\RR^6)}|||\p^\alpha
g|||_{\cB^0_\ell(\RR^6)}
\\&\le (\|\nabla_x(a,b,c)\|_{H^{N-1}(\RR^3_x)}+\|\p^\alpha g_2\|_{L^2_\ell(\RR^6)})|||\p^\alpha
g|||_{\cB^0_\ell(\RR^6)}
\\&\le
\|\nabla_x(a,b,c)\|_{H^{N-1}(\RR^3_x)}^2+
 {C_\eta||\p^\alpha
g_2||_{L^2(\RR^6)}^2}+\eta|||\p^\alpha
g|||_{\cB^0_\ell(\RR^6)}^2, \ \ (\eta>0).
\end{align*}

Thus, we have established
\begin{lemm}
Let $ 1\le|\alpha|\le N, N\ge 3$. Then,
\begin{align}\label{microenergy2}
 \frac{d}{dt}\|&\p^\alpha g\|^2_{L^2_\ell(\RR^6)}
+||| \p^\alpha g |||^2_{\cB^0_\ell(\RR^6)}
\\ \le&\notag
\mathcal{E}^{1/2}\mathcal{D}+||\p^\alpha
g_2||_{L^2(\RR^6)}^2
+\|\nabla_x (a,b,c)\|_{H^{N-1}(\RR^3_x)}^2.
\end{align}
\end{lemm}

\noindent
{\bf Step 3 .} We need also to estimate $W^\ell g_2$.
Apply $W^\ell (\iI-\pP)$ to the equation in \eqref{cauchy-problem} to have
\begin{align*}%\label{equation3}
\p_t(W^\ell g_2)&+ v\cdot\nabla_x (W^\ell g_2)
+ \cL(W^\ell g_2)
\\&
=W^\ell\Gamma(g,\,
g)+W^\ell[v\cdot\nabla_x,\, \pP]g+[\cL,\, W^\ell]g_2. \notag
\end{align*}
And then take the inner product with
$W^\ell g_2$ to get
\[
\frac{d}{dt}\|W^\ell g_2\|^2_{L^2(\RR^6)}+D_3\le H,
\]
where
\begin{align*}
 D_3&=\int_{\RR^3}||| (\iI-\pP)W^\ell g_2 |||^2dx
\\&\ge \frac 12|||  g_2 |||^2_{\cB^{0}_\ell(\RR^6)}-
C\|g_2\|_{L^2(\RR^3)},
\end{align*}
while
\begin{align*}
H=&H_1+H_2+H_3,
\\&H_1=(W^\ell\Gamma(g,g), W^\ell g_2)_{L^2(\RR^6)},
\\&H_2=(W^\ell[v\cdot\nabla_x,\, \pP]g,  W^\ell g_2)_{L^2(\RR^6)},
\\&H_3=([\cL,\, W^\ell]g_2, W^\ell g_2)_{L^2(\RR^6)}.
\end{align*}
\begin{align*}
H_1&=\sum_{i,j=2}(W^\ell \Gamma(g_i,\,g_j), W^\ell g_2)_{L^2(\RR^6)}
=\sum_{i,j=2}H_1^{(ij)}.
\end{align*}
Proceeding as in \eqref{G11},
\begin{align*}%\label{H11}
|H^{(11)}|& \le
\|W^{2\ell}\Gamma(g_1,\,g_1)\|_{L^2(\RR^6)}    \| g_2\|_{L^2(\RR^6)}
\\&%\notag
\le \|\Big(\|W^{2\ell} \Gamma(g_1,\,g_1)\|_{L^2(\RR^3_v)}\Big)\|_{L^2(\RR^3_x)}
\| g_2\|_{L^2(\RR^6)}
\\& %\notag
\sim \|(a,b,c)^2\|_{L^2(\RR^3_x)}\|W^{2\ell}
\Gamma(\psi_j,\psi_k)\|_{L^2(\RR^3_v)}
\| g_2\|_{L^2(\RR^6)},
\\ %\notag
\|W^{2\ell}\Gamma&(\psi_j,\psi_k)\|_{L^2(\RR^3_v)}^2
=\int\Big(W^{2\ell}\int\int b(\cos\theta)\mu_*^{1/2}
\{ (\psi_j)_*' \psi_k'-(\psi_j)_*\psi_k\}dv_*d\sigma\Big)^2dv %\notag
\\&\qquad %\notag
=\int\mu W^{4\ell}\Big(\int\int b(\cos\theta)\mu_*
\{ (p_j)_*' p_k'-(p_j)_*p_k\}dv_*d\sigma\Big)^2dv<\infty.
\end{align*}
Then we have, by using Lemma \ref{abc2},
\begin{align*}\label{H111}
|H^{(11)}|\le \|&\nabla_x(a,b,c)\|_{L^2(\RR^3_x)}\|(a,b,c)\|_{H^{1}(\RR^3_x)}\| g_2\|_{L^2(\RR^6)}.
\end{align*}

On the other hand, we have
\begin{align*}
||| g_1|||_{\cB^{0}_\ell(\RR^6)}^2&=
\int_{\RR^3_x} |||W^\ell
 g_1(x,\, \cdot\,)|||^2dx
\le \|(a,b,c)\|^2_{L^2(\RR^3_x)}.
\end{align*}
This fact and  Proposition \ref{prop3.2.3} yield
 \begin{align*}
|H_1^{(12)}|\le &||g_1||_{H^N_\ell(\RR^6)}\,\,
||| g_2|||_{\cB^{N}_\ell(\RR^6)}\,\, |||W^\ell g_2|||_{\cB^0_0(\RR^6)}
\\&
 \leq  ||(a,b,c)||_{H^N(\RR^3_x)}\,\,
 ||| g_2|||_{\cB^{N}_\ell(\RR^6)}^2,
\\
|H_1^{(21)}|\le&||g_2||_{H^N_\ell(\RR^6)}\,\,
||| g_1|||_{\cB^{N}_\ell(\RR^6)}\,\,|||W^\ell g_2|||_{\cB^0_0(\RR^6)}
\\&
 \leq \|g_2\|_{H^N_\ell(\RR^6)} ||(a,b,c)||_{H^N(\RR^3_x)}\,\,
 ||| g_2|||_{\cB^0_\ell(\RR^6)},
 \\
 |H_1^{(22)}|\le&||g_2||_{H^N_\ell(\RR^6)}\,\,
||| g_2|||_{\cB^{N}_\ell(\RR^6)}\,\,
||| W^\ell g_2|||_{\cB^0_0(\RR^6)}
\\&
 \leq \|g_2\|_{H^N_\ell(\RR^6)}
 ||| g_2|||_{\cB^N_\ell(\RR^6)}^2.%||| \p^\alpha g|||_{\cB^0_\ell(\RR^6)}.
 \end{align*}
And $H_2$ is evaluated as follows
\begin{align*}
|H_2|&\le|(W^{2\ell}[v\cdot\nabla_x,\, \pP]g,   g_2)_{L^2(\RR^6)}|\le
\|\nabla_xg\|_{L^2(\RR^6)}\|g_2\|_{L^2(\RR^6)}
\\&\le
\|\nabla_x(a,b,c)\|_{L^2(\RR^3)}^2+\|g_2\|_{H^{1}(\RR^3_{x};L^2(\RR^3_v))}^2.
\end{align*}

Finally, in view of Proposition \ref{prop3.2.4} ,
\begin{align*}
|H_3|&\le
\left|\Big([\cL_1,W^\ell \,] g_2,W^\ell
g_2)_{L^2(\RR^6)}\right|
\\&\hspace*{0.5cm}+
\left|\Big(W^\ell \,\cL_2( g_2),\,\,
W^\ell g_2\Big)_{L^2(\RR^6)}\right|+
\left|\Big( \cL_2(W^\ell  g_2),\,\,
W^\ell  g_2\Big)_{L^2(\RR^6)}\right|
\\&\le
 \|
g_2\|_{L^2_\ell(\RR^6)}|||
g_2|||_{\cB^0_\ell(\RR^6)}\,.
\end{align*}
Since it holds by interpolation inequality that
\begin{equation}\label{norm}
\|g_2\|_{L^2_\ell(\RR^6)}\le \eta
\|g_2\|_{L^2_{\ell+s} (\RR^6)}+C_\eta\|g_2\|_{L^2(\RR^6)}
\le \eta |||g_2|||_{\cB^0_\ell(\RR^6)}+C_\eta\|g_2\|_{L^2(\RR^6)},
\end{equation}
for any small enough $\eta>0$, we have established
\begin{lemm}
\begin{align}\label{microenergy3}
 \frac{d}{dt}\|&g_2\|^2_{L^2_\ell(\RR^6)}
+||| g_2 |||^2_{\cB^0_\ell(\RR^6)}
\\ \le&\notag
\mathcal{E}^{1/2}\mathcal{D}+||
g_2||_{L^2(\RR^6)}^2
+\|\nabla_x (a,b,c)\|_{L^2(\RR^3_x)}^2.
 \end{align}
\end{lemm}
\noindent
{\bf Step 4 .} Let
\[
\p^\beta=\p^\beta_{x,v}=\p^\alpha \p^\gamma,
\quad \p^\alpha=\p_x^\alpha, \quad
\p^\gamma=\p^\gamma_v,\qquad |\beta|=|\alpha|+|\gamma|\le N, \gamma \ne 0,  \ N\ge 3,
\]
and apply $W^\ell\p^\beta(\iI-\pP)$ to \eqref{g1-cauchy-problem} to have
\begin{align*}%\label{equation3+}
 \p_t (W^\ell\p^\beta &g_2)+ v\cdot\nabla_x
(W^\ell\p^\beta g_2) + \cL (W^\ell\p^\beta g_2)
\\=&\notag
W^\ell\p^\beta \Gamma(g,\,g)+
[v\cdot\nabla_x, W^\ell\p^\beta]g_2
\\
&- W^\ell \p^\beta [\pP,\,\, v\cdot \nabla]g \notag\\
&+\notag
[\mathcal{L},\, W^\ell\p^\beta]g_2
+
W^\ell\p^\beta (\p_t+v\cdot\nabla_x )g_1.
\end{align*}
 And then take the $L^2(\RR^6_{x, v})$ inner product with $W^\ell\p^\beta g_2$ to
 get
\begin{equation*}\label{energy1+}
\frac 12 \frac{d}{dt}\|\p^\beta g_2\|^2_{L^2_\ell(\RR^6)}+D_4\le K.
\end{equation*}
Here $D_4$ is a dissipation rate given by
\begin{align*}
 D_4&=\int_{\RR^3}||| (\iI-\pP)W^\ell\p^\beta g_2 |||^2dx
\\&\ge \frac 12||| \p^\beta g_2 |||^2_{\cB^{0}_\ell(\RR^6)}-
C|| \p^\alpha g_2 ||^2_{L^2(\RR^6)}.
\end{align*}
where we used, with $\psi\in \mathcal{N}$ and
$\tilde{\psi}=(-1)^{|\gamma|}\p^\gamma(W^\ell\psi)$,
\[
|||\pP W^\ell\p^\beta g_2|||^2 =|||\p^\alpha (\psi, W^\ell\p^\gamma g_2)_{L^2(\RR^3_v)}\psi|||^2
=\|(\tilde{\psi},\p^\alpha g_2)_{L^2_{v}}\|_{L^2(\RR^3_x)}^2|||\psi|||^2
\le ||\p^\alpha g_2||_{L^2(\RR^6_{x, v})}^2.
\]
On the other hand, $K$ is given by
\begin{align*}
K=& (W^\ell\p^\beta \Gamma(g,\,g), W^\ell\p^\beta g_2)_{L^2(\RR^6)}
\\&+
([v\cdot\nabla_x,\, W^\ell\p^\beta]g_2, W^\ell\p^\beta g_2)_{L^2(\RR^6)}
- (W^\ell \p^\beta [\pP,\, v\cdot \nabla]g, W^\ell\p^\beta g_2)_{L^2(\RR^6)}
\\&+
([\mathcal{L},\, W^\ell\p^\beta]g_2, W^\ell\p^\beta g_2)_{L^2{(\RR^6)}}
\\&+
(W^\ell\p^\beta (\p_t+v\cdot\nabla_x )g_1, W^\ell\p^\beta g_2)_{L^2(\RR^6)}
\\=&
K_1+K_2+K_3+K_4+ K_5.
\end{align*}
\begin{lemm}\label{K1}
Let $N\ge 3$. Then $
|K_1|\le  \mathcal{E}^{1/2} \mathcal{D}\,.
$
\end{lemm}
\begin{proof} First, write
\begin{align*}
K_1&=\sum_{i,j=2}(W^\ell\p^\beta \Gamma(g_i,\,g_j), W^\ell\p^\beta g_2)_{L^2(\RR^6)}
=\sum_{i,j=2}K_1^{(ij)}.
\end{align*}
In view of Lemma \ref{abc2},
\begin{align*}
|K_1^{(11)}|&=\big |(W^\ell\p^\beta\Gamma(g_1,\,g_1), W^\ell\p^\beta g_2)_{L^2(\RR^6)}\big|
 \\&\le B_2\|\p^\alpha   (a,b,c)^2\|_{L^2(\RR^3_x)}\|W^\ell\p^\beta g_2\|_{L^2(\RR^6)}
 \\& \le  \|\nabla_x (a,b,c)\|_{H^{N-1}_x}\|(a,b,c)\|_{H^{N-1}_x}
 \|W^\ell\p^\beta g_2\|_{L^2(\RR^6)},\quad
 \end{align*}
 where
 \begin{align*}
 B_2^2 \sim &\|W^\ell\p^{\gamma}_v\Gamma(\psi_j,\psi_k)\|_{L^2(\RR^3_v)}^2
\\
&=\int\big(W^\ell\int\int b(\cos\theta)(\p^{\gamma_1}_v\mu)_*^{1/2}
\\&\qquad \times \big\{ (\p^{\gamma_2}_v\psi_j)_*' (\p^{\gamma_3}_v\psi_k)'
-(\p^{\gamma_2}_v\psi_j)_* (\p^{\gamma_3}_v\psi_k)\big\}dv_*d\sigma\big)^2dv
\\&
 =\int\mu W^{2\ell}\Big(\int\int b(\cos\theta)\mu_*q_*
\big\{ (q_j)_*' q_k'-(q_j)_*q_k\big\}dv_*d\sigma\Big)^2dv<\infty.
\end{align*}
Here, $q,q_j,q_k$ are polynomials of $v$.

On the other hand,
we have
\[
||| g_1|||_{\cB^{N}_\ell(\RR^6)}^2=
 \sum_{|\beta| \le N} \int_{\RR^3_x} |||W^\ell\,
\p^\beta g_1(x,\, \cdot\,)|||^2dx
\le \|(a,b,c)\|^2_{H^N(\RR^3_x)}.
\]
This point and  Proposition \ref{prop3.2.3}  yield
 \begin{align*}
|K_1^{(12)}|\le &||g_1||_{H^N_\ell(\RR^6)}\,\,
||| g_2|||_{\cB^{N}_\ell(\RR^6)}\,\, ||| W^\ell\p^\beta g_2|||_{\cB^0_0(\RR^6)}
\\&
 \leq  ||(a,b,c)||_{H^N(\RR^3_x)}\,\,
 ||| g_2|||_{\cB^{N}_\ell(\RR^6)}^2,
\\
|K_1^{(21)}|=&||g_2||_{H^N_\ell(\RR^6)}\,\,
||| g_1|||_{\cB^{N}_\ell(\RR^6)}\,\, ||| W^\ell\p^\beta g_2|||_{\cB^0_0(\RR^6)}
\\&
 \leq \|g_2\|_{H^N_\ell(\RR^6)} ||(a,b,c)||_{H^N(\RR^3_x)}\,\,
 ||| g_2|||_{\cB^{N}_\ell(\RR^6)},
 \\
 |K_1^{(22)}|=&||g_2||_{H^N_\ell(\RR^6)}\,\,
||| g_2|||_{\cB^{N}_\ell(\RR^6)}\,\, ||| W^\ell\p^\beta g_2|||_{\cB^0_0(\RR^6)}
\\&
 \leq \|g_2\|_{H^N_\ell(\RR^6)}
 ||| g_2|||_{\cB^{N}_\ell(\RR^6)}^2.
 \end{align*}
 Now the proof of the lemma is completed.
 \end{proof}

$K_2,K_4, K_5$ are estimated as follows. We have, for
$|\beta|=|\alpha|+|\gamma|\le N, \gamma\not=0$,
 \begin{align*}
 |K_2|=&\big|([v\cdot\nabla_x, W^\ell\p^\beta]g_2, W^\ell\p^\beta g_2)_{L^2(\RR^6)}
 \big|
 \\&
 \le \|W^\ell\p_x^{\alpha+1}\p_v^{\gamma-1}g_2\|_{L^2(\RR^6)}
\|W^\ell\p^\beta g_2\|_{L^2(\RR^6)}
\\
&\le  C_\eta\|W^\ell\p_x^{\alpha+1}\p_v^{\gamma-1}g_2\|_{L^2(\RR^6)}^2+\eta
\|W^\ell\p^\beta g_2\|_{L^2(\RR^6)}^2.
\end{align*}

Note that
\begin{align*}
|K_3|=&|(W^\ell \p^\beta [\pP,\, v\cdot \nabla]g,\,
W^\ell\p^\beta g_2)_{L^2(\RR^6)},W^\ell\p^\beta g_2)_{L^2(\RR^6)}|
\\&
\le |(\p^\gamma \Big\{W^{2\ell} \p^\beta [\pP,\, v\cdot \nabla]g\Big\},
\p^\alpha g_2)_{L^2(\RR^6)}|
\\&
\le\Big(
\|\nabla_x\p^\alpha (a,b,c)\|_{L^2(\RR^3_x)}+\|\nabla_x\p^\alpha g_2\|_{L^2(\RR^6_{x, v})}\Big)
\|\p^\alpha  g_2\|_{L^2(\RR^6_{x, v})}
\\&\le
\|\nabla_x (a,b,c)\|_{H^{N-1}(\RR^3_x)}^2
+\|g_2\|_{H^N(\RR^3_x, L^2(\RR^3_v))}^2.
\end{align*}

In view of Proposition \ref{prop3.2.4},
\begin{align*}
|K_4|=&\big |([\mathcal{L},\, W^\ell\p^\beta]g_2, W^\ell\p^\beta g_2)_{L^2(\RR^6)}\big |
\\&\le
\left|\Big([\cL_1,\, W^\ell \,\partial^\beta_{x, v}] g_2,W^\ell
\,\partial^\beta
g_2)_{L^2(\RR^6)}\right|
\\&\quad+
\left|\Big(W^\ell \,\partial^\beta_{x, v} \cL_2(g_2),\,\,
W^\ell\p^\beta g_2\Big)_{L^2(\RR^6)}\right|+
\left|\Big( \cL_2(W^\ell \,\partial^\beta_{x, v}g_2),\,\,
W^\ell\p^\beta g_2\Big)_{L^2(\RR^6)}\right|
\\&\le
\Big( ||
g_2||_{H^{|\beta|}_\ell(\RR^6)}+ |||
g_2|||_{\cB^{|\beta|-1}_\ell(\RR^6)}\, \Big)\,\, |||
W^\ell\p^\beta g_2|||_{\cB^0_0(\RR^6)}
\end{align*}
Hence
\begin{align*}
|K_4|&\le C_\eta \Big(||g_2||_{H^{|\beta|}_\ell(\RR^6)}^2+
|||g_2|||_{\cB^{|\beta|-1}_\ell(\RR^6)}^2\Big)+
\eta|||W^\ell\p^\beta g_2|||^2_{\cB^0_0(\RR^6)}
\end{align*}
Finally, recalling \eqref{ptabc},
\begin{align*}
|K_5|=&\big |(W^\ell\p^\beta (\p_t+v\cdot\nabla_x )g_1,W^\ell\p^\beta g_2)_{L^2(\RR^6)}\big|
\\&
\le |(\p^\gamma \Big\{W^{2\ell}\p^\beta (\p_t+v\cdot\nabla_x )g_1\Big\},
\p^\alpha g_2)_{L^2(\RR^6)}|
\\&
\le\|(\p^\alpha(\p_t  +\nabla_x)(a,b,c)\|_{L^2(\RR^3_x)}
\|\p^\alpha g_2\|_{L^2(\RR^6)},\quad (|\alpha|\le N-1,|\gamma|\ge 1)
\\&
\le\Big(
\|\nabla_x\p^\alpha (a,b,c)\|_{L^2(\RR^3_x)}+\|\nabla_x\p^\alpha g_2\|_{L^2(\RR^6_{x, v})}\Big)
\|\p^\alpha  g_2\|_{L^2(\RR^6_{x, v})}
\\&\le
\|\nabla_x (a,b,c)\|_{H^{N-1}_x}^2
+\|g_2\|_{H^N(\RR^3_x, L^2(\RR^3_v))}^2.
\end{align*}
Now using \eqref{norm} we conclude the
\begin{lemm}
Let $|\beta|=|\alpha+\gamma|\le N, |\alpha|\le N-1, |\gamma|\ge 1, N\ge 3$. Then,
\begin{align}\label{microenergy4}
 \frac{d}{dt}\|& \p^\beta g_2\|^2_{L^2_\ell(\RR^6)}+||| \p^\beta g_2 |||^2_{\cB^{0}_\ell(\RR^6)}
\\ \le&\notag
\mathcal{E}^{1/2}\mathcal{D}+||g_2||_{H^{|\beta|}_\ell(\RR^6)}^2
\\&+\notag
||| g_2|||_{\cB^{|\alpha|+|\gamma|-1}_\ell(\RR^6)}^2+\|\nabla_x (a,b,c)\|_{H^{N-1}(\RR^3_x)}^2
+\|g_2\|_{H^N(\RR^3_x, L^2(\RR^3_v))}^2.
 \end{align}
\end{lemm}

\subsection{A Priori Estimate}
\setcounter{equation}{0}
We take the linear combination
\begin{center}
$\displaystyle
\sum_{|\alpha|\le N-1}C^{(1)}_\alpha$\eqref{pabc}$_\alpha$
+$\displaystyle
\sum_{|\alpha|\le N}C^{(2)}_\alpha$\eqref{microenergy11}$_\alpha$
+$\displaystyle
\sum_{1\le|\alpha|\le N}C^{(3)}_\alpha$\eqref{microenergy2}$_\alpha$
+$C^{(4)}$\eqref{microenergy3}+$\displaystyle
\sum_{|\beta|=|\alpha+\gamma|\le N, |\alpha|\le N-1,
|\gamma|\ge 1}C^{(3)}_{\alpha, \gamma}$\eqref{microenergy4}$_{\alpha, \gamma}.$
\end{center}
With a suitable choice of the coefficients
$C^{(1)}_\alpha, C^{(2)}_\alpha,
C^{(3)}_\alpha, C^{(4)}, C^{(5)}_{\alpha, \gamma}$, we get
\begin{align}\label{energyinq1}
\frac{d}{dt}\tilde{\mathcal{E}}+\tilde{\mathcal{D}}\le \mathcal{H},
\end{align}
where
\begin{align*}
\tilde{\mathcal{E}}=&
-\sum_{|\alpha|\le N-1}C^{(1)}_\alpha
\Big[(\p^\alpha r,\nabla_x\p^\alpha(a,-b,c))_{L^2(\RR^3_x)}
+(\p^\alpha b, \nabla_x \p^\alpha a)_{L^2(\RR^3_x)}\Big]
\\&+\notag
\sum_{|\alpha|\le N}C^{(2)}_\alpha\|\p^\alpha_x g\|^2_{L^2(\RR^6)}
+\sum_{1\le|\alpha|\le N}C^{(3)}_\alpha\|\p^\alpha_x g\|^2_{L^2_\ell(\RR^6)}
\\&+\notag
C^{(4)}\|g_2\|^2_{L^2_\ell(\RR^6)}
+\sum_{|\beta|=|\alpha+\gamma|\le N, |\alpha|\le N-1,
|\gamma|\ge 1}C^{(5)}_{\alpha, \gamma}\|W^\ell\p^\beta_{x,v} g_2\|^2_{L^2(\RR^6)},
\\
\tilde{\mathcal{D}}=&\sum_{|\alpha|\le N-1}C^{(1)}_\alpha\|\nabla_x \p^\alpha_x (a,b,c)\|_{L^2(\RR^3_x)}^2
+\sum_{|\alpha|\le N}C^{(2)}_\alpha|||\p^\alpha_x  g_2 |||_{\cB^{0}_0(\RR^6)}^2
\\&+\notag
\sum_{1\le |\alpha|\le N}C^{(3)}_\alpha|||\p^\alpha_x  g_2 |||_{\cB^{0}_\ell(\RR^6)}^2
+C^{(4)}|||g_2|||^2_{\cB^{0}_\ell(\RR^6)}
\\&+\sum_{|\beta|=|\alpha+\gamma|\le N, |\alpha|\le N-1,
|\gamma|\ge 1}C^{(3)}_{\alpha, \gamma}\notag
||| \p^\beta_{x,v} g_2 |||^2_{\cB^{0}_\ell(\RR^6)},
\\    \mathcal{H}=&\mathcal{D}_1\mathcal{E}_1
+\mathcal{D}_2\mathcal{E}_2+\mathcal{D}\mathcal{E}.
\end{align*}
Clearly, it holds that
\[
\mathcal{E}\sim \tilde{\mathcal{E}},
\quad
\mathcal{D}\sim \tilde{\mathcal{D}},
\]
and
\[
 \mathcal{H}
\le \mathcal{D}\mathcal{E}.
\]

Now \eqref{energyinq1} gives
\begin{equation*}\label{ineq1}
\mathcal{E}(t)+\Big[1-C\sup_{0\le \tau\le t}\mathcal{E}(\tau)\Big]
\int_0^t\mathcal{D}(\tau)d\tau\le C\mathcal{E}(0),
\end{equation*}
which leads to the closure of \`a priori estimate
and then completes the proof of Theorem \ref{aprioriestimate}.

Now, the proof of Theorem \ref{theo1}
can be completed by the usual continuation argument based on Theorem \ref{theo4.2.1} and Theorem \ref{aprioriestimate}.

\bigskip
\noindent
{\bf Acknowledgements :}
The last author's research was supported by the General Research
Fund of Hong Kong, CityU\#103109. Finally the authors would like to
thank the financial support of City University of Hong Kong, Kyoto
University and Wuhan University during each of their stays, mainly starting {}from 2006. These supports have enable the final conclusion through our previous papers.

\vskip0.5cm

\


\begin{thebibliography}{99}
\bibitem{alex-lin} R. Alexandre, Remarks on 3D Boltzmann linear operator without cutoff. {\it Transp. Th. and Stat. Phys.}, {\bf 28-5} (1999) 433--473.

\bibitem{alex-2} R. Alexandre, Around 3D Boltzmann operator without cutoff. A New
formulation. {\it Math. Modelling and Num. Analysis}, {\bf 343} (2000) 575--590.


\bibitem{alex-two-maxwellian}
R. Alexandre,  Some solutions of the Boltzmann equation without
angular cutoff, {\it J. Stat. Physics}, {\bf 104} (2001) 327--358.


\bibitem{alex-sing}
R. Alexandre, Integral kernel estimates
for a linear singular operator linked with Boltzmann equation. Part
I: Small singularities $0 < \nu < 1$. {\it Indiana Univ. Math. J}. \textbf{55} (2006), no. 6, 1975--2021.

\bibitem{alex-review} R. Alexandre, A review of Boltzmann equation with singular kernels. {\it Kinetic and related models}, {\bf
2-4} (2009) 551--646.

\bibitem{al-1} R. Alexandre, L. Desvillettes, C. Villani and  B. Wennberg,
Entropy  dissipation and long-range interactions, {\it Arch.
Rational Mech. Anal.}, {\bf 152} (2000) 327-355.

\bibitem{al-saf} R. Alexandre, M. ElSafadi, Littlewood-Paley decomposition and regularity issues
in Boltzmann homogeneous equations.I.
Non cutoff and Maxwell cases. {\it Math. Models Methods Appl. Sci}. \textbf{15} (2005), 907--920.

\bibitem{amuxy-nonlinear} R.Alexandre, Y.Morimoto, S.Ukai, C.-J.Xu and T.Yang,
Uncertainty principle and regularity for Boltzmann type equations,
{\it C. R. Acad. Sci. Paris, Ser. I}, {\bf 345} (2007) 673-677.

\bibitem{amuxy-nonlinear-b} R.Alexandre, Y.Morimoto, S.Ukai, C.-J.Xu and T.Yang,
Uncertainty principle and kinetic equations, {\it J. Funct. Anal.},
{\bf 255} (2008) 2013-2066.

\bibitem{amuxy2} R.Alexandre, Y.Morimoto, S.Ukai, C.-J. Xu and T. Yang,
 Regularity of solutions for the Boltzmann equation without angular cutoff.
 {\it C. R. Math. Acad. Sci. Paris, Ser. I}, {\bf 347} (2009), no. 13-14, 747--752.

\bibitem{amuxy2b} R.Alexandre, Y.Morimoto, S. Ukai, C.-J.Xu and T.Yang,
Local existence for non-cutoff Boltzmann equation, {\it C. R. Math. Acad. Sci. Paris, Ser. I}, {\bf 347} (2009), no. 21-22, 1237--1242.

\bibitem{amuxy-nonlinear-3} R.Alexandre, Y.Morimoto, S. Ukai, C.-J.Xu
and T.Yang, Regularizing effect and local existence
for non-cutoff Boltzmann equation, {\it Arch. Rational Mech. Anal.},
{\bf  198, Issue 1 }(2010), 39-123. 

\bibitem{amuxy4-1} R.Alexandre, Y.Morimoto, S. Ukai, C.-J.Xu
and T.Yang, Boltzmann equation without angular cutoff in the whole
space : I, Global existence for soft
potential, Preprint HAL,
http://hal.archives-ouvertes.fr/hal-00477662/fr/

%\bibitem{amuxy4-2} R.Alexandre, Y.Morimoto, S. Ukai, C.-J.Xu
%and T.Yang, Boltzmann equation without angular cutoff in the whole
%space: II. Global existence for soft potential, Preprint HAL,
%http://hal.archives-ouvertes.fr/hal-00496950/fr/


\bibitem{amuxy4-3} R.Alexandre, Y.Morimoto, S. Ukai, C.-J.Xu
and T.Yang, Boltzmann equation without angular cutoff in the whole
space : II, Global existence for hard potential,Preprint HAL,
http://hal.archives-ouvertes.fr/hal-00510633/fr/


\bibitem{amuxy4-4} R.Alexandre, Y.Morimoto, S. Ukai, C.-J.Xu
and T.Yang, Qualitative properties of solutions to the Boltzmann equation without angular cutof,
preprint 2010.



\bibitem{al-3}
R. Alexandre and  C. Villani, On the Boltzmann equation for
long-range interaction, {\it Comm. Pure  Appl.
Math.} {\bf 55} (2002) 30--70.

\bibitem{Arkeryd}
L. Arkeryd, Intermolecular forces of infinite range and the Boltzmann equation,
{\it Arch. Rational Mech. Anal.}, {\bf 77}(1981) No.1, 11-21.

\bibitem{bobylev}
A. Bobylev, The theory of nonlinear, spatially uniform Boltzmann
equation for Maxwell molecules, {\it Sov. Sci. Rev. C. Math.
Phys.}, {\bf 7} (1988) 111--233.

\bibitem{BD}
F. Bouchut and L. Desvillettes, A proof of smoothing properties of
the positive part of Boltzmann's kernel. {\it Rev. Mat. Iberoamericana},
{\bf 14} (1998) 47--61.

\bibitem{bouchut-1} F. Bouchut, F. Golse and M. Pulvirenti,
{\it Kinetic equations and asymptotic theory}, Series in Appl. Math.,
Gauthiers-Villars, 2000.

\bibitem{desv}
L. Boudin and L. Desvillettes, On the singularities of the global small
solutions of the full Boltzmann equation, {\it Monatschefte f\"{u}r
Mathematik}, {\bf 131} (2000), 91--108.

\bibitem{Ce88} \rm C. Cercignani, {\it The Boltzmann equation and its applications},
Applied mathematical sciences, {\bf 67}, Springer-Verlag, 1988.

\bibitem{Cercignani-Illner-Pulvirenti}
C. Cercignani, R. Illner and  M. Pulvirenti, {\it The Mathematical
theory of Dilute gases}, Applied mathematical sciences, {\bf 106}.
Springer-Verlag, New York, 1994. viii+347 pp.

\bibitem{chen-li-xu1} H. Chen, W.-X. Li and C.-J. Xu, The Gevrey hypoellipticity for a class
of kinetic equations {\it To appear in Comm. PDE}.

\bibitem{desv-landau} Y. Chen, L. Desvillettes and  L. He,: Smoothing effects
for classical solutions of the full Landau equation. {\it Arch. Rat. Mech. Anal.} {\bf 193}, 21-55 (2009)


\bibitem{D95} {L. Desvillettes,} About the regularization properties
of the non cut-off Kac equation, {\it Comm. Math. Phys.}, {\bf 168} (1995) 417--440.

\bibitem{desv-ajout1} L. Desvillettes, Regularization for the non Cutoff 2D Radially Symmetric Boltzmann Equation with a Velocity Dependant Cross Section, {\it Transport Theory and Statistical Physics}, vol. 25, n. 3-5 (1996), 383--394

\bibitem{desv-ajout2} L. Desvillettes, Regularization Properties of the 2-Dimensional Non Radially Symmetric Non Cutoff Spatially Homogeneous Boltzmann Equation for Maxwellian Molecules, {\it Transport Theory and Statistical Physics}, vol. 26, n. 3 (1997) 341--357.

\bibitem{desv-F} L. Desvillettes, About the use of the Fourier transform
for the Boltzmann equation. {\it Summer School on "Methods and
Models of Kinetic Theory" (M\& MKT 2002). Riv. Mat. Univ. Parma}
{\bf (7)} 2 (2003), 1--99.

\bibitem{Des-Fu-Ter}
L. Desvillettes, G. Furioli and E. Terraneo, Propagation of Gevrey
regularity for solutions of Boltzmann equation for Maxwellian
molecules, {\it Trans. Amer. Math. Soc.}, {\bf 361}(2009) 1731-1747.

\bibitem{desv-vil} L. Desvillettes and C. Villani, On the spatially homogeneous Landau
equation for hard potentials. Part I: existence, uniqueness and
smoothness. {\it Comm. Partial Differential Equations}, {\bf 25-1-2}
(2000) 179--259

\bibitem{desv-wen1} L. Desvillettes and B. Wennberg, Smoothness of the solution
of the spatially homogeneous Boltzmann equation without cutoff.
{\it Comm. Partial Differential Equations}, {\bf 29-1-2} (2004) 133--155.


\bibitem{desv-golse}L. Desvillettes and F. Golse,  On a model Boltzmann equation without angular cutoff, {\it Differential Integral Equations}  {\bf 13 }  (2000),  no. 4-6, 567--594.

\bibitem{D-L} R. J. DiPerna and P. L. Lions, On the Cauchy problem for Boltzmann
equations: global existence and weak stability. {\it Ann. Math.},
{\bf 130} (1989), 321-366.



\bibitem{duan}
R.-J.Duan, Stability of the Boltzmann equation with potential forces on torus. Physica D,
{\bf 238}(2009). 1808-1820. 


\bibitem{dly} R.-J.Duan, M.-R. Li
and T. Yang, Propagation of singularities in the solutions to the
Boltzmann equation near equilibrium,  {\it Math. Models Methods Appl.
Sci.}, {\bf 18} (2008) 1093--1114.

\bibitem{duan-ukai-yang}
R.-J. Duan, S.Ukai, and T.Yang, A combination of 
energy method and spectral analysis for study of equations of gas motion.  
Front. Math. China  4  (2009),  no. 2, 253--282.



\bibitem{Fefferman}C. Fefferman,
The uncertainty principle, Bull. Amer. Math. Soc.,
{\bf 9} (1983), 129-206.


\bibitem{Grad}
H. Grad, {\it Asymptotic Theory of the Boltzmann Equation II}, Rarefied
Gas Dynamics, J. A. Laurmann, Ed. {\bf Vol. 1}, Academic Press,
New York, 1963, 26--59.

\bibitem{grad1} H. Grad, Asymptotic equivalence of the Navier-Stokes and
nonlinear Boltzmann equations, {\it Proc. Symp. Appl. Math.} vol.
17,  154-183, AMS, Providence, 1965, editor R. Finn.

\bibitem{gr-st} P.-T. Gressman and R.-M. Strain, Global strong solutions of the Boltzmann equation without cutoff. http://www.math.upenn.edu/~strain/preprints/0912.0888.pdf, 2009.

\bibitem{gr-st2} P.-T. Gressman and R.-M. Strain, Global classical solutions of the Boltzmann equation 
with long range potential. Proc. Nat. Acad. Sci. U.S.A, {\bf 107(13)}(2010), 5744-5749.

\bibitem{gr-st3} P.-T. Gressman and R.-M. Strain, Global classical solutions of the Boltzmann equation 
with long range interations and soft potentials. Preprint 2010


\bibitem{gr-st4} P.-T. Gressman and R.-M. Strain, Sharp anisotropic estimates for the Boltzmann collision operator and its entropy production, Preprint 2010



\bibitem{guo} Y. Guo, The Landau equation in a periodic box.
{\it Comm. Math. Phys.}, {\bf 231} (2002) 391--434.

\bibitem{guo-1} Y. Guo, The Boltzmann equation in the whole space. {\it Indiana Univ. Maths. J.}, {\bf 53-4} (2004), 1081--1094.






\bibitem{HMUY} Z. H. Huo, Y.Morimoto, S.Ukai and T.Yang, Regularity of
solutions for spatially homogeneous Boltzmann equation without
Angular cutoff. {\it Kinetic and Related Models}, {\bf 1} (2008)
453-489.

\bibitem{K-S} S. Kaniel and M. Shinbrot, The Boltzmann equation.
I. Uniqueness and local existence. {\it Comm. Math. Phys.}, {\bf 59}
(1978) 65-84.


\bibitem{Lions94} P. L. Lions, Compactness in Boltzmann's equation via Fourier
integral operators and applications, I, II, and III.
{\it J. Math. Kyoto Univ.}, {\bf 34} (1994), 391--427, 429--461, 539--584.

\bibitem{lions-landau} P. L. Lions, On Boltzmann and Landau equations, {\it Philos. Trans. Roy. Soc. London Ser. A}
    {\bf 346} (1994), no. 1679, 191--204.

\bibitem{Lions98} P. L. Lions, Regularity and compactness for Boltzmann collision operator
without angular cut-off. {\it C. R. Acad. Sci. Paris S\'eries I}, {\bf 326} (1998), 37--41.

\bibitem{liu-1} T.-P Liu and S.-H.Yu, Micro-macro decompositions and positivity of shock profiles.  Comm. Math. Phys. , {\bf 246}  (2004),  no. 1, 133--179.

\bibitem{liu-2}
T.-P. Liu, T.Yang and S.-H.Yu,
Energy method for Boltzmann equation.
{\it Phys. D }, {\bf 188}  (2004), 178-192.

\bibitem{liu-yang-yu-zhao}
T.-P. Liu, T.Yang, S.-H. Yu, and H.-J.Zhao, Nonlinear stability of rarefaction waves for the Boltzmann equation.  Arch. Ration. Mech. Anal.  181  (2006),

\bibitem{Lu} X. A. Lu,
Direct method for the regularity of the gain term in  the Boltzmann
equation. {\it J. Math. Anal. Appl.}, {\bf 228} (1998), 409--435.

\bibitem{morimo873}Y. Morimoto,
The uncertainty principle and hypoelliptic
operators,
Publ. RIMS Kyoto Univ., {\bf 23} (1987), 955-964.

\bibitem{morimo921}Y. Morimoto,
Estimates for degenerate Schr\"odinger operators
and hypoellipticity for infinitely degenerate elliptic operators,
J. Math. Kyoto Univ., {\bf 32} (1992), 333-372.
%
\bibitem{morimoto-mori}Y. Morimoto and T. Morioka,
The positivity of Schr\"odinger operators and the hypoellipticity of second order
degenerate elliptic operators, {\it Bull. Sc. Math.}  {\bf121} (1997) , 507-547.

\bibitem{MU} Y. Morimoto and S. Ukai,
Gevrey smoothing effect of solutions for  spatially
homogeneous nonlinear Boltzmann  equation without angular cutoff,
to appear on {\it J. Pseudo Differential Operator.}

\bibitem{MUXY-DCDS}
Y.Morimoto, S.Ukai, C.-J.Xu and T.Yang, Regularity of solutions to
the spatially homogeneous Boltzmann equation without Angular cutoff.
{\it Discrete and Continuous Dynamical Systems -
Series A} {\bf 24} 1 (2009), 187-212.

\bibitem{M-X}Y. Morimoto and C.-J. Xu, Hypoelliticity for a class
of kinetic equations, {\it J. Math. Kyoto Univ.} {\bf 47} (2007), 129--152.


\bibitem{mouhot} C. Mouhot, Explicit coercivity estimates for the linearized Boltzmann and Landau operators. { \it Comm. Part. Diff. Equations}, vol. 31, (2006) 1321-1348.

\bibitem{mouhot-strain} C. Mouhot and R.M. Strain, { Spectral gap and coercivity estimates
for linearized Boltzmann collision
operators without angular cutoff},  {\it J. Math. Pures Appl. } (9)  {\bf 87}  (2007),  no. 5, 515--535.

\bibitem{Pao}
Y. P. Pao, Boltzmann collision operator with inverse power
intermolecular potential, I, II. {\it Commun. Pure Appl. Math.},
{\bf 27} (1974), 407--428, 559--581.



\bibitem{strain} R.M.Strain, 
The Vlasov-Maxwell-Boltzmann system in the whole space. %(English summary)
Comm. Math. Phys. {\bf 268}(2006), no. 2, 543--567.

\bibitem{strain-guo} R. Strain and Y. Guo, Exponential decay for soft potentials near Maxwellian.  Arch. Ration. Mech. Anal.  187  (2008),  no. 2, 287--339. 

\bibitem{ukai-1a}S. Ukai, On the existence of global solutions of mixed problem for non-linear Boltzmann equation.  {\it Proc. Japan Acad. ,} {\bf  50  } (1974), 179--184.


\bibitem{ukai-1b}S. Ukai, Les solutions globales de l'equation de Boltzmann dans l'espace tout entier et dans le demi-espace,  C. R. Acad. Sci. Paris Ser. A-B  {\bf 282 } (1976), no. 6, Ai, A317--A320.


\bibitem{ukai}S. Ukai, Local solutions in Gevrey classes
to the nonlinear Boltzmann equation
without cutoff, {\it Japan J. Appl. Math.}, {\bf 1-1} (1984) 141--156.

\bibitem{ukai-2} S. Ukai,
Solutions of the Boltzmann equation,
\textit{Pattern and Waves -- Qualitative Analysis of Nonlinear Differential
Equations} (eds. M.Mimura and T.Nishida), Studies of Mathematics and Its
Applications {\bf 18}, pp37-96, Kinokuniya-North-Holland, Tokyo, 1986.

\bibitem{villani}C. Villani, On a new class of weak solutions to the spatially
homogeneous Boltzmann and Landau equations, {\it Arch. Rational
Mech. Anal.}, {\bf 143} (1998) 273--307.

\bibitem{Vi99}
C. Villani, Regularity estimates via entropy dissipation for the spatially
homogeneous Boltzmann equation, {\it Rev. Mat. Iberoamericana}, {\bf 15-2} (1999) 335--352.

\bibitem{villani2}C. Villani, {\it A review of mathematical
topics in collisional kinetic theory}.
Handbook of Fluid Mechanics. Ed. S. Friedlander, D.Serre, 2002.

\bibitem{h-yu}
H. Yu, Convergence rate for the Boltzmann and Landau equations with soft potentials, Proc. Royal Soc. Edingburgh, {\bf 139A}(2009), 393-416. 

\end{thebibliography}
\end{document}